\documentclass[12pt,twoside,reqno,draft]{extarticle}
\usepackage[latin9]{inputenc}
\usepackage[a4paper]{geometry}
\usepackage{color}
\usepackage{amsmath}
\usepackage{amsthm}
\usepackage{amssymb}

\makeatletter

\DeclareRobustCommand*{\lyxarrow}{%
\@ifstar
{\leavevmode\,$\triangleleft$\,\allowbreak}
{\leavevmode\,$\triangleright$\,\allowbreak}}

\numberwithin{equation}{section}
\numberwithin{figure}{section}
\theoremstyle{plain}
\newtheorem{thm}{\protect\theoremname}
  \theoremstyle{remark}
  \newtheorem{rem}[thm]{\protect\remarkname}
  \theoremstyle{plain}
  \newtheorem*{assumption*}{\protect\assumptionname}
  \theoremstyle{plain}
  \newtheorem{lem}[thm]{\protect\lemmaname}
  \theoremstyle{plain}
  \newtheorem{cor}[thm]{\protect\corollaryname}
  \theoremstyle{plain}
  \newtheorem{prop}[thm]{\protect\propositionname}

\allowdisplaybreaks

\usepackage{geometry}
\makeatother

  \providecommand{\assumptionname}{Assumption}
  \providecommand{\corollaryname}{Corollary}
  \providecommand{\lemmaname}{Lemma}
  \providecommand{\propositionname}{Proposition}
  \providecommand{\remarkname}{Remark}
\providecommand{\theoremname}{Theorem}

\begin{document}

\title{A high order time discretization of the solution of the non-linear
filtering problem}

\date{\today}

\author{D. Crisan\thanks{The work of D. C. was partially supported by the EPSRC Grant EP/H0005500/1.}
\thanks{Department of Mathematics, Imperial College London, Huxley's Building,180
Queen's Gate, London SW7 2AZ, United Kingdom, E-mail: dcrisan@imperial.ac.uk } and S. Ortiz-Latorre\thanks{The work of S. O-L. was supported by the BP-DGR 2009 grant and the
project \emph{Energy Markets: Modeling, Optimization and Simulation
(EMMOS)}, funded by the Norwegian Research Council under grant Evita/205328.} \thanks{Department of Mathematics, University of Oslo, P.O. Box 1053 Blindern,
N-0316 Oslo, Norway, E-mail: salvadoo@math.uio.no}}
\maketitle
\begin{abstract}
The solution of the continuous time filtering problem can be represented
as a ratio of two expectations of certain functionals of the signal
process that are parametrized by the observation path. We introduce
a class of discretization schemes of these functionals of arbitrary
order. The result generalizes the classical work of Picard, who introduced
first order discretizations to the filtering functionals. For a given
time interval partition, we construct discretization schemes with
convergence rates that are proportional with the $m$-power of the
mesh of the partition for arbitrary $m\in\mathbb{N}$. The result
paves the way for constructing high order numerical approximation
for the solution of the filtering problem.


\end{abstract}
\tableofcontents{}

\section{Introduction}

Partially observed dynamical systems are ubiquitous in a multitude
of real-life phenomena. The dynamical system is typically modelled
by a continuous time stochastic process called the signal process
$X$. The signal process cannot be measured directly, but only via
a related process $Y$, called the observation process. The filtering
problem is that of estimating the current state of the dynamical system
at the current time given the observation data accumulated up to that
time. Mathematically the problem entails computing the conditional
distribution of the signal process $X_{t}$, denoted by $\pi_{t}$,
given $\mathcal{Y}_{t},$ the $\sigma$-algebra generated by $Y$.
In a few special cases, $\pi_{t}$ can be expressed in closed form
as a functional of the observation path. For example, the celebrated
Kalman-Bucy filter does this in the linear case. In general, an explicit
formula for $\pi_{t}$ is not available and inferences can only be
made by numerical approximations of $\pi_{t}$. As expected the problem
has attracted a lot of attention in the last fifty years (see Chapter
8 of \cite{BaCr08} for a survey of existing numerical methods for
approximating $\pi_{t}$.

The basis of this class of numerical methods is the representation
of $\pi_{t}$ given by the Kallianpur\textendash Striebel formula
(see (\ref{eq: Kallianpur-Striebel}) below). In the case when the
signal process is modelled by the solution of a stochastic differential
equation (SDE) and the observation process is a function of the signal
perturbed by white noise (see Section \ref{sec: Main Result} below
for further details), the formula entails the computation of expectations
of functionals of the solution of the signal SDE that are parametrized
by the observation path. The numerical approximation of $\pi_{t}$
requires three procedures:\medskip{}

$\bullet$ the discretization of the functionals (corresponding to
a partition of the interval $\left[0,t\right]$).\smallskip{}

$\bullet$ the approximation of the law of the signal with a discrete
measure.\smallskip{}

$\bullet$ the control of the computational effort. \medskip{}

The first step is typically achieved by the discretization scheme
introduced by Picard in \cite{Pi84}. This offers a first order approximation
for the functionals appearing in formula (\ref{eq: Kallianpur-Striebel}).
More precisely, the $L^{1}$-rate of convergence of the approximation
is proportional with the mesh of the partition of the time interval
$\left[0,t\right]$ (see Theorem 21.5 in \cite{Cris11}). The second
and the third step are achieved by a combination of an Euler approximation
of the solution of the SDE, a Monte Carlo step that gives a sample
from the law of the Euler approximation and a re-sampling step that
acts as a variance reduction method and keeps the computational effort
in control. There are a variety of algorithms that follow this template.
Further details can be found, for instance, in Part VII of \cite{CrRo11}.
It is worth pointing out that once the functional discretization and
the Euler approximation have been applied, the problem can be reduced
to one where the signal evolves and is observed in discrete time.
The discrete version of the filtering problem is popular both with
practitioners and with theoreticians. The majority of the existing
theoretical results and the numerical algorithms are constructed and
analyzed in the discrete framework. For more details, the interested
reader can consult the comprehensive theoretical monograph \cite{Delm04}
and the reference therein and the equally comprehensive methodological
volume \cite{DFG01} and the references therein with some updates
in Part VII of \cite{CrRo11}.

The first order discretization introduced by Picard creates a bottleneck:
There exist higher order schemes for approximating the law of the
signal that can be used, but which won't bring any substantial improvements
because of this. For example, in the recent paper \cite{CrOr2013},
the authors employ high order cubature methods to approximate the
law of the signal with only minimal improvements due to the low order
discretization of the required functionals. The aim of this paper
is to address this issue. More precisely, we introduce a class of
high order discretizations of the functionals. As we shall see, we
prove that the $L^{1}$-rate of convergence of the approximations
is proportional with the $m$-power of the mesh of the partition of
the time interval $\left[0,t\right]$. For details, see Theorem \ref{thm: Main Filtering_2}
below. In a work in progress, this discretization procedure is employed
to produce a second order particle filter. It is hoped that this discretization
will be used in conjunction with other high order approximations of
the law of the signal, in particular with cubature methods. We are
not aware of any other similar high order discretization schemes.

The paper is organized as follows: In Section \ref{sec: Main Result}
we introduce some basic definitions and state the main result of the
paper, Theorem \ref{thm: Main Filtering_2}. Section \ref{sec: ProofMainResult}
is devoted to prove our main result. We start by proving several auxiliary
results on iterated stochastic integrals and on the integrability
of the likelihood function and its discretizations. These lead to
the two main results of the section, Proposition \ref{prop: Main1}
and Proposition \ref{prop: Main2}, from which we will deduce our
main result. In Section \ref{sec:Technical-Lemmas} we address the
most technical aspects of the paper. We first introduce some technical
tools on Malliavin calculus (subsection \ref{subsec:Malliavin-calculus}),
the Stroock-Taylor formula (subsection \ref{subsec:MartingaleRep})
and backward martingales (subsection \ref{subsec:BackwardMart}).
Then, with the aid of the these tools, we prove in subsection \ref{subsec:CondExpEst}
the estimates on the conditional expectation with respect to $\mathcal{Y}_{t}$
that are essential in proving Proposition \ref{prop: Main1}.

\section{Basic framework and statement of the main result\label{sec: Main Result}}

Let $(\Omega,\mathcal{F},P)$ be a probability space together with
a filtration $(\mathcal{F}_{t})_{t\geq0}$ which satisfies the usual
conditions. On $(\Omega,\mathcal{F},P)$ we consider a $d_{X}\times d_{Y}$-dimensional
partially observed system $(X,Y)$ satisfying 
\begin{align*}
X_{t} & =X_{0}+\int_{0}^{t}f(X_{s})ds+\int_{0}^{t}\sigma(X_{s})dV_{s},\\
Y_{t} & =\int_{0}^{t}h(X_{s})ds+W_{t},
\end{align*}
where $V$ is a standard $\mathcal{F}_{t}$-adapted $d_{V}$-dimensional
Brownian motion and and $W$ is a a standard $\mathcal{F}_{t}$-adapted
$d_{Y}$-dimensional Brownian motion, independent of each other. We
also assume that $X_{0}$ is a random variable independent of $V$
and $W$ and denote by $\pi_{0}$ its law. We assume that $f=(f_{i})_{i=1,...,d_{X}}:\mathbb{R}^{d_{X}}\rightarrow\mathbb{R}^{d_{X}}$
and $\sigma=(\sigma_{i,j})_{i=1,...d_{X},j=1,...,d_{V}}:\mathbb{R}^{d_{X}}\rightarrow\mathbb{R}^{d_{X}\times d_{V}}$
are globally Lipschitz continuous. In addition, we assume that $h=\left(h_{i}\right)_{i=1,...,d_{Y}}:\mathbb{R}^{d_{X}}\rightarrow\mathbb{R}^{d_{Y}}$
is measurable and has linear growth.

Let $\mathbb{Y}=\{\mathcal{Y}_{t}\}_{t\geq0}$ be the usual augmentation
of the filtration generated by the process $Y,$ that is, $\mathcal{Y}_{t}=\sigma\left(Y_{s},s\in\left[0,t\right]\right)\vee\mathcal{N},$
where $\mathcal{N}$ are all the $P$-null sets of $(\Omega,\mathcal{F},P)$.
We are interested in determining $\pi_{t},$ the conditional law of
the signal $X$ at time $t$ given the information accumulated from
observing $Y$ in the interval $[0,t].$ More precisely, for any Borel
measurable and bounded function $\varphi,$ we want to compute $\pi_{t}\left(\varphi\right)=\mathbb{E}[\varphi\left(X_{t}\right)|\mathcal{Y}_{t}].$
By an application of Girsanov's theorem (see, for example, Chapter
3 in \cite{BaCr08}) one can construct a new probability measure $\tilde{P}$,
absolutely continuous with respect to $P$, under which $Y$ becomes
a Brownian motion independent of $X$ and the law of $X$ remains
unchanged. The Radon-Nikodym derivative of $\tilde{P}$ with respect
to $P$ is given by the process $Z(X,Y)=\left(Z_{t}(X,Y)\right)_{t\geq0}$
given by 
\begin{equation}
Z_{t}(X,Y)=\exp\left(\sum_{i=1}^{d_{Y}}\int_{0}^{t}h_{i}(X_{s})dY_{s}^{i}-\frac{1}{2}\sum_{i=1}^{d_{Y}}\int_{0}^{t}h_{i}^{2}\left(X_{s}\right)ds\right),\quad t\geq0,\label{eq: Likelihood}
\end{equation}
which is an $\mathcal{F}_{t}$-adapted martingale under $\tilde{P}$
under the assumptions introduced above. We will denote by $\mathbb{\tilde{E}}$
to be the expectation with respect to $\tilde{P}$. In the following
we will make use of the measure valued process $\rho=\left(\rho_{t}\right)_{t\geq0},$
defined by the formula $\rho_{t}\left(\varphi\right)=\mathbb{\tilde{E}}\left[\varphi\left(X_{t}\right)Z_{t}|\mathcal{Y}_{t}\right],$
for any bounded Borel measurable function $\varphi$.\ The processes
$\pi$ and $\rho$ are connected through the Kallianpur-Striebel's
formula: 
\begin{align}
\pi_{t}\left(\varphi\right)=\frac{\rho_{t}\left(\varphi\right)}{\rho_{t}\left(\boldsymbol{1}\right)} & =\frac{\mathbb{\tilde{E}}\left[\varphi\left(X_{t}\right)\exp\left.\left(\sum_{i=1}^{d_{Y}}\int_{0}^{t}h_{i}(X_{s})dY_{s}^{i}-\frac{1}{2}\sum_{i=1}^{d_{Y}}\int_{0}^{t}h_{i}^{2}\left(X_{s}\right)ds\right)\right\vert \mathcal{Y}_{t}\right]}{\mathbb{\tilde{E}}\left[\exp\left.\left(\sum_{i=1}^{d_{Y}}\int_{0}^{t}h_{i}(X_{s})dY_{s}^{i}-\frac{1}{2}\sum_{i=1}^{d_{Y}}\int_{0}^{t}h_{i}^{2}\left(X_{s}\right)ds\right)\right\vert \mathcal{Y}_{t}\right]},\label{eq: Kallianpur-Striebel}
\end{align}
$P$-a.s., where $\boldsymbol{1}$ is the constant function. As a
result, $\rho$ is called the unnormalized conditional distribution
of the signal. For further details on the filtering framework, see
\cite{BaCr08}.

It follows from (\ref{eq: Kallianpur-Striebel}) that $\pi_{t}\left(\varphi\right)$
is a ratio of two conditional expectations of functionals of the signal
that depend on the stochastic integrals with respect to the process
$Y.$ In the following we will introduce a class of time discretization
schemes for these conditional expectations which, in turn, will generate
time discretisation schemes $\pi_{t}$ (of any order). This is the
main result of the paper and is stated Theorem \ref{thm: Main Filtering_2}
below.

We first introduce some useful notation and definitions. We denote
by: 
\begin{itemize}
\item $\mathcal{B}_{b}$ the space of bounded Borel-measurable functions. 
\item $\mathcal{B}_{P}$ the space of Borel-measurable functions with polynomial
growth. 
\item $C_{b}^{k}$ the space of continuously differentiable functions up
to order $k\in\mathbb{Z}_{+}$ with bounded derivatives of order greater
or equal to one. 
\item $C_{P}^{k}$ the space of continuously differentiable functions up
to order $k\in\mathbb{Z}_{+}$ such that the function and its derivatives
have at most polynomial growth. 
\item $L^{p}(\Omega,\mathcal{F},\tilde{P})$ the space of $p$-integrable
random variables (with respect to $\tilde{P}$) and denote by $\left\vert \left\vert \cdot\right\vert \right\vert _{p}$
the corresponding norm on $L^{p}(\Omega,\mathcal{F},\tilde{P})$,
i.e., for $\xi\in L^{p}(\Omega,\mathcal{F},\tilde{P})$,\ $\left\vert \left\vert \xi\right\vert \right\vert _{p}\triangleq\mathbb{\tilde{E}}[\left\vert \xi\right\vert ^{p}]^{1/p}$. 
\end{itemize}
In the following, we will use the notation introduced in Section 5.4
in Kloeden and Platen \cite{KlPl92}. More precisely, let $S$ be
a subset of $\mathbb{Z}_{+}$ and denote by $\mathcal{M}^{\ast}(S)$
the set of all multi-indices with values in $S.$ In addition, define
$\mathcal{M}(S)\triangleq\mathcal{M}^{\ast}(S)\cup\{v\},$ where $v$
denotes the multi-index of lenght zero . For $\alpha=(\alpha_{1},...,\alpha_{k})\in\mathcal{M}(S)$
define the following operations 
\begin{align*}
\left\vert \alpha\right\vert  & \triangleq k,\\
\left\vert \alpha\right\vert _{0} & \triangleq\#\{j:\alpha_{j}=0,j=1...,k\},\\
\alpha- & \triangleq(\alpha_{1},...,\alpha_{k-1}),\\
-\alpha & \triangleq(\alpha_{2},...,\alpha_{k}),
\end{align*}
where $\left|v\right|=0,-v=v-=v$. Given two multi-indices $\alpha,\beta\in\mathcal{M}(S)$
we denote its concatenation by $\alpha\ast\beta$ . Itô-Taylor expansions
are usually done with a particular subsets of multi-indices, the so
called hierarchical sets. We call a subset $\mathcal{A}\subset\mathcal{M\left(S\right)}$
a hierarchical set if $\mathcal{A}$ is nonempty, $\sup_{\alpha\in\mathcal{A}}\left|\alpha\right|<\infty,$
and 
\[
-\alpha\in\mathcal{A\mbox{ \quad if}\quad}\alpha\in\mathcal{A}\setminus\{v\}.
\]
For any given hierarchichal set $\mathcal{A}$ we define the remainder
set $\mathcal{R\left(A\right)}$ of $\mathcal{A}$ by 
\[
\mathcal{R\left(A\right)}\triangleq\left\{ \alpha\in\mathcal{M\left(S\right)}\setminus\mathcal{A}:-\alpha\in\mathcal{A}\right\} .
\]
We will consider the hierarchical set $\mathcal{M}_{m}(S)$ and its
associated remainder set $\mathcal{R}\left(\mathcal{M}_{m}(S)\right),$
that is, 
\[
\mathcal{M}_{m}(S)\triangleq\{\alpha\in\mathcal{M}(S):\left\vert \alpha\right\vert \leq m\},
\]
and 
\[
\mathcal{R}\left(\mathcal{M}_{m}(S)\right)\triangleq\{\alpha\in\mathcal{M}(S):\left\vert \alpha\right\vert =m+1\}.
\]
Observe that $\mathcal{R}\left(\mathcal{M}_{m}(S)\right)=\mathcal{M}_{m+1}(S)\setminus\mathcal{M}_{m}(S)$.
We shall use the sets of multi-indices with values in the sets $S_{0}=\{0,1,...,d_{V}\}$
and $S_{1}=\{1,...,d_{V}\}$. Note also that the set $\mathcal{R}\left(\mathcal{M}_{m}(S_{0})\right)$
can be partioned in the following way 
\[
\mathcal{R}\left(\mathcal{M}_{m}(S_{0})\right)={\displaystyle \biguplus\limits _{k=0}^{m+1}}\mathcal{R}\left(\mathcal{M}_{m}(S_{0})\right)_{k},
\]
where $\mathcal{R}\left(\mathcal{M}_{m}(S_{0})\right)_{k}=\{\alpha\in\mathcal{R}\left(\mathcal{M}_{m}(S_{0})\right):\left\vert \alpha\right\vert _{0}=k\},k=0,...,m+1,$
that is, $\mathcal{R}\left(\mathcal{M}_{m}(S_{0})\right)_{k}$ is
the set of multi-indices of lenght $m+1$ with values in $S_{0}$
which contains $k$ zeros.

To simplify the notation, it is convenient to add an additional component
to the Brownian motion $V.$ Let $V_{s}^{0}=s,$ for all $s\geq0$
and consider the $(d_{V}+1)$-dimensional process $V=(V^{i})_{i=0}^{d_{V}}.$
We will consider the filtration $\mathbb{F}^{0,V}=\{\mathcal{F}_{s}^{0,V}\}_{s\geq0}$
defined to be the usual augmentation of the filtration generated by
the process $V$ and initially enlarged with the random variable $X_{0}.$
Moreover, for fixed $t$, we will also consider the filtration $\mathbb{H}^{t}=\{\mathcal{H}_{s}^{t}\triangleq\mathcal{F}_{s}^{0,V}\lor\mathcal{Y}_{t}\}_{s\geq0}$.
For $\alpha\in\mathcal{M}(S_{0}),$ denote by $I_{\alpha}(h_{\cdot})_{s,t}$
the following Itô iterated integral 
\[
I_{\alpha}(h_{\cdot})_{s,t}=\left\{ \begin{array}{ccc}
h_{t} & \text{if} & \alpha=v\\
\int_{s}^{t}I_{\alpha-}(h_{\cdot})_{s,u}dV_{u}^{\alpha_{|\alpha|}} & \text{if} & \left\vert \alpha\right\vert \geq1
\end{array}\right.,
\]
where $h=\{h_{s}\}_{s\geq0}$ is an $\mathbb{H}^{t}$-adapted process
(satisfying appropriate integrability conditions). We introduce the
differential operators $L^{0},L^{r},r=1,...,d_{V}$ defined by 
\begin{align*}
L^{0}g(x) & \triangleq\sum_{k=1}^{d_{X}}f^{k}(x)\frac{\partial g}{\partial x^{k}}(x)+\frac{1}{2}\sum_{k,l=1}^{d_{X}}\sum_{r=1}^{d_{V}}\sigma_{k,r}(x)\sigma_{l,r}(x)\frac{\partial^{2}g}{\partial x^{k}\partial x^{l}}(x).\\
L^{r}g(x) & \triangleq\sum_{k=1}^{d_{X}}\sigma_{k,r}(x)\frac{\partial g}{\partial x^{k}}(x),\quad r=1,...,d_{V},
\end{align*}
where $g:\mathbb{R}^{d_{X}}\rightarrow\mathbb{R}$ belongs to $C_{P}^{2}\left(\mathbb{R}^{d_{X}};\mathbb{R}\right).$
For $\alpha\in\mathcal{M}(S_{0}),$ with $\alpha=(\alpha_{1},...,\alpha_{k}),$
and the differential operator $L^{\alpha}$ is defined by 
\begin{align*}
L^{\alpha}g & =L^{\alpha_{1}}\circ L^{\alpha_{2}}\circ\cdots\circ L^{\alpha_{k}}g,
\end{align*}
and, by convention $L^{v}g=g$. Finally, let $\tau\triangleq\{0=t_{0}<\cdots<t_{i}<\cdots<t_{n}=t\}$
be a partition of $[0,t].$ Associated to $\tau$ we define the following
elements 
\begin{align*}
\delta_{i} & \triangleq t_{i}-t_{i-1},\quad i=1,...,n,\\
\delta & \triangleq\max_{i=1,...,n}\delta_{i},\\
\delta_{\mbox{min}} & \triangleq\min_{i=1,...,n}\delta_{i},\\
\tau(s) & \triangleq t_{i-1},\quad s\in[t_{i-1},t_{i}),i=1,...,n,\\
\eta(s) & \triangleq t_{i},\quad s\in[t_{i-1},t_{i}),i=1,...,n.
\end{align*}
We will only consider partitions satisfying the following condition
\begin{equation}
\delta\leq C\delta_{\mbox{min}},\label{eq: UnifPartition}
\end{equation}
for some finite constant $C\geq1$. We denote by $\Pi(t)$ the set
of all partitions of $[0,t]$ satisfying $\left(\ref{eq: UnifPartition}\right)$
and such that $\delta$ converges to zero when $n$ tends to infinity.
We denote by $\Pi(t,\delta_{0})$ the set of all partitions of $[0,t]$
satisfying $\left(\ref{eq: UnifPartition}\right)$, such that $\delta$
converges to zero when $n$ tends to infinity and $\delta<\delta_{0}$. 
\begin{rem}
\label{rem: Uniform}Under the assumption $\left(\ref{eq: UnifPartition}\right)$
one has that 
\begin{equation}
n\leq t\delta_{\mbox{min}}^{-1}\leq Ct\delta^{-1}.\label{eq: UnifPartition2}
\end{equation}
\end{rem}
To simplify the notation, we will add an additional component to the
Brownian motion $Y.$ Let $Y^{0}$ be the process $Y_{s}^{0}=s,$
for all $s\geq0$ and consider the $(d_{Y}+1)$-dimensional process
$Y=(Y^{i})_{i=0}^{d_{Y}}.$ Then the martingale $Z=\left(Z_{t}\right)_{t\geq0}$
defined in (\ref{eq: Likelihood}) can be written as $Z_{t}=\exp\left(\xi_{t}\right),t\geq0,$
where 
\[
\xi_{t}=\sum_{i=0}^{d_{Y}}\int_{0}^{t}h_{i}(X_{s})dY_{s}^{i},\quad t\geq0,
\]
and $h_{0}=-\frac{1}{2}\sum_{i=1}^{d_{Y}}h_{i}^{2}.$ For $\tau\in\Pi(t)$
and $m\in\mathbb{N}$ we consider the processes 
\begin{align*}
\xi_{t}^{\tau,m} & \triangleq\sum_{j=0}^{n-1}\xi_{t}^{\tau,m}\left(j\right)\triangleq\sum_{j=0}^{n-1}\sum_{i=0}^{d_{Y}}\sum_{\alpha\in\mathcal{M}_{m-1}(S_{0})}L^{\alpha}h_{i}(X_{t_{j}})\int_{t_{j}}^{t_{j+1}}I_{\alpha}(\boldsymbol{1})_{t_{j},s}dY_{s}^{i}\\
 & =\sum_{i=0}^{d_{Y}}\int_{0}^{t}\left\{ \sum_{\alpha\in\mathcal{M}_{m-1}(S_{0})}L^{\alpha}h_{i}(X_{\tau(s)})I_{\alpha}(\boldsymbol{1})_{\tau(s),s}\right\} dY_{s}^{i}.
\end{align*}
For $m>2$, we can write

\[
\xi_{t}^{\tau,m}=\xi_{t}^{\tau,2}+\sum_{j=0}^{n-1}\mu^{\tau,m}\left(j\right),
\]
where 
\[
\mu^{\tau,m}\left(j\right)\triangleq\sum_{i=0}^{d_{Y}}\sum_{\alpha\in\mathcal{M}_{1,m-1}(S_{0})}L^{\alpha}h_{i}(X_{t_{j}})\int_{t_{j}}^{t_{j+1}}I_{\alpha}(\boldsymbol{1})_{t_{j},s}dY_{s}^{i},
\]
and 
\begin{align*}
\mathcal{M}_{1,m-1}(S_{0}) & \triangleq\mathcal{M}_{m-1}(S_{0})\setminus\mathcal{M}_{1}(S_{0})\\
 & =\left\{ \alpha:\left\vert \alpha\right\vert \in\left[2,m-1\right],\alpha_{k}\in\{0,...,d_{V}\},k=1,...,\left\vert \alpha\right\vert \right\} .
\end{align*}
The processes $\xi^{\tau,m}$ are obtained by replacing $h_{i}(X_{s})$
in the formula for the process $\xi$ with the truncation of degree
$(m-1)$ of the corresponding stochastic Taylor expansion of $h_{i}(X_{s})$.
They are used to produce discretization schemes of order 1 and 2 for
$\pi_{t}(\varphi)$. They \emph{cannot} be used to produce discretization
schemes of order $m>2$ as they don't have finite exponential moments
(required to define the discretization schemes). More precisely, the
quantities $\mu^{\tau,m}\left(j\right)$ \ do not have finite exponential
moments because of the high order iterated integral involved. For
this, we need to introduce a truncation of $\mu^{\tau,m}\left(j\right)$
resulting in a (partial) taming procedure to the stochastic Taylor
expansion of $h_{i}(X)$. We define the processes 
\[
\bar{\xi}_{t}^{\tau,m}\triangleq\sum_{j=0}^{n-1}\bar{\xi}_{t}^{\tau,m}\left(j\right),
\]
where 
\[
\bar{\xi}_{t}^{\tau,i}\left(j\right)=\left\{ \begin{array}{ll}
\xi_{t}^{\tau,i}\left(j\right) & \mathrm{if}\ \ i=1,2\\
\xi_{t}^{\tau,2}\left(j\right)+\Gamma_{m-\frac{1}{2},\delta_{j}}\left(\mu^{\tau,m}\left(j\right)\right) & \mathrm{if}\ \ i>2
\end{array}\right.,~~j=0,...,n-1
\]
with the truncation function $\Gamma$ being defined as 
\[
\Gamma_{q,\delta}\left(z\right)=\frac{z}{1+\left(z/\delta\right)^{2q}},\qquad z\in\mathbb{R}
\]
for some $\delta>0$ and $q\in\mathbb{N}$. Finally, for $\tau\in\Pi(t)$
and $m\in\mathbb{N}$ consider the processes $Z^{\tau,m}=$$\left(Z_{t}^{\tau,m}\right)_{t\geq0}$
given by 
\begin{align}
Z_{t}^{\tau,m} & =\exp\left(\bar{\xi}_{t}^{\tau,m}\right).\label{eq: Discrete Z}
\end{align}
For any Borel measurable function $\varphi$ such that $\varphi\left(X_{t}\right)Z_{t}^{\tau,m}\in L^{1}(\Omega,\mathcal{F},\tilde{P})$
define the $m$-th order discretizations 
\begin{align*}
\rho_{t}^{\tau,m}\left(\varphi\right) & \triangleq\mathbb{\tilde{E}}\left[\varphi\left(X_{t}\right)Z_{t}^{\tau,m}|\mathcal{Y}_{t}\right],
\end{align*}
and 
\[
\pi_{t}^{\tau,m}\left(\varphi\right)\triangleq\rho_{t}^{\tau,2}\left(\varphi\right)/\rho_{t}^{\tau,m}\left(\boldsymbol{1}\right),
\]
of $\rho_{t}$ and $\pi_{t}$, respectively.

Let $m\in\mathbb{N}$, our main assumption is the following: 
\begin{assumption*}[\textbf{\emph{H$\left(m\right)$}}]
 We have that:\\
$\bullet$ $f=\left(f_{i}\right){}_{i=1,...,d_{X}}:\mathbb{R}^{d_{X}}\rightarrow\mathbb{R}^{d_{X}}\in\mathcal{B}_{b}\cap C_{b}^{2\vee(2m-1)},$\\
 $\bullet$ $\sigma=\left(\sigma_{i,j}\right){}_{i=1,...d_{X},j=1,...,d_{V}}:\mathbb{R}^{d_{X}}\rightarrow\mathbb{R}^{d_{X}\times d_{V}}\in\mathcal{B}_{b}\cap C_{b}^{2m},$\\
 $\bullet$ $h=\left(h_{i}\right){}_{i=0,...,d_{Y}}:\mathbb{R}^{d_{X}}\rightarrow\mathbb{R}^{d_{Y}+1}\in\mathcal{B}_{b}\cap C_{b}^{2m+1},$\\
 $\bullet$ $X_{0}$ has moments of all orders. 
\end{assumption*}
Note that if assumption \textbf{H}$(m)$ holds for some $m\in\mathbb{N},$
then it also holds for any $n\leq m$. 
\begin{thm}
\label{thm: Main Filtering_2}Let assumption \textbf{$\mathbf{H}\left(m\right)$}
be satisfied. Then, there exists constants $\delta_{0},C>0$ not depending
on the choice of the partition $\tau\in\Pi(t,\delta_{0}),$ such that
\[
\left\Vert \rho_{t}\left(\varphi\right)-\rho_{t}^{\tau,m}\left(\varphi\right)\right\Vert _{2}\leq C\delta^{m},
\]
for $\varphi\in C_{P}^{m+1}$. Moreover, if $\sup_{\tau\in\Pi(t,\delta_{0})}\left\Vert \pi_{t}^{\tau,m}(\varphi)\right\Vert _{2+\varepsilon}<\infty,$
for some $\varepsilon>0,$ then 
\[
\mathbb{E}\left[\left\vert \pi_{t}\left(\varphi\right)-\pi_{t}^{\tau,m}\left(\varphi\right)\right\vert \right]\leq\bar{C}\delta^{m},
\]
where $\bar{C}$ is another constant independent of $\tau\in\Pi(t,\delta_{0})$. 
\end{thm}
\begin{rem}
The assumption $\sup_{\tau\in\Pi(t,\delta_{0})}\mathbb{\tilde{E}}\left[\left\vert \pi_{t}^{\tau,m}(\varphi)\right\vert ^{2+\varepsilon}\right]<\infty$
for some $\varepsilon>0$ is satisfied if $\varphi$ is bounded. If
$\varphi$ is unbounded, note that by using Jensen's inequality one
has 
\begin{align*}
\mathbb{\tilde{E}}\left[\left\vert \pi_{t}^{\tau,m}(\varphi)\right\vert ^{2+\varepsilon}\right] & =\mathbb{\tilde{E}}\left[\left\vert \mathbb{\tilde{E}}\left[\frac{\varphi(X_{t})Z_{t}^{\tau,m}}{\mathbb{\tilde{E}}\left[Z_{t}^{\tau,m}|\mathcal{Y}_{t}\right]}|\mathcal{Y}_{t}\right]\right\vert ^{2+\varepsilon}\right]\\
 & \leq\mathbb{\tilde{E}}\left[|\varphi(X_{t})|^{2+\varepsilon}\mathbf{\exp}((2+\varepsilon)(\bar{\xi}_{t}^{\tau,m}-\mathbb{\tilde{E}}[\bar{\xi}_{t}^{\tau,m}|\mathcal{Y}_{t}]))\right].
\end{align*}
Hence, one can reason as in Lemma \ref{lem: Xi_Tau2^p_Integrability}
to justify that $\sup_{\tau\in\Pi(t,\delta_{0})}\mathbb{\tilde{E}}\left[\left\vert \pi_{t}^{\tau,m}(\varphi)\right\vert ^{2+\varepsilon}\right]<\infty.$ 
\end{rem}
\begin{rem}
\label{rem: High Order} $\left.\right.$\\
 i. In the case $m=1$ we can consider any partition $\tau\in\Pi(t)$.
For $m\geq2,$ we must consider partitions $\tau$ with mesh $\delta$
smaller than 
\begin{equation}
\delta_{0}=\frac{1}{2\left\Vert Lh\right\Vert _{\infty}\sqrt{d_{Y}d_{V}}},\label{eq: Delta_0}
\end{equation}
where 
\[
\left\Vert Lh\right\Vert _{\infty}\triangleq\max_{i=1,...d_{Y}\ r=1,...,d_{V}}\left\Vert L^{r}h^{i}\right\Vert _{\infty}.
\]

ii. The functional discretization given in (\ref{eq: Discrete Z})
is recursive. More precisely, if $\tau^{\prime}\in\Pi(t+s)$ is a
partition that includes $t$ as an intermediate point, for example
$\tau^{\prime}\triangleq\{0=t_{0}<\cdots<t_{k}=t<t_{k+1}\cdots<t_{n}=t+s\}$
with $0<k<n$, then 
\begin{align*}
Z_{t+s}^{\tau^{\prime},m} & =Z_{t}^{\tau,m}\prod_{j=k}^{n-1}\exp\left(\bar{\xi}_{t}^{\tau,m}\left(j\right)\right).
\end{align*}
This property is essential for implementation purposes as at every
discretization time we only need to use the previous functional discretization
and the term corresponding to the next interval to obtain the new
functional discretization.

iii. The discretization introduced by Picard in \cite{Pi84} corresponds
to the case $m=1$. In this case, $\rho_{t}^{\tau,m}$ can be explicitly
written as 
\begin{equation}
\rho_{t}^{\tau,m}\left(\varphi\right)\triangleq\mathbb{\tilde{E}}\left[\varphi\left(X_{t}\right)\left.\exp\left(\sum_{j=0}^{n-1}\sum_{i=0}^{d_{Y}}h_{i}(X_{t_{j}})\left(Y_{t_{j+1}}^{i}-Y_{t_{j}}^{i}\right)\right)\right\vert \mathcal{Y}_{t}\right],\label{PicardFilter}
\end{equation}
This discretization scheme leads to a wealth of numerical methods
that can be used to approximate $\pi_{t}$. Among them, particle methods\footnote{Also known as \emph{particle filters }or \emph{sequential Monte Carlo
methods}.} are algorithms which approximate $\pi_{t}$ with discrete random
measures of the form $\sum_{i}a_{i}(t)\delta_{v_{i}(t)},$ in other
words with empirical distributions associated with sets of randomly
located particles of stochastic masses $a_{1}(t)$,$a_{2}(t)$, \dots ,
which have stochastic positions $v_{1}(t),v_{2}(t),\ldots$\ . These
methods are currently among the most successful and versatile for
numerically solving the filtering problem. Based on \eqref{PicardFilter},
the ``garden variety\char`\"{} particle filter uses particles that
evolve according to the signal equation (or, rather, the Euler approximation
of the signal) and carry exponential weights. These weights are proportional
with 
\[
\exp\left(\sum_{i=0}^{d_{Y}}h_{i}(v_{t_{n}}^{j})\left(Y_{t_{n+1}}^{i}-Y_{t_{n}}^{i}\right)\right),
\]
where $v^{j}$ is the process modelling the trajectory of the particle
and $t_{n}$ is the update time. The method also involves a variance
reduction procedure (for further details, see for example Chapter
9 in \cite{BaCr08}). Alternatively one can use a cubature method
to approximate the law of the signal, see \cite{CrOr2013}. In both
cases, higher order approximations of the signal can be used, but
this would not improve the rate of convergence of the method as Picard's
discretisation has an error of order 1. The remedy is to exploit the
result in this paper and use a higher order discretisation. The second
author is working on a particle filter that uses the second order
discretisation presented in this paper. 
\end{rem}

\section{Proof of the main result \label{sec: ProofMainResult}}

We start by recalling and introducing some basic results on iterated
integrals and martingale representations. Throughout the rest of the
paper we will be assuming that \textbf{H}$(m)$ holds, without recalling
it in each result statement. Moreover, $C$ will denote a constant
that usually depends on $d_{V},$ $d_{X},$ $d_{Y},$ $f,$ $\sigma,$
$h$ and possibly other parameters but NOT on the partition $\tau.$
As we are interested in showing a rate of convergence for our approximations,
the particular form of dependence of $C$ with respect to these parameters
is not relevant and, hence, omitted. Of course, the choice of the
constant $C$ may change from line to line. 
\begin{rem}
\label{rem: M} Some immediate consequences of assumption \textbf{H}$(m)$
are the following: 
\begin{enumerate}
\item The signal process $X$ has moments of all orders and for any $p\geq1,$
we have 
\[
\mathbb{\tilde{E}}\left[\sup_{s\in[0,t]}\left|X_{s}^{i}\right|{}^{p}\right]<\infty,
\]
for all $i\in\{1,...,d_{X}\}$. 
\item If $\Upsilon:\mathbb{R}^{d_{X}}\rightarrow\mathbb{R}$ is a function
with polynomial growth we have 
\[
\mathbb{\tilde{E}}\left[\sup_{s\in[0,t]}\left|\Upsilon(X_{s})\right|{}^{p}\right]<\infty,
\]
in particular, 
\[
\mathbb{\tilde{E}}\left[\sup_{s\in[0,t]}\left|L^{\alpha}h_{i}(X_{s})\right|^{p}\right]<\infty,
\]
for $i=0,...,d_{Y}$ and $\alpha\in\mathcal{M}_{m}\left(S_{0}\right)=\mathcal{M}_{m-1}\left(S_{0}\right)\biguplus\mathcal{R}\left(\mathcal{M}_{m-1}\left(S_{0}\right)\right).$ 
\item The processes $\xi_{t}$ and $\xi_{t}^{\tau,m},m\in\mathbb{N},$ as
defined above have finite moments of all orders. 
\end{enumerate}
\end{rem}
\begin{rem}
\label{rem:TruncProc}Consider the truncation function 
\[
\Gamma_{q,\delta}(z)=\frac{z}{1+(z/\delta)^{2q}}.
\]
defined as above corresponding to the real parameters $q\geq1$ and
$\delta>0$. 
\begin{enumerate}
\item For any $z\in\mathbb{R}$, $\left\vert \Gamma_{q,\delta}(z)\right\vert \le\delta$.
To check this observe that if $\left\vert z\right\vert \leq\delta$
we have that $1+(z/\delta)^{2q}\geq1$ and then 
\[
\left\vert \Gamma_{q,\delta}(z)\right\vert =\frac{\left\vert z\right\vert }{1+(z/\delta)^{2q}}\leq\left\vert z\right\vert \leq\delta.
\]
On the other hand, if $\left\vert z\right\vert >\delta$ we have that
$\left\vert z/\delta\right\vert ^{-1}+\left\vert z/\delta\right\vert ^{2q-1}>1$
and then 
\[
\left\vert \Gamma_{q,\delta}(z)\right\vert =\frac{\left\vert z\right\vert }{1+\left\vert z\right\vert ^{2q}\delta^{-2q}}=\frac{1}{\left\vert z\right\vert ^{-1}+\left\vert z\right\vert ^{2q-1}\delta^{-2q}}=\frac{\delta}{\left\vert z/\delta\right\vert ^{-1}+\left\vert z/\delta\right\vert ^{2q-1}}\leq\delta.
\]
\item Moreover, if we define 
\[
\mathcal{E}_{q,\delta}(z)\triangleq\Gamma_{q,\delta}(z)-z,
\]
we get that 
\begin{align}
\left|\mathcal{E}_{q,\delta}(z)\right| & =\left\vert \Gamma_{q,\delta}(z)-z\right\vert =\left\vert \frac{z}{1+(z/\delta)^{2q}}-z\right\vert =\frac{\left\vert z\right\vert ^{2q+1}\delta^{-2q}}{1+(z/\delta)^{2q}}=\frac{\left\vert z\right\vert ^{2q+1}}{\delta^{2q}+z^{2q}},\nonumber \\
 & =\delta^{-2q}\frac{\left\vert z\right\vert ^{2q+1}}{1+\left(z/\delta\right)^{2q}}\leq\delta^{-2q}\left\vert z\right\vert ^{2q+1},\quad\forall z\in\mathbb{R}.\label{eq:EpsilonInequality}
\end{align}
\item Finally, note that for $m$$\geq3$ we can write 
\begin{align}
\bar{\xi}_{t}^{\tau,m} & \triangleq\xi_{t}^{\tau,2}+\sum_{j=0}^{n-1}\Gamma_{m-\frac{1}{2},\delta_{j}}\left(\mu^{\tau,m}\left(j\right)\right)\nonumber \\
 & =\xi_{t}^{\tau,m}+\sum_{j=0}^{n-1}\mathcal{E}_{m-\frac{1}{2},\delta_{j}}\left(\mu^{\tau,m}\left(j\right)\right)\label{eq:BigEpsilon}
\end{align}
\end{enumerate}
\end{rem}

\subsection{\label{subsec:IteratedIntegrals}Iterated integrals}

The following two results are well known and can be found in Kloeden
and Platen \cite{KlPl92}, Theorem 5.5.1 and Lemma 5.7.5, respectively. 
\begin{thm}
\label{thm: Ito Taylor Expansion}Let $\rho_{1}$ and $\rho_{2}$
be two stopping times with $0\leq\rho_{1}\leq\rho_{2}\leq t,$ a.s.,
let $\mathcal{A}\subset\mathcal{M}(S_{0})$ be a hierarchical set
and $g:\mathbb{R}^{d_{X}}\rightarrow\mathbb{R}.$ Then, the Itô-Taylor
expansion 
\begin{equation}
g(X_{\rho_{2}})=\sum_{\alpha\in\mathcal{A}}L^{\alpha}g(X_{\rho_{1}})I_{\alpha}(\boldsymbol{1})_{\rho_{1},\rho_{2}}+\sum_{\alpha\in\mathcal{R}\left(\mathcal{A}\right)}I_{\alpha}(L^{\alpha}g(X_{\cdot}))_{\rho_{1},\rho_{2}},\label{eq: Ito Taylor Expansion}
\end{equation}
holds, provided all of the derivatives of $g,f$ and $\sigma$ and
all of the iterated Itô integrals appearing in $\left(\ref{eq: Ito Taylor Expansion}\right)$
exist. 
\end{thm}
\begin{lem}
\label{lem: Moments Iterated Integral} Let $\alpha\in\mathcal{M}(S_{0}),$
let $\theta=\{\theta_{s}\}_{s\in[0,t]}$ be an $\mathbb{H}^{t}$-adapted
process, let $p\geq1$ and let $\rho_{1}$ and $\rho_{2}$ be two
stopping times with $0\leq\rho_{1}\leq\rho_{2}\leq t$ and $\rho_{2}$
being $\mathcal{H}_{\rho_{1}}^{t}$-measurable. Then, 
\[
\mathbb{\tilde{E}}\left[\left\vert I_{\alpha}(\theta_{\cdot})_{\rho_{1},\rho_{2}}\right\vert ^{2p}|\mathcal{H}_{\rho_{1}}^{t}\right]\leq CR(\theta)\left(\rho_{2}-\rho_{1}\right)^{p\{|\alpha|+|\alpha|_{0}\}},
\]
where 
\[
R(\theta)=\mathbb{\tilde{E}}\left[\sup_{\rho_{1}\leq s\leq\rho_{2}}\left\vert \theta_{s}\right\vert ^{2p}|\mathcal{H}_{\rho_{1}}^{t}\right].
\]
\end{lem}
The following lemma gives a basic estimate on the difference between
the log likelihood functional $\xi_{t}$ and its $m$-th order discretization.
Its proof relies on Theorem \ref{thm: Ito Taylor Expansion} and Lemma
\ref{lem: Moments Iterated Integral}. 
\begin{lem}
\label{lem: Difference Moment Estimate}We have that 
\begin{equation}
\xi_{t}-\xi_{t}^{\tau,m}=\sum_{i=0}^{d_{Y}}\int_{0}^{t}\left\{ \sum_{\alpha\in\mathcal{R}\left(\mathcal{M}_{m-1}(S_{0})\right)}I_{\alpha}(L^{\alpha}h_{i}(X_{\cdot}))_{\tau(s),s}\right\} dY_{s}^{i},\label{eq: Difference Taylor Expansion}
\end{equation}
and 
\[
\mathbb{\tilde{E}}\left[\left\vert \xi_{t}-\xi_{t}^{\tau,m}\right\vert ^{2p}\right]\leq C\delta^{pm}.
\]
\end{lem}
\begin{proof}
By Remark \ref{rem: M}, we can apply Theorem \ref{thm: Ito Taylor Expansion}
and get equation \ref{eq: Difference Taylor Expansion}. Applying
the Itô isometry and Jensen's inequality (or Jensen's inequality directly
if $i=0$), we obtain the following bound 
\begin{align*}
\mathbb{\tilde{E}}\left[\left|\xi_{t}-\xi_{t}^{\tau,m}\right|^{2p}\right] & =\mathbb{\tilde{E}}\left[\left\vert \sum_{i=0}^{d_{Y}}\int_{0}^{t}\left\{ \sum_{\alpha\in\mathcal{R}\left(\mathcal{M}_{m-1}(S_{0})\right)}I_{\alpha}(L^{\alpha}h_{i}(X_{\cdot}))_{\tau(s),s}\right\} dY_{s}^{i}\right\vert ^{2p}\right]\\
 & \leq C\sum_{\alpha\in\mathcal{R}\left(\mathcal{M}_{m-1}(S_{0})\right)}\int_{0}^{t}\mathbb{\tilde{E}}\left[\left\vert I_{\alpha}(L^{\alpha}h_{i}(X_{\cdot}))_{\tau(s),s}\right\vert ^{2p}\right]ds.
\end{align*}
Let $\alpha\in\mathcal{R}\left(\mathcal{M}_{m-1}(S_{0})\right),$
by Lemma \ref{lem: Moments Iterated Integral} and Remark \ref{rem: M}
we get that 
\begin{align*}
\mathbb{\tilde{E}}\left[\left\vert I_{\alpha}(L^{\alpha}h_{i}(X_{\cdot}))_{\tau(s),s}\right\vert ^{2p}\right] & \leq C\mathbb{\tilde{E}}\left[\sup_{\tau(s)\leq u\leq s}\left\vert L^{\alpha}h_{i}(X_{u})\right\vert ^{2p}\right]\left(s-\tau(s)\right)^{p\{|\alpha|+|\alpha|_{0}\}}\leq C\delta^{pm},
\end{align*}
where in the last inequality we have used that $|\alpha|+|\alpha|_{0}\geq m$
for $\alpha\in\mathcal{R}\left(\mathcal{M}_{m-1}(S_{0})\right).$
From the previous inequality the result follows easily. 
\end{proof}
\begin{lem}
\label{lem: SumEpsilon}Let $p,q\geq1$ and $m\geq3.$ Then, we have
that 
\[
\mathbb{\tilde{E}}\left[\left\vert \sum_{j=0}^{n-1}\mathcal{E}_{q,\delta_{j}}\left(\mu^{\tau,m}\left(j\right)\right)\right\vert ^{2p}\right]\leq C\left(t,d_{Y},p,m\right)\delta^{p\left(2q+1\right)}.
\]
\end{lem}
\begin{proof}
We have that 
\begin{align*}
\mathbb{\tilde{E}}\left[\left\vert \sum_{j=0}^{n-1}\mathcal{E}_{q,\delta_{j}}\left(\mu^{\tau,m}\left(j\right)\right)\right\vert ^{2p}\right] & \leq C\left(p\right)n^{2p-1}\sum_{j=0}^{n-1}\mathbb{\tilde{E}}\left[\left|\mathcal{E}_{q,\delta_{j}}\left(\mu^{\tau,m}\left(j\right)\right)\right|^{2p}\right].
\end{align*}
Moreover, using similar arguments as in Lemma \ref{lem: Difference Moment Estimate},
for any $r\geq1$, we have that 
\begin{align*}
\mathbb{\tilde{E}}\left[\left|\mu^{\tau,m}\left(j\right)\right|^{r}\right] & \leq C\left(d_{Y},m\right)\delta^{\frac{r}{2}(1+\left|\alpha\right|+\left|\alpha\right|_{0})}\\
 & \leq C\left(d_{Y},m\right)\delta^{\frac{3}{2}r},
\end{align*}
because as $\alpha\in\mathcal{M}_{1,m-1}(S_{0}),m\geq3$ then $\left|\alpha\right|\in[2,m-1]$
and $\left|\alpha\right|_{0}\in[0,m-1]$. Then, using equation $\left(\ref{eq:EpsilonInequality}\right)$
and Remark \ref{rem: Uniform} we get that 
\begin{align*}
n^{2p-1}\sum_{j=0}^{n-1}\mathbb{\tilde{E}}\left[\left|\mathcal{E}_{q,\delta_{j}}\left(\mu^{\tau,m}\left(j\right)\right)\right|^{2p}\right] & \leq n^{2p-1}\sum_{j=0}^{n-1}\delta^{-4pq}\mathbb{\tilde{E}}\left[\left|\mu^{\tau,m}\left(j\right)\right|^{2p\left(2q+1\right)}\right]\\
 & \leq C\left(t,d_{Y},m\right)n\delta^{-\left(2p-1\right)}\delta^{-4pq}\delta^{3p(2q+1)}\\
 & \leq C\left(t,d_{Y},m\right)\delta^{p(2q+1)}.
\end{align*}
\end{proof}
\begin{lem}
\label{lem: CondExpect}Let $\theta=\{\theta_{s}\}_{s\in[0,t]}$ and
$\Psi=\{\Psi_{s}\}_{s\in[0,t]}$ be two $\mathbb{H}^{t}$-adapted
process. Then: 
\begin{enumerate}
\item \label{lem: CE1}For $\alpha\in\mathcal{M}(S_{0})$ and $0\leq s_{1}\leq s_{2}\leq s_{3}\leq t,$
we have that 
\[
\mathbb{\tilde{E}}\left[I_{\alpha}(\theta)_{s_{2},s_{3}}|\mathcal{H}_{s_{1}}^{t}\right]=\boldsymbol{1}_{\{\left\vert \alpha\right\vert =\left\vert \alpha\right\vert _{0}\}}I_{\alpha}\left(\mathbb{\tilde{E}}[\theta|\mathcal{H}_{s_{1}}^{t}]\right)_{s_{2},s_{3}}.
\]
\item \label{lem: CE2}For $\alpha\in\mathcal{M}(S_{0})$ with $\left\vert \alpha\right\vert \neq\left\vert \alpha\right\vert _{0},r\in\{1,...,d_{V}\},0\leq s_{1}\leq s_{2}\leq t$
and $0\leq s_{3}\leq s_{4}\leq t$ we have that 
\begin{align*}
\mathbb{\tilde{E}}\left[\left(\int_{s_{3}}^{s_{4}}\Psi_{s}dV_{s}^{r}\right)I_{\alpha}(\theta)_{s_{1},s_{2}}|\mathcal{Y}_{t}\right] & =\boldsymbol{1}_{\{\alpha_{\left\vert \alpha\right\vert }=0\}}\int_{s_{1}}^{s_{2}}\mathbb{\tilde{E}}\left[\left(\int_{s_{3}}^{s_{4}}\Psi_{s}dV_{s}^{r}\right)I_{\alpha-}(\theta)_{s_{1},u}|\mathcal{Y}_{t}\right]du\\
 & \quad+\boldsymbol{1}_{\{\alpha_{\left\vert \alpha\right\vert }=r\}}\int_{s_{1}\vee s_{3}}^{s_{2}\wedge s_{4}}\mathbb{\tilde{E}}\left[\Psi_{u}I_{\alpha-}(\theta)_{s_{1},u}|\mathcal{Y}_{t}\right]du.
\end{align*}
\item \label{lem: CE3}For $\alpha\in\mathcal{M}(S_{0})$ with $\left\vert \alpha\right\vert \geq2,\alpha_{\left|\alpha\right|}\neq0,\alpha_{\left|\alpha\right|-1}\neq0,r_{1},r_{2}\in\{1,...,d_{V}\},0\leq s_{1}\leq s_{2}\leq t$
and $0\leq s_{3}\leq s_{4}\leq s_{5}\leq s_{6}\leq t$ we have that
\begin{align*}
 & \mathbb{\tilde{E}}\left[\left(\int_{s_{3}}^{s_{6}}\int_{s_{3}}^{s_{5}}\Psi_{s_{4}}dV_{s_{4}}^{r_{1}}dV_{s_{5}}^{r_{2}}\right)I_{\alpha}(\theta)_{s_{1},s_{2}}|\mathcal{Y}_{t}\right]\\
 & =\boldsymbol{1}_{\{\alpha_{\left\vert \alpha\right\vert }=r_{2}\}}\int_{s_{1}\vee s_{3}}^{s_{2}\wedge s_{6}}\mathbb{\tilde{E}}\left[\left(\int_{s_{3}}^{s_{5}}\Psi_{s_{4}}dV_{s_{4}}^{r_{1}}\right)I_{\alpha-}(\theta)_{s_{1},u}|\mathcal{Y}_{t}\right]du\\
 & =\boldsymbol{1}_{\{\alpha_{\left\vert \alpha\right\vert }=r_{2},\alpha_{\left\vert \alpha\right\vert -1}=r_{1}\}}\int_{s_{1}\vee s_{3}}^{s_{2}\wedge s_{6}}\int_{s_{1}\vee s_{3}}^{u}\mathbb{\tilde{E}}\left[\Psi_{v}I_{\left(\alpha-\right)-}(\theta)_{s_{1},v}|\mathcal{Y}_{t}\right]dvdu
\end{align*}
\end{enumerate}
\end{lem}
\begin{proof}
$\left.\right.$\\
 \ref{lem: CE1}. If $\left\vert \alpha\right\vert \neq\left\vert \alpha\right\vert _{0},$
then the iterated integral $I_{\alpha}(\theta)_{s_{2},s_{3}}$ contains
a Brownian differential $dV^{r}$ and it vanishes when we take the
conditional expectation with respect to $\mathcal{H}_{s_{1}}^{t}$.
If $\left\vert \alpha\right\vert =\left\vert \alpha\right\vert _{0},$
all the differentials in the iterated integral $I_{\alpha}(\theta)_{s_{2},s_{3}}$
are Lebesgue differentials and we can write the conditional expectation
inside the inner integral. \\[2mm] \ref{lem: CE2}. Note that if
$\alpha_{\left|\alpha\right|}=0$ we can write 
\[
\mathbb{\tilde{E}}\left[\left(\int_{s_{3}}^{s_{4}}\Psi_{s}dV_{s}^{r}\right)\int_{s_{2}}^{s_{3}}I_{\alpha-}\left(\theta\right)_{s_{1},u}du|\mathcal{Y}_{t}\right]=\mathbb{\tilde{E}}\left[\int_{s_{2}}^{s_{3}}\left(\int_{s_{3}}^{s_{4}}\Psi_{s}dV_{s}^{r}\right)I_{\alpha-}\left(\theta\right)_{s_{2},u}du|\mathcal{Y}_{t}\right],
\]
because we can push the random variable $\left(\int_{s_{3}}^{s_{4}}\Psi_{s}dV_{s}^{r}\right)$
inside the Lebesgue integral. If $\alpha_{\left|\alpha\right|}\neq0$
we can write 
\begin{align*}
 & \mathbb{\tilde{E}}\left[\left(\int_{s_{3}}^{s_{4}}\Psi_{s}dV_{s}^{r}\right)\int_{s_{1}}^{s_{2}}I_{\alpha-}(\theta)_{s_{1},u}dV_{u}^{\alpha_{\left|\alpha\right|}}|\mathcal{Y}_{t}\right]\\
 & =\mathbb{\tilde{E}}\left[\mathbb{\tilde{E}}\left[\left(\int_{s_{3}}^{s_{4}}\Psi_{s}dV_{s}^{r}\right)\int_{s_{1}}^{s_{2}}I_{\alpha-}(\theta)_{s_{1},u}dV_{u}^{\alpha_{\left|\alpha\right|}}|\mathcal{H}_{0}^{t}\right]|\mathcal{Y}_{t}\right],\\
 & =\boldsymbol{1}_{\{\alpha_{\left\vert \alpha\right\vert }=r\}}\int_{s_{1}\vee s_{3}}^{s_{2}\wedge s_{4}}\mathbb{\tilde{E}}\left[\Psi_{u}I_{\alpha-}(\theta)_{s_{1},u}|\mathcal{Y}_{t}\right]du
\end{align*}
where we have just applied the $\mathcal{H}_{s}^{t}$-semimartingale
covariation formula. \\[2mm] \ref{lem: CE3}. Same reasoning as for
statement \ref{lem: CE2}. 
\end{proof}

\subsection{\label{subsec:IntegrabilityLikelihood}Integrability of the likelihood
functional and its discretizations}

In this section we state some integrability results for the likelihood
functional and its discretizations. The first result is on the integrability
of the likelihood functional. It follows from the basic fact that
any Gaussian distribution has exponential moments of all orders. 
\begin{lem}
\label{lem: Z_t^p_Integrability}Assume that $\mathbf{H}\left(1\right)$
holds. Let $p\geq1$ and $\tau$ be any partition. Then, one has that
\[
\mathbb{\tilde{E}}\left[\left|Z_{t}\right|^{p}\right]=\mathbb{\tilde{E}}\left[\exp(p\xi_{t})\right]\leq\mathbb{\tilde{E}}\left[\exp(p\left|\xi_{t}\right|)\right]<\infty,
\]
and 
\[
\mathbb{\tilde{E}}\left[\left|Z_{t}^{\tau,1}\right|^{p}\right]=\mathbb{\tilde{E}}\left[\exp(p\xi_{t}^{\tau,1})\right]<\infty.
\]
\end{lem}
\begin{proof}
We have that 
\begin{align*}
\mathbb{\tilde{E}}\left[\exp(p\left|\xi_{t}\right|)\right] & =\mathbb{\tilde{E}}\left[\exp\left(\left|p\sum_{i=1}^{d_{Y}}\int_{0}^{t}h^{i}(X_{s})dY_{s}^{i}-\frac{p}{2}\sum_{i=1}^{d_{Y}}\int_{0}^{t}h^{i}\left(X_{s}\right)^{2}ds\right|\right)\right]\\
 & \leq\exp\left(\frac{p}{2}d_{Y}\left\Vert h\right\Vert _{\infty}^{2}t\right)\mathbb{\tilde{E}}\left[\exp\left(p\sum_{i=1}^{d_{Y}}\left|\int_{0}^{t}h^{i}(X_{s})dY_{s}^{i}\right|\right)\right].
\end{align*}
Recall that if $Z\thicksim\mathcal{N}\left(0,\sigma^{2}\right)$ under
$\tilde{P}$, then 
\[
\mathbb{\tilde{E}}\left[e^{p\left|Z\right|}\right]=2e^{\frac{p^{2}\sigma^{2}}{2}}\Phi\left(p\sigma\right),
\]
where $\Phi$ is the cumulative distribution function of a standard
normal random variable. As $Y$ is a Brownian motion independent of
$X$ under $\tilde{P}$, we have that 
\begin{align*}
 & \mathbb{\tilde{E}}\left[\exp\left(p\sum_{i=1}^{d_{Y}}\left|\int_{0}^{t}h^{i}(X_{s})dY_{s}^{i}\right|\right)\right]\\
 & =\mathbb{\tilde{E}}\left[\mathbb{\tilde{E}}\left[\exp\left(p\sum_{i=1}^{d_{Y}}\left|\int_{0}^{t}h^{i}(X_{s})dY_{s}^{i}\right|\right)|\mathcal{F}_{t}^{0,V}\right]\right]\\
 & =\mathbb{\tilde{E}}\left[2\exp\left(\frac{p^{2}}{2}\sum_{i=1}^{d_{Y}}\int_{0}^{t}h^{i}(X_{s})^{2}ds\right)\Phi\left(p\left(\sum_{i=1}^{d_{Y}}\int_{0}^{t}h^{i}(X_{s})^{2}ds\right)^{1/2}\right)\right]\\
 & \leq2\exp\left(\frac{p^{2}}{2}d_{Y}\left\Vert h\right\Vert _{\infty}^{2}t\right),
\end{align*}
and we can conclude that $\mathbb{\tilde{E}}\left[\exp(p\left|\xi_{t}\right|)\right]<\infty.$
The proof that $\mathbb{\tilde{E}}\left[\left|Z_{t}^{\tau,1}\right|^{p}\right]<\infty$
follows by similar arguments. 
\end{proof}
The following lemma ensures the $L^{p}\left(\text{\ensuremath{\Omega}}\right)$
integrability of the second order discretization of the likelihood
function, provided the discretization is done on a sufficiently fine
partition. We give a bound on the mesh of the partition in terms of
$p$, the uniform bounds on the sensor function $h$ and its derivatives
and the dimensions of the noise driving the signal and the observation
process. The proof is based on the fact that the square of a centered
Gaussian random variable has finite exponential moment of order sufficiently
small. 
\begin{lem}
\label{lem: Xi_Tau2^p_Integrability}Assume that $\mathbf{H}\left(2\right)$
holds. Let $p\geq1$ and $\tau$ be a partition with mesh size 
\[
\delta<\left(p\left\Vert Lh\right\Vert _{\infty}\sqrt{d_{Y}d_{V}}\right)^{-1},
\]
where 
\[
\left\Vert Lh\right\Vert _{\infty}\triangleq\max_{\substack{i=1,...d_{Y}\\
r=1,...,d_{V}
}
}\left\Vert L^{r}h^{i}\right\Vert _{\infty}.
\]
Then, one has that 
\[
\mathbb{\tilde{E}}\left[\left|Z_{t}^{\tau,2}\right|^{p}\right]=\mathbb{\tilde{E}}\left[\exp(p\xi^{\tau,2})\right]<\infty.
\]
\end{lem}
\begin{proof}
We can write $\exp\left(p\xi^{\tau,2}\right)\triangleq{\displaystyle \prod\limits _{i=1}^{4}}\left(K_{t}^{\tau,2,i}\right)^{p},$
where 
\begin{align*}
K_{t}^{\tau,2,1} & \triangleq\exp\left(\sum_{i=1}^{d_{Y}}\sum_{r=1}^{d_{V}}\int_{0}^{t}L^{r}h^{i}(X_{\tau(s)})(V_{s}^{r}-V_{\tau(s)}^{r})dY_{s}^{i}\right),\\
K_{t}^{\tau,2,2} & \triangleq\exp\left(\sum_{i=1}^{d_{Y}}\int_{0}^{t}\{h^{i}(X_{\tau(s)})+L^{0}h^{i}(X_{\tau(s)})(s-\tau(s))\}dY_{s}^{i}\right),\\
K_{t}^{\tau,2,3} & \triangleq\exp\left(-\frac{1}{2}\sum_{i=1}^{d_{Y}}\int_{0}^{t}\{(h^{i})^{2}\left(X_{\tau(s)}\right)+L^{0}((h^{i})^{2})(X_{\tau(s)})(s-\tau(s))\}ds\right),\\
K_{t}^{\tau,2,4} & \triangleq\exp\left(-\frac{1}{2}\sum_{i=1}^{d_{Y}}\sum_{r=1}^{d_{V}}\int_{0}^{t}L^{r}((h^{i})^{2})(X_{\tau(s)})(V_{s}^{r}-V_{\tau(s)}^{r})\}ds\right).
\end{align*}
Let $\varepsilon>0,$ then, by Hölder inequality, we have 
\[
\mathbb{\tilde{E}}\left[\exp\left(p\xi^{\tau,2}\right)\right]\leq\mathbb{\tilde{E}}\left[\left\vert K_{t}^{\tau,2,1}\right\vert ^{p(1+\varepsilon)}\right]^{\frac{1}{1+\varepsilon}}\mathbb{\tilde{E}}\left[{\displaystyle \prod\limits _{i=2}^{4}}\left\vert K_{t}^{\tau,2,i}\right\vert ^{p\frac{(1+\varepsilon)}{\varepsilon}}\right]^{\frac{\varepsilon}{1+\varepsilon}}.
\]
Hence, the result follows by showing that $K_{t}^{\tau,2,1}$ has
finite $p(1+\varepsilon)$-moment and 
\begin{equation}
\mathbb{\tilde{E}}\left[{\displaystyle \prod\limits _{i=2}^{4}}\left\vert K_{t}^{\tau,2,i}\right\vert ^{p\frac{(1+\varepsilon)}{\varepsilon}}\right]<\infty.\label{eq: ProductK}
\end{equation}
Applying Hölder inequality twice, condition $\left(\ref{eq: ProductK}\right)$
follows by showing that $K_{t}^{\tau,2,i},i=2,...,4$ have finite
moments of all orders. In what follows, let $q\geq1$ be a fixed real
constant. We start by the easiest term, $K_{t}^{\tau,2,3}.$ We have
that 
\[
\mathbb{\tilde{E}}\left[\left\vert K_{t}^{\tau,2,3}\right\vert ^{q}\right]\leq\exp\left(\frac{qd_{Y}}{2}t(\left\Vert h\right\Vert _{\infty}^{2}+\delta\left\Vert L^{0}h^{2}\right\Vert _{\infty}\right)<\infty,
\]
because $\left\Vert h\right\Vert _{\infty}^{2}$ and $\left\Vert L^{0}h^{2}\right\Vert _{\infty}=\max_{i=1,...,d_{Y}}\left\Vert L^{0}(h_{i}^{2})\right\Vert _{\infty}$
are finite due to the assumptions on $f,\sigma$ and $h.$ For the
term $K_{t}^{\tau,2,4},$ we can write 
\begin{align*}
\mathbb{\tilde{E}}\left[\left\vert K_{t}^{\tau,2,4}\right\vert ^{q}\right] & \leq\mathbb{\tilde{E}}\left[\exp\left(\frac{qd_{Y}d_{V}}{2}\left\Vert L((h)^{2})\right\Vert _{\infty}\int_{0}^{t}\left\vert V_{s}^{1}-V_{\tau(s)}^{1}\right\vert ds\right)\right]\\
 & =\mathbb{\tilde{E}}\left[\exp\left(\frac{qd_{Y}d_{V}}{2}\left\Vert L((h)^{2})\right\Vert _{\infty}\left(\int_{0}^{t}(s-\tau(s))ds\right)\left\vert V_{1}^{1}\right\vert \right)\right]\\
 & \leq\mathbb{\tilde{E}}\left[\exp\left(\frac{qd_{Y}d_{V}}{2}\left\Vert L((h)^{2})\right\Vert _{\infty}t\sqrt{\delta}\left\vert V_{1}^{1}\right\vert \right)\right]<\infty,
\end{align*}
because $\left\Vert L((h)^{2})\right\Vert _{\infty}=\max_{\substack{i=1,...d_{Y}\\
r=1,...,d_{V}
}
}\left\Vert L^{r}(h_{i}^{2})\right\Vert _{\infty}$ is finite, the law of $V_{s}^{1}-V_{\tau(s)}^{1}$ coincides with
the law of $(s-\tau(s))^{1/2}V_{1}^{1}$ by the scaling properties
of the Brownian motion and $\left\vert V_{1}^{1}\right\vert $ has
exponential moments of any order.

For the term $K_{t}^{\tau,2,2},\ $we first condition with respect
to $\mathcal{F}_{t}^{V}=\sigma(V_{s},0\leq s\leq t)$ and use the
fact that, conditionally to $\mathcal{F}_{t}^{V}$, the stochastic
integrals with respect to $Y$ are Gaussian. We get 
\begin{align*}
 & \mathbb{\tilde{E}}\left[\left\vert K_{t}^{\tau,2,2}\right\vert ^{q}\right]=\mathbb{\tilde{E}}\left[\mathbb{\tilde{E}}\left[\exp\left(q\sum_{i=1}^{d_{Y}}\int_{0}^{t}\{h^{i}(X_{\tau(s)})+L^{0}h^{i}(X_{\tau(s)})(s-\tau(s))\}dY_{s}^{i}\right)|\mathcal{F}_{t}^{V}\right]\right]\\
 & =\mathbb{\tilde{E}}\left[\exp\left(\frac{q^{2}}{2}\sum_{i=1}^{d_{Y}}\int_{0}^{t}\left\{ h^{i}(X_{\tau(s)})+L^{0}h^{i}(X_{\tau(s)})(s-\tau(s))\right\} ^{2}ds\right)\right]\\
 & =\exp(q^{2}d_{Y}t\{\left\Vert h\right\Vert _{\infty}^{2}+\left\Vert L^{0}h\right\Vert _{\infty}^{2}\})<\infty.
\end{align*}

Finally, the term $K_{t}^{\tau,2,1}$ is more delicate because, in
order to show that has finite $(p+\varepsilon)$-moment, a relationship
between the mesh of the partition $\delta$ and $p+\varepsilon$ is
needed. Proceeding as with the term $K_{t}^{\tau,2,2},$ we obtain
\begin{align*}
\mathbb{\tilde{E}}\left[\left\vert K_{t}^{\tau,2,1}\right\vert ^{p(1+\varepsilon)}\right] & =\mathbb{\tilde{E}}\left[\exp\left(p(1+\varepsilon)\sum_{i=1}^{d_{Y}}\sum_{r=1}^{d_{V}}\int_{0}^{t}L^{r}h^{i}(X_{\tau(s)})(V_{s}^{r}-V_{\tau(s)}^{r})dY_{s}^{i}\right)\right]\\
 & =\mathbb{\tilde{E}}\left[\prod_{i=1}^{d_{Y}}\mathbb{\tilde{E}}\left[\exp\left(\int_{0}^{t}p(1+\varepsilon)\sum_{r=1}^{d_{V}}L^{r}h^{i}(X_{\tau(s)})(V_{s}^{r}-V_{\tau(s)}^{r})dY_{s}^{i}\right)|\mathcal{F}_{t}^{V}\right]\right].
\end{align*}
Now, conditionally to $\mathcal{F}_{t}^{V},$ the terms in the exponential
are centered Gaussian random variables and we get that 
\begin{align*}
\mathbb{\tilde{E}}\left[\left\vert K_{t}^{\tau,2,1}\right\vert ^{p(1+\varepsilon)}\right] & =\mathbb{\tilde{E}}\left[\prod_{i=1}^{d_{Y}}\exp\left(\frac{p^{2}(1+\varepsilon)^{2}}{2}\int_{0}^{t}\left(\sum_{r=1}^{d_{V}}L^{r}h^{i}(X_{\tau(s)})(V_{s}^{r}-V_{\tau(s)}^{r})\right)^{2}ds\right)\right]\\
 & \leq\mathbb{\tilde{E}}\left[\prod_{i=1}^{d_{Y}}\exp\left(\frac{p^{2}(1+\varepsilon)^{2}d_{V}}{2}\int_{0}^{t}\left(\sum_{r=1}^{d_{V}}|L^{r}h^{i}(X_{\tau(s)})|^{2}(V_{s}^{r}-V_{\tau(s)}^{r})^{2}\right)ds\right)\right]\\
 & =\mathbb{\tilde{E}}\left[\exp\left(\frac{p^{2}(1+\varepsilon)^{2}d_{Y}d_{V}\left\Vert Lh\right\Vert _{\infty}^{2}}{2}\sum_{r=1}^{d_{V}}\int_{0}^{t}(V_{s}^{r}-V_{\tau(s)}^{r})^{2}ds\right)\right]\\
 & =\mathbb{\tilde{E}}\left[\exp\left(\frac{p^{2}(1+\varepsilon)^{2}d_{Y}d_{V}\left\Vert Lh\right\Vert _{\infty}^{2}}{2}\int_{0}^{t}(V_{s}^{1}-V_{\tau(s)}^{1})^{2}ds\right)\right]^{d_{V}}.
\end{align*}
So we need to find conditions on $\beta>0,$ such that $\mathbb{\tilde{E}}\left[\exp\left(\beta\int_{0}^{t}(V_{s}^{1}-V_{\tau(s)}^{1})^{2}ds\right)\right]<\infty.$
We can write 
\begin{align*}
\mathbb{\tilde{E}}\left[\exp\left(\beta\int_{0}^{t}(V_{s}^{1}-V_{\tau(s)}^{1})^{2}ds\right)\right] & =\mathbb{\tilde{E}}\left[\exp\left(\beta\sum_{j=1}^{n}\int_{t_{j-1}}^{t_{j}}(V_{s}^{1}-V_{t_{j-1}}^{1})^{2}ds\right)\right]\\
 & =\prod_{j=1}^{n}\mathbb{\tilde{E}}\left[\exp\left(\beta\int_{t_{j-1}}^{t_{j}}(V_{s}^{1}-V_{t_{j-1}}^{1})^{2}ds\right)\right]\\
 & \triangleq\prod_{j=1}^{n}\Theta\left(\beta,\delta_{j}\right).
\end{align*}
Denote by $M_{t}\triangleq\sup_{0\leq s\leq t}V_{s}^{1}$ and recall
that the density of $M_{t}$ is given by 
\begin{align*}
f_{M_{t}}\left(x\right) & =\frac{2}{\sqrt{2\pi t}}e^{-\frac{x^{2}}{2t}}\boldsymbol{1}_{(0,\infty)},
\end{align*}
see Karatzas and Shreve \cite{KS91}, page 96. Moreover, note that
for any $A>0,$ 
\[
\frac{2}{\sqrt{2\pi\sigma^{2}}}\int_{0}^{\infty}\exp\left\{ -A\frac{x^{2}}{2\sigma^{2}}\right\} dx=A^{-1/2}.
\]
Then, we have that 
\begin{align*}
\Theta\left(\beta,\delta_{j}\right) & \leq\mathbb{\tilde{E}}[\exp(\beta\delta_{j}M_{\delta}^{2}]=\int_{0}^{\infty}\frac{2}{\sqrt{2\pi\delta_{j}}}\exp\left\{ \beta\delta_{j}x^{2}-\frac{x^{2}}{2\delta_{j}}\right\} \\
 & =\int_{0}^{\infty}\frac{2}{\sqrt{2\pi\delta_{j}}}\exp\left\{ -\left(1-2\beta\delta_{j}^{2}\right)\frac{x^{2}}{2\delta_{j}}\right\} =\left(1-2\beta\delta_{j}^{2}\right)^{-1/2}<\infty,
\end{align*}
as long as $1-2\beta\delta_{j}^{2}>0.$ On the other hand, 
\begin{align*}
\left(1-2\beta\delta_{j}^{2}\right)^{-1} & =\sum_{k=0}^{\infty}\left(2\beta\delta_{j}^{2}\right)^{k}=1+2\beta\delta_{j}^{2}\left(\sum_{k=0}^{\infty}\left(2\beta\delta_{j}^{2}\right)^{k}\right)\\
 & \leq1+2\beta\delta_{j}^{2}\left(\sum_{k=0}^{\infty}\left(2\beta\delta^{2}\right)^{k}\right)=1+\frac{2\beta\delta_{j}^{2}}{1-2\beta\delta^{2}}\\
 & \leq\exp\left(\frac{2\beta\delta_{j}^{2}}{1-2\beta\delta^{2}}\right),
\end{align*}
and, therefore, 
\begin{align*}
\prod_{j=1}^{n}\Theta\left(\beta,\delta_{j}\right) & \leq\prod_{j=1}^{n}\exp\left(\frac{\beta\delta_{j}^{2}}{1-2\beta\delta^{2}}\right)\leq\exp\left(\frac{\beta\sum_{j=1}^{n}\delta_{j}^{2}}{1-2\beta\delta^{2}}\right)\\
 & \leq\exp\left(\frac{\beta\delta t}{1-2\beta\delta^{2}}\right)<\infty.
\end{align*}

As $\beta=\frac{p^{2}(1+\varepsilon)^{2}d_{Y}d_{V}\left\Vert Lh\right\Vert _{\infty}^{2}}{2}$
and $\varepsilon>0$ can be made arbitrary small we get the following
condition for the partition mesh $\delta<\left(p\left\Vert Lh\right\Vert _{\infty}\sqrt{d_{Y}d_{V}}\right)^{-1}.$ 
\end{proof}
We complete the section with an application of the previous two lemmas.
Note that, in order to control the high order discretizations of the
likelihood function, we reduce the problem to the control of the second
order discretization via the truncation procedure as described in
Remark \ref{rem:TruncProc}. 
\begin{cor}
\label{cor: BoundsFi}Let $\varphi\in\mathcal{B}_{P}$. One has that: 
\end{cor}
\begin{enumerate}
\item \label{enu: condL}If $\mathbf{H}\left(1\right)$ holds, then there
exists $\varepsilon>0$ such that 
\begin{equation}
\mathbb{\tilde{E}}\left[\left\vert \varphi(X_{t})e^{\xi_{t}}\right\vert ^{2+\varepsilon}\right]<\infty,\label{eq: CondL}
\end{equation}
and 
\begin{equation}
\sup_{\tau\in\Pi(t)}\mathbb{\tilde{E}}\left[\left\vert \varphi(X_{t})e^{\xi_{t}^{\tau,1}}\right\vert ^{2+\varepsilon}\right]<\infty.\label{eq: CondL1Pi}
\end{equation}
\item \label{enu: condL2}If $\mathbf{H}\left(2\right)$ holds, then there
exists $\varepsilon>0$ and $\delta_{0}=\delta_{0}\left(h,f,\sigma,\right)>0$
such that 
\begin{equation}
\sup_{\tau\in\Pi(t,\delta_{0})}\mathbb{\tilde{E}}\left[\left\vert \varphi(X_{t})e^{\xi_{t}^{\tau,2}}\right\vert ^{2+\varepsilon}\right]<\infty.\label{eq: CondL2Pi}
\end{equation}
\item \label{enu: condLm}If $\mathbf{H}\left(m\right)$ with $m\geq3$
holds, then there exists $\varepsilon>0$ and $\delta_{0}>0$ such
that 
\begin{equation}
\sup_{\tau\in\Pi(t,\delta_{0})}\mathbb{\tilde{E}}\left[\left\vert \varphi(X_{t})e^{\bar{\xi}_{t}^{\tau,m}}\right\vert ^{2+\varepsilon}\right]<\infty.\label{eq: CondLmPi}
\end{equation}
\end{enumerate}
\begin{proof}
Combining Lemmas \ref{lem: Z_t^p_Integrability} and \ref{lem: Xi_Tau2^p_Integrability}
with Hölder inequality and Remark \ref{rem: M} we obtain $\left(\ref{eq: CondL}\right)$,
$\left(\ref{eq: CondL1Pi}\right)$ and $\left(\ref{eq: CondL2Pi}\right)$.
Moreover, for $m\geq3$, note that 
\begin{align*}
\bar{\xi}_{t}^{\tau,m} & =\xi_{t}^{\tau,2}+\sum_{j=0}^{n-1}\Gamma_{m-\frac{1}{2},\delta_{j}}\left(\mu^{\tau,m}\left(j\right)\right)\\
 & \leq\xi_{t}^{\tau,2}+\sum_{j=0}^{n-1}\left|\Gamma_{m-\frac{1}{2},\delta_{j}}\left(\mu^{\tau,m}\left(j\right)\right)\right|\\
 & \leq\xi_{t}^{\tau,2}+\sum_{j=0}^{n-1}\delta_{j}=\xi_{t}^{\tau,2}+t,
\end{align*}
and (\ref{eq: CondLmPi}) follows from (\ref{eq: CondL2Pi}). 
\end{proof}

\subsection{Proof of the Theorem \protect\ref{thm: Main Filtering_2}\label{sec: Proof Main Theorem}}

In this section we prove the main theorem of the paper. We start by
stating and proving two main propositions. 
\begin{prop}
\label{prop: Main1}Let $m\in\mathbb{N}$ and assume that condition
\textbf{H}$(m)$ holds and $\varphi\in C_{P}^{m+1}$. Then, there
exists a constant $C$ independent of the partition $\pi\in\Pi$ such
that 
\[
\mathbb{\tilde{E}}\left[\left\vert \mathbb{\tilde{E}}\left[(\xi_{t}-\xi_{t}^{\tau,m})\varphi(X_{t})e^{\xi_{t}}|\mathcal{Y}_{t}\right]\right\vert ^{2}\right]\leq C\delta^{2m}.
\]
\end{prop}
\begin{proof}
By Lemma \ref{lem: Difference Moment Estimate} we can write 
\[
(\xi_{t}-\xi_{t}^{\tau,m})\varphi(X_{t})e^{\xi_{t}}=\varphi(X_{t})e^{\xi_{t}}\left(\sum_{i=0}^{d_{Y}}\int_{0}^{t}\left\{ \sum_{\alpha\in\mathcal{R}(\mathcal{M}_{m-1}(S_{0}))}I_{\alpha}(L^{\alpha}h_{i}(X_{\cdot}))_{\tau(s),s}\right\} dY_{s}^{i}\right).
\]
For $i=0$, the result follows from Lemmas \ref{lem: Alpha0=00003D00003Dm}
and \ref{lem: ds}. Recall that 
\[
\mathcal{R}\left(\mathcal{M}_{m-1}(S_{0})\right)_{k}=\{\alpha\in\mathcal{R}\left(\mathcal{M}_{m-1}(S_{0})\right):\left\vert \alpha\right\vert _{0}=k\},k=0,..,m,
\]
that is, $\mathcal{R}\left(\mathcal{M}_{m-1}(S_{0})\right)_{k}$ is
the set of multi-indices in $\mathcal{R}\left(\mathcal{M}_{m-1}(S_{0})\right)$
that contain $k$ zeros. This collection of sets are obviously a disjoint
partition of $\mathcal{R}\left(\mathcal{M}_{m-1}(S_{0})\right),$
that is, $\mathcal{R}\left(\mathcal{M}_{m-1}(S_{0})\right)={\displaystyle \biguplus\limits _{k=0}^{m}}\mathcal{R}\left(\mathcal{M}_{m-1}(S_{0})\right)_{k}.$
For $i\neq0$, we will divide the proof of the theorem in cases, depending
on $\alpha$ belonging to one of these subsets. The cases for $m\in\left\{ 1,2\right\} $
are:\\
 $\bullet$ $m=1,\alpha\in\mathcal{R}\left(\mathcal{M}_{0}(S_{0})\right)_{0}$:
Lemma \ref{lem: dY_m=00003D00003D1}. \\
 $\bullet$ $m=1,\alpha\in\mathcal{R}\left(\mathcal{M}_{0}(S_{0})\right)_{1}$:
Lemma \ref{lem: Alpha0=00003D00003Dm}. \\
 $\bullet$ $m=2,\alpha\in\mathcal{R}\left(\mathcal{M}_{1}(S_{0})\right)_{0}$:
Lemma \ref{lem: dY_m=00003D00003D2_alpha=00003D00003D1}. \\
 $\bullet$ $m=2,\alpha\in\mathcal{R}\left(\mathcal{M}_{1}(S_{0})\right)_{1}$:
Lemma \ref{lem: dY_m=00003D00003D2_alpha=00003D00003D2}. \\
 $\bullet$ $m=2,\alpha\in\mathcal{R}\left(\mathcal{M}_{1}(S_{0})\right)_{2}$:
Lemma \ref{lem: Alpha0=00003D00003Dm}.

For arbitrary $m$$>2$, the proof follows the same ideas as for $m\in\{1,2\}.$
In the case that $\alpha\in\mathcal{R}\left(\mathcal{M}_{m-1}(S_{0})\right)_{m}$
the result follows from applying Lemma \ref{lem: Alpha0=00003D00003Dm}.
For $\mathcal{R}\left(\mathcal{M}_{m-1}(S_{0})\right)_{k}$ with $k\in\{0,m-1\}$,
first one needs to use the truncated Stroock-Taylor formula of order
$k$ to express $\varphi(X_{t})e^{\xi_{t}}$ as a sum of iterated
integrals with respect to the Brownian motion. The goal is to use
the covariance between the iterated integrals in the Stroock-Taylor
expansion of $\varphi(X_{t})e^{\xi_{t}}$ and $I_{\alpha}(L^{\alpha}h_{i}(X_{\cdot}))_{\tau(s),s}$
in order to generate the right order of convergence in $\delta$.
However, this is not straightforward due to the presence of the stochastic
integral with respect to Y. The process Y as an integrator makes impossible
to use directly an integration by parts formula because the two iterated
integrals are semimartingales with respect to different filtrations.
To overcome this difficulty, the idea is to compute this covariance
along a partition. We use an integration by parts formula, in each
subinterval and only to the integral with respect to Y, to obtain
\[
\int_{t_{j}}^{t_{j+1}}I_{\alpha}(L^{\alpha}h_{i}(X_{\cdot}))_{\tau(s),s}dY_{s}^{i}=\int_{t_{j}}^{t_{j+1}}\left(Y_{t_{j+1}}^{i}-Y_{s}^{i}\right)I_{\alpha-}(L^{\alpha}h_{i}(X_{\cdot}))_{\tau(s),s}dV^{\alpha_{|\alpha|}}.
\]
The term on the right hand side in the last expression is an $\mathbb{H}^{t}$-semimartingale
and we can compute its covariation with the terms in the Stroock-Taylor
expansion of $\varphi(X_{t})e^{\xi_{t}}$, see Lemmas \ref{lem: dY_m=00003D00003D1},
\ref{lem: dY_m=00003D00003D2_alpha=00003D00003D1} and \ref{lem: dY_m=00003D00003D2_alpha=00003D00003D2}. 
\end{proof}
\begin{prop}
\label{prop: Main2}Let $m\in\mathbb{N}$ and assume that condition
\textbf{H}$(m)$ holds and $\varphi\in\mathcal{B}_{P}$. Then, there
exist $\delta_{0}>0$ and constant $C$ independent of any partition
$\pi\in\Pi\left(t,\delta_{0}\right)$ such that 
\[
\mathbb{\tilde{E}}\left[\left\vert \left\vert \varphi(X_{t})\right\vert \left(e^{\xi_{t}}+e^{\bar{\xi}_{t}^{\tau,m}}\right)(\xi_{t}-\bar{\xi}_{t}^{\tau,m})^{2}\right\vert ^{2}\right]\leq C\delta^{2m}.
\]
\end{prop}
\begin{proof}
As \textbf{H}$(m)$ holds, let $\varepsilon>0$ such as in Corollary
\ref{cor: BoundsFi}. By Hölder inequality we have that 
\begin{align*}
 & \mathbb{\tilde{E}}\left[\left\vert \left\vert \varphi(X_{t})\right\vert \left(e^{\xi_{t}}+e^{\bar{\xi}_{t}^{\tau,m}}\right)(\xi_{t}-\bar{\xi}_{t}^{\tau,m})^{2}\right\vert ^{2}\right]\\
 & \leq\mathbb{\tilde{E}}\left[\left\vert \left\vert \varphi(X_{t})\right\vert \left(e^{\xi_{t}}+e^{\bar{\xi}_{t}^{\tau,m}}\right)\right\vert ^{2+\varepsilon}\right]^{\frac{2}{2+\varepsilon}}\mathbb{\tilde{E}}\left[\left\vert \xi_{t}-\bar{\xi}_{t}^{\tau,m}\right\vert ^{4\frac{2+\varepsilon}{\varepsilon}}\right]^{\frac{\varepsilon}{2+\varepsilon}},
\end{align*}
Corollary \ref{cor: BoundsFi} yields that there exists $\delta_{0}>0$
such that 
\begin{equation}
\sup_{\tau\in\Pi\left(t,\delta_{0}\right)}\mathbb{\tilde{E}}\left[\left\vert \left\vert \varphi(X_{t})\right\vert \left(e^{\xi_{t}}+e^{\bar{\xi}_{t}^{\tau,m}}\right)\right\vert ^{2+\varepsilon}\right]<\infty.\label{eq: BoundAuxiliary}
\end{equation}
On the other hand, by equation $\left(\ref{eq:BigEpsilon}\right)$,
for any $p\geq1$ we get that 
\begin{align*}
\left\vert \xi_{t}-\bar{\xi}_{t}^{\tau,m}\right\vert ^{2p} & \leq C\left\{ \left|\xi_{t}-\xi_{t}^{\tau,m}\right|^{2p}+\left|\sum_{j=0}^{n-1}\mathcal{E}_{m-\frac{1}{2}}\left(\mu^{\tau,m}\left(j\right)\right)\right|^{2p}\right\} .
\end{align*}
By Lemma \ref{lem: Difference Moment Estimate}, we obtain 
\[
\mathbb{\tilde{E}}\left[\left|\xi_{t}-\xi_{t}^{\tau,m}\right|^{2p}\right]\leq C\delta^{pm},
\]
and by Lemma \ref{lem: SumEpsilon} with $q=m-\frac{1}{2}$ we have
that 
\[
\mathbb{\tilde{E}}\left[\left\vert \sum_{j=0}^{n-1}\mathcal{E}_{m-\frac{1}{2},\delta_{j}}\left(\mu^{\tau,m}\left(j\right)\right)\right\vert ^{2p}\right]\leq C\left(t,d_{Y},p,m\right)\delta^{2pm}
\]
Hence, setting $p=2\left(2+\varepsilon\right)/\varepsilon$, we obtain
\[
\mathbb{\tilde{E}}\left[\left\vert \xi_{t}-\bar{\xi}_{t}^{\tau,m}\right\vert ^{4\left(2+\varepsilon\right)/\varepsilon}\right]^{\varepsilon/\left(2+\varepsilon\right)}\leq C\delta^{2m}.
\]
\end{proof}
We are finally ready to put everything together and deduce Theorem
\ref{thm: Main Filtering_2}. 
\begin{proof}[\textbf{Proof of Theorem \ref{thm: Main Filtering_2}}]

To get the desired rate of convergence for the unnormalised conditional
distribution $\rho_{t}^{\tau,m},$ we can write 
\begin{align*}
\rho_{t}\left(\varphi\right)-\rho_{t}^{\tau,m}\left(\varphi\right) & =\mathbb{\tilde{E}}[\varphi(X_{t})(\xi_{t}-\bar{\xi}_{t}^{\tau,m})e^{\xi_{t}}|\mathcal{Y}_{t}]\\
 & +\mathbb{\tilde{E}}\left[\varphi(X_{t})e^{\xi_{t}}-\varphi(X_{t})e^{\bar{\xi}_{t}^{\tau,m}}-\varphi(X_{t})(\xi_{t}-\bar{\xi}_{t}^{\tau,m})e^{\xi_{t}}|\mathcal{Y}_{t}\right].
\end{align*}
Using the inequality 
\[
\left\vert e^{x}-e^{y}-(x-y)e^{x}\right\vert \leq\frac{e^{x}+e^{y}}{2}(x-y)^{2},
\]
we get that 
\begin{align*}
 & \mathbb{\tilde{E}}\left[\left\vert \rho_{t}\left(\varphi\right)-\rho_{t}^{\tau,m}\left(\varphi\right)\right\vert ^{2}\right]\\
 & \leq C\left\{ \mathbb{\tilde{E}}\left[\left\vert \mathbb{\tilde{E}}\left[(\xi_{t}-\bar{\xi}_{t}^{\tau,m})\varphi(X_{t})e^{\xi_{t}}|\mathcal{Y}_{t}\right]\right\vert ^{2}\right]+\mathbb{\tilde{E}}\left[\left\vert \left\vert \varphi(X_{t})\right\vert \frac{e^{\xi_{t}}+e^{\bar{\xi}_{t}^{\tau,m}}}{2}(\xi_{t}-\bar{\xi}_{t}^{\tau,m})^{2}\right\vert ^{2}\right]\right\} .
\end{align*}
Now, Propositions \ref{prop: Main1} and \ref{prop: Main2} yield
\[
\mathbb{\tilde{E}}\left[\left\vert \rho_{t}\left(\varphi\right)-\rho_{t}^{\tau,m}\left(\varphi\right)\right\vert ^{2}\right]\leq C\delta^{2m}.
\]

To prove the rate for the normalised conditional distribution observe
that we can write 
\[
\pi_{t}^{\tau,m}\left(\varphi\right)-\pi_{t}\left(\varphi\right)=\frac{1}{\rho_{t}\left(\boldsymbol{1}\right)}\frac{\rho_{t}^{\tau,m}\left(\varphi\right)}{\rho_{t}^{\tau,m}\left(\boldsymbol{1}\right)}\left(\rho_{t}\left(\boldsymbol{1}\right)-\rho_{t}^{\tau,m}\left(\boldsymbol{1}\right)\right)+\frac{1}{\rho_{t}\left(\boldsymbol{1}\right)}\left(\rho_{t}^{\tau,m}\left(\varphi\right)-\rho_{t}\left(\varphi\right)\right).
\]
Hence, 
\begin{align*}
 & \mathbb{E}\left[\left\vert \pi_{t}\left(\varphi\right)-\pi_{t}^{\tau,m}\left(\varphi\right)\right\vert \right]\\
 & \leq C\mathbb{\tilde{E}}\left[\frac{Z_{t}}{\left\vert \rho_{t}\left(\boldsymbol{1}\right)\right\vert }\left\{ \left\vert \pi_{t}^{\tau,m}\left(\varphi\right)\right\vert \left\vert \rho_{t}\left(\boldsymbol{1}\right)-\rho_{t}^{\tau,m}\left(\boldsymbol{1}\right)\right\vert +\left\vert \rho_{t}^{\tau,m}\left(\varphi\right)-\rho_{t}\left(\varphi\right)\right\vert \right\} \right]\\
 & =C\mathbb{\tilde{E}}\left[\frac{\mathbb{\tilde{E}}[Z_{t}|\mathcal{Y}_{t}]}{\left\vert \rho_{t}\left(\boldsymbol{1}\right)\right\vert }\left\{ \left\vert \pi_{t}^{\tau,m}\left(\varphi\right)\right\vert \left\vert \rho_{t}\left(\boldsymbol{1}\right)-\rho_{t}^{\tau,m}\left(\boldsymbol{1}\right)\right\vert +\left\vert \rho_{t}^{\tau,m}\left(\varphi\right)-\rho_{t}\left(\varphi\right)\right\vert \right\} \right]\\
 & \leq C\left\{ \mathbb{\tilde{E}}\left[\left\vert \pi_{t}^{\tau,m}\left(\varphi\right)\right\vert \left\vert \rho_{t}\left(\boldsymbol{1}\right)-\rho_{t}^{\tau,m}\left(\boldsymbol{1}\right)\right\vert \right]+\mathbb{\tilde{E}}\left[\left\vert \rho_{t}^{\tau,m}\left(\varphi\right)-\rho_{t}\left(\varphi\right)\right\vert \right]\right\} .\\
 & \leq C\left\{ \mathbb{\tilde{E}}\left[\left\vert \pi_{t}^{\tau,m}\left(\varphi\right)\right\vert ^{2}\right]^{1/2}\mathbb{\tilde{E}}\left[\left\vert \rho_{t}\left(\boldsymbol{1}\right)-\rho_{t}^{\tau,m}\left(\boldsymbol{1}\right)\right\vert ^{2}\right]^{1/2}+\mathbb{\tilde{E}}\left[\left\vert \rho_{t}^{\tau,m}\left(\varphi\right)-\rho_{t}\left(\varphi\right)\right\vert ^{2}\right]^{1/2}\right\} ,
\end{align*}
where in the last inequality we have applied Hölder inequality. Combining
the bounds for the unnormalised distribution and the hypothesis on
$\pi_{t}^{\tau,m}\left(\varphi\right)$ we can conclude. 
\end{proof}

\section{\label{sec:Technical-Lemmas}Technical Lemmas}

We collate in this section the technical lemmas required to prove
the main results. We begin with some limited background material on
Malliavin Calculus (and partial Maliavin Calculus) with a view to
deduce the necessary properties of the functionals to be discretised. 

\subsection{\label{subsec:Malliavin-calculus}Malliavin calculus}

Let $B$=$\left\{ B_{t}\right\} _{t\in\left[0,T\right]}$ be a $d$-dimensional
standard Brownian motion defined on a complete probability space $\left(\Omega,\mathcal{F},P\right).$
Let $\mathcal{S}$ denote the class of smooth random variables such
that a random variable $F$$\in\mathcal{S}$ has the form 
\[
F=f\left(B_{t_{1}};...;B_{t_{n}}\right),
\]
where the function $f\left(x^{11},...,x^{d1};...;x^{1n},...,x^{dn}\right)$
belongs to $C_{b}^{\infty}\left(\mathbb{R}^{dn}\right)$ and $t_{1},...,t_{n}\in\left[0,T\right]$.
The Malliavin derivative of a smooth functional $F$ can be defined
as the $d$-dimensional stochastic processes given by 
\[
\left(DF\right)_{t}^{j}=\sum_{i=1}^{n}\frac{\partial f}{\partial x^{ji}}\left(B_{t_{1}};...;B_{t_{n}}\right)\mathbf{1}_{\left[0,t_{i}\right]}\left(t\right),
\]
for $t\in\left[0,T\right]$ and $j=1,...,d$. The derivative $DF$
can be regarded as a random variable taking values in the Hilbert
space $H=L^{2}\left(\left[0,T\right];\mathbb{R}^{d}\right)$. Noting
the isometry $L^{2}\left(\Omega\times\left[0,T\right];\mathbb{R}^{d}\right)\simeq L^{2}\left(\Omega;H\right)$
we can identify $\left(DF\right)_{t}^{j}$ as the value at time $t$
of the $j$th component of and $\mathbb{R}$$^{d}-$valued stochastic
process. We will also the notation $D_{t}^{j}F$ for $\left(DF\right)_{t}^{j}$.
One can see that the operator $D$ is closable from $L^{p}\left(\Omega\right)$
to $L$$^{p}\left(\Omega;H\right)$, $p\geq1$ and we will denote
the domain of $D$ in $L$$^{p}\left(\Omega\right)$ by $\mathbb{D}^{1,p}.$
That is, meaning that $\mathbb{D}^{1,p}$ is the closure of smooth
random variables $\mathcal{S}$ with respect to the norm 
\[
\left\Vert F\right\Vert _{\mathbb{D}^{1,p}}=\left(\mathbb{E}\left[\left|F\right|^{p}\right]+\mathbb{E}\left[\left\Vert DF\right\Vert _{H}^{p}\right]\right)^{1/p}.
\]
We define the $k$-th derivative of $F$, $D$$^{k}F$, as the $H$$^{\otimes k}$-valued
random variable 
\[
\left(D^{k}F\right)_{s_{1},...,s_{k}}^{j_{1},...,j_{k}}=\sum_{i_{1},...,i_{k}=1}^{n}\frac{\partial^{k}f}{\partial x^{j_{1}i_{1}}\cdots\partial x^{j_{k}i_{k}}}\left(B_{t_{1}};...;B_{t_{n}}\right)\mathbf{1}_{\left[0,t_{i_{1}}\right]}\left(s_{1}\right)\cdots\mathbf{1}_{\left[0,t_{i_{k}}\right]}\left(s_{k}\right),
\]
where $s_{1},...,s_{k}$$\in\left[0,T\right]$ and $j_{1},...,j_{k}=1,...,d$.
We will also write $D_{s_{1},...,s_{k}}^{j_{1},...,j_{k}}F$ for $\left(D^{k}F\right)_{s_{1},...,s_{k}}^{j_{1},...,j_{k}}$
and notice that it coincides with the iterated derivative $D_{t_{1}}^{j_{1}}\cdots D_{t_{k}}^{j_{k}}F$.
For any integer $k$$\geq1$ and any real number $p>1$ we introduce
the norm on $\mathcal{S}$ given by 
\[
\left\Vert F\right\Vert _{k,p}=\left(\mathbb{E}\left[\left|F\right|^{p}\right]+\sum_{j=1}^{k}\mathbb{E}\left[\left\Vert D^{j}F\right\Vert _{H^{\otimes j}}^{p}\right]\right)^{1/p},
\]
where 
\[
\left\Vert D^{k}F\right\Vert _{H^{\otimes k}}=\left(\sum_{j_{1}\cdots j_{k}=1}^{d}\int_{\left[0,T\right]^{k}}\left|D_{s_{1},...,s_{k}}^{j_{1},...,j_{k}}F\right|^{2}ds_{1}\cdots ds_{k}\right)^{1/2}.
\]
We will denote by $\mathbb{D}^{k,p}$ the completion of the family
of random variables $\mathcal{S}$ with respect to the norm $\left\Vert \cdot\right\Vert _{k,p}$.
We also define the space $\mathbb{D}^{k,\infty}=\bigcap_{p\geq1}\mathbb{D}^{k,p}$.
We have the following chain rule formula for the Malliavin derivative. 
\begin{prop}
\label{prop: ChainRule}Let $\varphi:\mathbb{R}^{m}\rightarrow\mathbb{R}$
be of class $C_{P}^{1}$$\left(\mathbb{R}^{m}\right)$. Suppose that
$F=\left(F^{1},...,F^{m}\right)$ is a random vector whose components
belong to $\mathbb{D}^{1,\infty}$. Then, $\varphi\left(F\right)$$\in\mathbb{D}^{1,\infty}$
and 
\[
D_{t}^{j}\varphi\left(F\right)=\sum_{i=1}^{m}\frac{\partial\varphi}{\partial x^{i}}\left(F\right)D_{t}^{j}F^{i},
\]
where $t\in\left[0,T\right]$ and $\dot{j}=1,...,d.$ 
\end{prop}
\begin{proof}
The proof follows the same ideas as the proof of Proposition 1.2.3
in Nualart \cite{Nu06}, where is proved for $\varphi\in C_{b}^{1}\left(\mathbb{R}^{m}\right)$
and $F$$\in\mathbb{D}^{1,p}$. One can extend the result to $\varphi\in C_{P}^{1}\left(\mathbb{R}^{m}\right)$
by requiring $F\in\mathbb{D}^{1,\infty}$ and using Hölder inequality. 
\end{proof}
As a corollary of Proposition \ref{prop: ChainRule} one obtains that
the product rule and the binomial formula holds for the Malliavin
derivative of products of random variables in $\mathbb{D}^{1,\infty}$.
However, Proposition \ref{prop: ChainRule} does not apply to the
exponential function. In order to show that the likelihood functional
$e^{\xi_{t}}$ is smooth in the Malliavin sense we need the following
lemma.
\begin{lem}
\label{lem: MallDerExponential}Let $F$$\in\mathbb{D}^{1,\infty}$
and such that 
\begin{equation}
\mathbb{E}\left[\exp\left(p\left|F\right|\right)\right]<\infty,\label{eq:ExpAbsMoments}
\end{equation}
for all $p\geq1.$ Then $G=e^{F}$$\in\mathbb{D}^{1,\infty}$ and
\begin{equation}
D_{t}^{j}G=GD_{t}^{j}F,\label{eq:DerivExpo}
\end{equation}
where $t\in\left[0,T\right]$ and $\dot{j}=1,...,d.$ 
\end{lem}
\begin{proof}
Define $G_{n}=\sum_{k=0}^{n}\frac{F^{k}}{k!}.$ As $F$$\in\mathbb{D}^{1,\infty},$
Proposition \ref{prop: ChainRule} yields that $G_{n}\in\mathbb{D}^{1,\infty}$
and 
\[
DG_{n}=\sum_{k=1}^{n}\frac{kF^{k-1}}{k!}DF=G_{n-1}DF.
\]
In order to prove that $G\in\mathbb{D}^{1,\infty}$ and that the identity
$\left(\ref{eq:DerivExpo}\right)$ is satisfied, it suffices to show
that for all $p\geq1$ one has that $G_{n}$ converges to $G$ in
$L^{p}\left(\Omega\right)$ and 
\[
\mathbb{E}\left[\left\Vert DG-DG_{n}\right\Vert _{H}^{p}\right]\longrightarrow0,
\]
when $n$ tends to infinity. Note that 
\begin{align*}
\mathbb{E}\left[\left\Vert DG-DG_{n}\right\Vert _{H}^{p}\right] & =\mathbb{E}\left[\left|G-G_{n-1}\right|^{p}\left|\int_{0}^{T}\left|D_{t}F\right|^{2}dt\right|^{p/2}\right]\\
 & \leq\mathbb{E}\left[\left|G-G_{n-1}\right|^{2p}\right]^{1/2}\mathbb{E}\left[\left\Vert DF\right\Vert _{L^{2}\left(\left[0,T\right]\right)}^{2p}\right]^{1/2}.
\end{align*}
Hence, the problem is reduced to show that $G_{n}$ converges to $G$
in $L$$^{p}$ for all $p$$\geq1$. Equivalently, defining 
\[
G_{n}^{c}:=G-G_{n}=\sum_{k=n+1}^{\infty}\frac{F^{k}}{k!},
\]
it suffices to prove that $G_{n}^{c}$ converges to $0$ in $L$$^{p}$
for all $p$$\geq1$. Clearly, $G_{n}^{c}$ converges to $0$ almost
surely and, thanks to assumption $\left(\ref{eq:ExpAbsMoments}\right)$,
the dominated convergence theorem yields that $G_{n}^{c}$ also converges
to 0 in $L^{p}\left(\Omega\right)$ for all $p$$\geq1$. 
\end{proof}
We also have the following relationship between the conditional expectation
and the Malliavin derivative. 
\begin{lem}
\label{lem:CondExpMallDer}Let $F$$\in\mathbb{D}^{1,2}$ and $\mathbb{F}=\left\{ \mathcal{F}_{t}\right\} _{t\in\left[0,T\right]}$
be the P-augmented natural filtration generated by $B$. Then $\mathbb{E}\left[F|\mathcal{F}_{t}\right]\in\mathbb{D}^{1,2}$
and $D_{s}^{j}\mathbb{E}\left[F|\mathcal{F}_{t}\right]=\mathbb{E}\left[D_{s}^{j}F|\mathcal{F}_{t}\right]\mathbf{1}_{\left[0,t\right]}\left(s\right)$$,j=1,...,d.$ 
\end{lem}
\begin{proof}
The lemma is a particular case of Proposition 1.2.8 in Nualart \cite{Nu06}. 
\end{proof}
The following is an important result regarding the Malliavin differentiability
of the solution of a stochastic differential equation. 
\begin{lem}
\label{lem: MomentMalliavinDerivative}If $X_{t}\in\mathbb{R}^{n}$
is the solution to 
\[
X_{t}=x+\int_{0}^{t}V_{0}\left(X_{s}\right)ds+\int_{0}^{t}V\left(X_{s}\right)dB_{s},
\]
where the components of V$_{0}$ and $V$ are $m$-times continuously
differentiable with bounded derivatives of order greater or equal
than one and $B_{t}=(B_{t}^{1},...,B_{t}^{d})$ is a $d$-dimensional
Brownian motion. Then, $X_{t}^{i}\in\mathbb{D}^{m,\infty},t\in[0,T],i=1,...,n.$
Furthermore, for any $p\geq1$ one has that 
\[
\sup_{r_{1},r_{2},...,r_{k}\in[0,T]}\mathbb{E}\left[\sup_{r_{1}\vee r_{2}\vee\cdots\vee r_{k}\leq t\leq T}\left\vert D_{r_{1},r_{2},...,r_{k}}^{j_{1},j_{2},...,j_{k}}X_{t}^{i}\right\vert ^{p}\right]<\infty,
\]
for all $p\geq1,$$i=1,...,n,$ $j_{k}\in\left\{ 1,...,d\right\} $
and $1\leq k\leq m$. 
\end{lem}
\begin{proof}
See Nualart \cite{Nu06}, Theorem 2.2.1. and 2.2.2.
\end{proof}
\begin{rem}
\label{rem:CommuteDandY}We will be using a variation of the classical
Malliavin calculus known as \emph{partial Malliavin calculus}. This
calculus was introduced in Kusuoka and Stroock \cite{KuStr84} and
Nualart and Zakai \cite{NuZa89} with a view towards its application
to the stochastic filtering problem, see also Tanaka \cite{Ta14}.
The idea is to consider only the Malliavin derivative operator with
respect some of the components of the Brownian motion $B$. In our
setting $B$=$\left(V,Y\right)$ is a $d_{V}+d_{Y}$-dimensional Brownian
motion under $\tilde{P}$ and the Malliavin differentiation will be
only with respect to the Brownian motion $V$. The main consequence
of this approach is that the Malliavin derivative with respect to
$V$ commutes with the stochastic integral with respect to $Y$. 
\end{rem}
\begin{lem}
\label{lem: BoundedIteratedMD1}Let $m\in\mathbb{N}$ and assume that
\textbf{H}$(m)$ holds and $\varphi\in C_{P}^{m+1}$. Then, the random
variable $\varphi(X_{t})e^{\xi_{t}}$ belongs to $\mathbb{D}^{m+1,\infty}.$
Moreover, 
\[
\sup_{r_{1},...,r_{\left\vert \alpha\right\vert }\in[0,t]}\mathbb{\tilde{E}}\left[\left\vert D_{r_{1},...,r_{_{\left\vert \alpha\right\vert }}}^{\alpha_{1},...,\alpha_{\left\vert \alpha\right\vert }}(\varphi(X_{t})e^{\xi_{t}})\right\vert ^{p}\right]<\infty,
\]
for all $p\geq1$ and $\alpha\in\mathcal{M}_{m+1}(S_{1}).$ 
\end{lem}
\begin{proof}
To ease the notation we are only going to give the proof for $d_{V}=d_{Y}=d_{X}=1.$
We will also use the notation $D_{r_{1},...,r_{k}}^{k}F=D_{r_{1},...,r_{k}}^{1,\overset{k}{\overbrace{...}},1}F$.
Lemma \ref{lem: MomentMalliavinDerivative} yields that $X_{t}\in\mathbb{D}^{1,\infty}$.
Applying iteratively Proposition \ref{prop: ChainRule} we obtain
that $\varphi(X_{t})\in\mathbb{D}^{m+1,\infty}$ and $h$$(X_{t})\in\mathbb{D}^{m+1,\infty}$.
Taking into account Remark \ref{rem:CommuteDandY}, we have that $\xi_{t}\in\mathbb{D}^{m+1,\infty}$.
Moreover, thanks to Lemma \ref{lem: Z_t^p_Integrability}, we can
apply iteratively Lemma \ref{lem:CondExpMallDer} and conclude that
$e^{\xi_{t}}\in\mathbb{D}^{m+1,\infty}$. For any $\alpha\in\mathcal{M}_{m+1}(S_{1}),$
by Leibniz's rule, we can write 
\begin{align*}
D_{r_{1},...,r_{_{\left\vert \alpha\right\vert }}}^{\alpha_{1},...,\alpha_{\left\vert \alpha\right\vert }}\left(\varphi(X_{t})e^{\xi_{t}}\right) & =D_{r_{1},...,r_{_{\left\vert \alpha\right\vert }}}^{\left\vert \alpha\right\vert }\left(\varphi(X_{t})e^{\xi_{t}}\right)\\
 & =\sum_{k=0}^{\left\vert \alpha\right\vert }\binom{\left\vert \alpha\right\vert }{k}\left(D_{r_{1},...,r_{k}}^{k}\varphi(X_{t})\right)(D_{r_{1},...,r_{\left\vert \alpha\right\vert -k}}^{\left\vert \alpha\right\vert -k}e^{\xi_{t}}),
\end{align*}
and applying Schwartz's inequality one has that 
\begin{align*}
\mathbb{\tilde{E}}\left[\left\vert D_{r_{1},...,r_{_{\left\vert \alpha\right\vert }}}^{\left\vert \alpha\right\vert }\left(\varphi(X_{t})e^{\xi_{t}}\right)\right\vert ^{p}\right] & \leq C\sum_{k=0}^{\left\vert \alpha\right\vert }\binom{\left\vert \alpha\right\vert }{k}\mathbb{\tilde{E}}\left[\left\vert \left(D_{r_{1},...,r_{k}}^{k}\varphi(X_{t})\right)(D_{r_{1},...,r_{\left\vert \alpha\right\vert -k}}^{\left\vert \alpha\right\vert -k}e^{\xi_{t}})\right\vert ^{p}\right]\\
 & \leq C\sum_{k=0}^{\left\vert \alpha\right\vert }\binom{\left\vert \alpha\right\vert }{k}\mathbb{\tilde{E}}\left[\left\vert D_{r_{1},...,r_{k}}^{k}\varphi(X_{t})\right\vert ^{2p}\right]^{1/2}\mathbb{\tilde{E}}\left[\left\vert D_{r_{1},...,r_{\left\vert \alpha\right\vert -k}}^{\left\vert \alpha\right\vert -k}e^{\xi_{t}}\right\vert ^{2p}\right]^{1/2}.
\end{align*}
Hence, the result follows if we show that 
\begin{align}
\sup_{r_{1},...,r_{k}\in[0,t]}\mathbb{\tilde{E}}\left[|D_{r_{1},...,r_{k}}^{k}\varphi(X_{t})|^{p}\right] & <\infty,\quad0\leq k\leq\left\vert \alpha\right\vert ,\label{eq: MDF}\\
\sup_{r_{1},...,r_{k}\in[0,t]}\mathbb{\tilde{E}}[|D_{r_{1},...,r_{k}}^{k}e^{\xi_{t}}|^{p}] & <\infty,\quad0\leq k\leq\left\vert \alpha\right\vert ,\label{eq: MDG}
\end{align}
for any $p\geq1$.

\emph{\lyxarrow{}Proof of (\ref{eq: MDF})}$:$

If $k=0,$ using that $\mathbf{H}\left(m\right)$ holds and $\varphi\in C_{P}^{m+1}$,
we have that $\mathbb{\tilde{E}}[|\varphi(X_{t})|^{p}]<\infty$, by
Remark \ref{rem: M}. If $1\leq k\leq\left\vert \alpha\right\vert ,$
we use Faà di Bruno's formula to obtain an expression for $D_{r_{1},...,r_{k}}^{k}\varphi(X_{t})$
in terms of the so called partial Bell polynomials, which are given
by 
\[
B_{k,a}(x_{1},...,x_{k})=\sum_{(j_{1},...,j_{k})\in\Lambda(k,a)}\frac{k!}{j_{1}!\left(1!\right)^{j_{1}}j_{2}!\left(2!\right)^{j_{2}}\cdots j_{k}!(k!)^{j_{k}}}x_{1}^{j_{1}}x_{2}^{j_{2}}\cdots x_{k}^{j_{k}},
\]
where $1\leq a\leq k$ and 
\[
\Lambda(k,a)=\{(j_{1},...,j_{k})\in\mathbb{Z}_{+}^{k}:j_{1}+2j_{2}+\cdots+kj_{k}=k,j_{1}+j_{2}+\cdots+j_{k}=a\}.
\]
In particular, we have that 
\[
D_{r_{1},...,r_{k}}^{k}\varphi(X_{t})=\sum_{a=1}^{k}\varphi^{(a)}(X_{t})B_{k,a}(D_{r_{1}}^{1}X_{t},D_{r_{1},r_{2}}^{2}X_{t},...,D_{r_{1},...,r_{k}}^{k}X_{t}).
\]
Hence, for any $p\geq1,$ applying Cauchy-Schwarz inequality we get
\begin{align*}
 & \mathbb{\tilde{E}}[|D_{r_{1},...,r_{k}}^{k}\varphi(X_{t})|^{p}]\\
 & \leq C\sum_{a=1}^{k}\mathbb{\tilde{E}}[|\varphi^{(a)}(X_{t})B_{k,a}(D_{r_{1}}^{1}X_{t},D_{r_{1},r_{2}}^{2}X_{t},...,D_{r_{1},...,r_{k}}^{k}X_{t})|^{p}]\\
 & \leq C\sum_{a=1}^{k}\mathbb{\tilde{E}}[|\varphi^{(a)}(X_{t})|^{2p}]^{1/2}\mathbb{\tilde{E}}[|B_{k,a}(D_{r_{1}}^{1}X_{t},D_{r_{1},r_{2}}^{2}X_{t},...,D_{r_{1},...,r_{k}}^{k}X_{t})|^{2p}]^{1/2}.
\end{align*}
The terms $\mathbb{\tilde{E}}[\left\vert \varphi^{(a)}(X_{t})\right\vert ^{2p}]<\infty,a=1,...,k,$
due to Remark \ref{rem: M} combined with that $\mathbf{H}\left(m\right)$
holds and $\varphi\in C_{P}^{m+1}.$ On the other hand, using the
generalized version of Hölder's inequality, Lemma \ref{lem: Generalized Holder},
we can bound 
\[
\mathbb{\tilde{E}}[|B_{k,a}(D_{r_{1}}^{1}X_{t},D_{r_{1},r_{2}}^{2}X_{t},...,D_{r_{1},...,r_{k}}^{k}X_{t})|^{2p}],\quad1\leq a\leq k,
\]
by a sum of products of expectations of powers of Malliavin derivatives
of $X$ of different orders. Combining this bound with Lemma \ref{lem: MomentMalliavinDerivative}
we get that the integrability condition $\left(\ref{eq: MDF}\right)$
is satisfied.

\emph{\lyxarrow{}Proof of (\ref{eq: MDG})}$:$

If $k=0,$ we have that $\mathbb{\tilde{E}}\left[|e^{\xi_{t}}|^{p}\right]<\infty$
due to Lemma \ref{lem: Z_t^p_Integrability}. If $1\leq k\leq\left\vert \alpha\right\vert ,$
using again Faà di Bruno's formula we get 
\begin{align*}
D_{r_{1},...,r_{k}}^{k}e^{\xi_{t}} & =\sum_{a=1}^{k}\left.\frac{d^{a}}{dx^{a}}e^{x}\right\vert _{x=\xi_{t}}B_{k,a}(D_{r_{1}}^{1}\xi_{t},D_{r_{1},r_{2}}^{2}\xi_{t},...,D_{r_{1},...,r_{k}}^{k}\xi_{t})\\
 & =\sum_{a=1}^{k}\exp(\xi_{t})B_{k,a}(D_{r_{1}}^{1}\xi_{t},D_{r_{1},r_{2}}^{2}\xi_{t},...,D_{r_{1},...,r_{k}}^{k}\xi_{t}).
\end{align*}
We can repeat exactly the same arguments as in the proof of $\left(\ref{eq: MDF}\right),$
due to the fact that by Lemma \ref{lem: Z_t^p_Integrability} $e^{\xi_{t}}$
has moment of all orders, provided we can show that 
\begin{equation}
\sup_{r_{1},...,r_{a}\in[0,t]}\mathbb{\tilde{E}}[|D_{r_{1},...,r_{a}}^{a}\xi_{t}|^{p}]<\infty,\quad1\leq a\leq k,\label{eq: MDXi}
\end{equation}
for any $p\geq1.$ As noted in Remark \ref{rem:CommuteDandY}, the
Malliavin derivative commute with the stochastic integral with respect
to $Y$ and we can write 
\begin{align*}
D_{r_{1},...,r_{a}}^{a}\xi_{t} & =D_{r_{1},...,r_{a}}^{a}\left(\int_{0}^{t}h(X_{s})dY_{s}-\frac{1}{2}\int_{0}^{t}h^{2}(X_{s})ds\right)\\
 & =\int_{0}^{t}D_{r_{1},...,r_{a}}^{a}h(X_{s})dY_{s}-\frac{1}{2}\int_{0}^{t}D_{r_{1},...,r_{a}}^{a}\left(h^{2}(X_{s})\right)ds.
\end{align*}

Hence, by Burkholder-Davis-Gundy inequality and Jensen's inequality,
we get for any $p\geq1$ that 
\begin{align*}
 & \mathbb{\tilde{E}}[|D_{r_{1},...,r_{a}}^{a}\xi_{t}|^{2p}]\\
 & \leq C\left\{ \mathbb{\tilde{E}}\left[\left\vert \int_{0}^{t}D_{r_{1},...,r_{a}}^{a}h(X_{s})dY_{s}\right\vert ^{2p}\right]+\mathbb{\tilde{E}}\left[\left\vert \int_{0}^{t}D_{r_{1},...,r_{a}}^{a}\left(h^{2}(X_{s})\right)ds\right\vert ^{2p}\right]\right\} \\
 & \leq C\left\{ \mathbb{\tilde{E}}\left[\int_{0}^{t}\left\vert D_{r_{1},...,r_{a}}^{a}h(X_{s})\right\vert ^{2p}ds\right]+\mathbb{\tilde{E}}\left[\int_{0}^{t}\left\vert D_{r_{1},...,r_{a}}^{a}\left(h^{2}(X_{s})\right)\right\vert ^{2p}ds\right]\right\} \\
 & \leq C\left\{ A_{1}+A_{2}\right\} .
\end{align*}
Applying Faà di Bruno formula we can write 
\begin{align*}
A_{1} & \leq C\sum_{l=1}^{a}\int_{0}^{t}\mathbb{\tilde{E}}[\left\vert h^{(l)}(X_{s})B_{a,l}(D_{r_{1}}^{1}X_{s},D_{r_{1},r_{2}}^{2}X_{s},...,D_{r_{1},...,r_{a}}^{a}X_{s})\right\vert ^{2p}]ds\\
 & \leq C\left\Vert h\right\Vert _{\infty,a}^{q}\sum_{l=1}^{a}\int_{0}^{t}\mathbb{\tilde{E}}[\left\vert B_{a,l}(D_{r_{1}}^{1}X_{s},D_{r_{1},r_{2}}^{2}X_{s},...,D_{r_{1},...,r_{a}}^{a}X_{s})\right\vert ^{2p}]ds,
\end{align*}
where 
\[
\left\Vert h\right\Vert _{\infty,a}\triangleq\sum_{i=0}^{d_{Y}}\sum_{l=0}^{a}\left\Vert h_{i}^{(l)}\right\Vert _{\infty}<\infty,
\]
because $\mathbf{H}\left(m\right)$ holds. Therefore, using the generalized
version of Hölder inequality, Lemma \ref{lem: Generalized Holder},
and Lemma \ref{lem: MomentMalliavinDerivative} we get $A_{1}<\infty$.
We can repeat the same argument for $A_{2}$ and obtain $\left(\ref{eq: MDXi}\right).$ 
\end{proof}

\subsection{\label{subsec:MartingaleRep}Martingale representations and Clark-Ocone
formula}

In this section we recall the Clark-Ocone formula. This formula relates
the kernels in the Itô martingale representation of Malliavin differentiable
functionals with the Malliavin derivatives of such functionals. We
present a truncated version of the well known Stroock-Taylor formula,
see Stroock \cite{Stro87}, that can be seen as an extension of the
Clark-Ocone formula and it will be essential in deducing several conditional
expectation estimates (see Section \ref{subsec:CondExpEst}). We also
show that, if the coefficients $f$ and $\sigma$ of of the SDE modeling
the signal, the sensor function $h$ and the test function $\varphi$
are regular enough with bounded derivatives then, the kernels in the
truncated Stroock-Taylor formula for $\varphi\left(X_{t}\right)e^{\xi_{t}}$
satisfy a uniform integrability property. Finally, we show that those
kernels also satisfy a Hölder continuity property. 
\begin{thm}
\label{thm: Integral Representation}Let $F\in L^{2}(\Omega,\mathcal{H}_{t}^{t},\tilde{P})$
\textbf{.} Then, $F$ admits the following martingale representation
\[
F=\mathbb{\tilde{E}}[F|\mathcal{H}_{0}^{t}]+\sum_{r=1}^{d_{V}}\int_{0}^{t}J_{s}^{r}dV_{s}^{r},
\]
where $J^{r}=\{J_{s}^{r},s\in\left[0,t\right]\},r=1,...,d_{V}$ are
$\mathcal{H}_{s}^{t}$-progressively measurable processes such that
\[
\mathbb{\tilde{E}}\left[\int_{0}^{t}\left\vert J_{s}^{r}\right\vert ^{2}ds_{1}\right]<\infty,\quad r=1,...,d_{V}.
\]
Moreover, if $F$$\in\mathbb{D}^{1,2}$ then 
\[
J_{s}^{r}=\mathbb{\tilde{E}}[D_{s}^{r}F|\mathcal{H}_{s}^{t}],\quad s\in\left[0,t\right],
\]
which is known as the Clark-Ocone formula. 
\end{thm}
\begin{proof}
The proof is similar to that of Lemma 17 in Crisan \cite{Cris11}
and the proof of the Clark-Ocone formula can be found in Nualart \cite{Nu06},
Proposition 1.3.14. 
\end{proof}
By applying Theorem \ref{thm: Integral Representation} to the kernels
$J^{r},r=1,...,d_{V}$ one can get the following result. 
\begin{thm}[Stroock-Taylor formula of order m]
\label{thm: ST formula}Assume that $F\in L^{2}(\Omega,\mathcal{H}_{t}^{t},\tilde{P}).$
Then, for $m\in\mathbb{N}$ we can write 
\[
F=\sum_{\beta\in\mathcal{M}_{m-1}(S_{1})}I_{\beta}\left(\mathbb{\tilde{E}}\left[J_{s_{1},...,s_{\left\vert \beta\right\vert }}^{\beta}|\mathcal{H}_{0}^{t}\right]\right){}_{0,t}+\sum_{\beta\in\mathcal{R}\left(\mathcal{M}_{m-1}(S_{1})\right)}I_{\beta}\left(J_{s_{1},...,s_{\left\vert \beta\right\vert }}^{\beta}\right){}_{0,t},
\]
where the kernels $J_{s_{1},...,s_{\left\vert \beta\right\vert }}^{\beta}$
for $\beta\in\mathcal{M}_{m}(S_{1})$ are obtained from the martingale
representation of $J_{s_{2},...,s_{\left\vert \beta\right\vert }}^{-\beta}$,
that is, they satisfy the following relationship

\begin{align*}
J^{v} & \triangleq F,\\
J_{s_{2},...,s_{\left\vert \beta\right\vert }}^{-\beta} & =\mathbb{\tilde{E}}\left[J_{s_{2},...,s_{\left\vert \beta\right\vert }}^{-\beta}|\mathcal{H}_{0}^{t}\right]+\sum_{\beta_{1}=1}^{d_{V}}\int_{0}^{s_{2}}J_{s_{1},...,s_{\left\vert \beta\right\vert }}^{\beta_{1}*\left(-\beta\right)}dV_{s_{1}}^{\beta_{1}}.
\end{align*}
Moreover, if $\varphi(X_{t})e^{\xi_{t}}$$\in\mathbb{D}^{m,2}$ then
\[
J_{s_{1},...,s_{\left\vert \beta\right\vert }}^{\beta}=\mathbb{\tilde{E}}\left[D_{s_{1},...,s_{\left\vert \beta\right\vert }}^{\beta}F|\mathcal{H}_{s_{1}}^{t}\right],\quad\beta\in\mathcal{M}_{m}(S_{1}).
\]
\end{thm}
\begin{proof}
We prove the result by induction. For $m=1$, the result is precisely
Theorem \ref{thm: Integral Representation}. We assume that the result
holds for $m-1\ge0$ and prove that this implies that it also holds
for $m$. By the induction hypothesis we have that 
\[
F=\sum_{\beta\in\mathcal{M}_{m-2}(S_{1})}I_{\beta}\left(\mathbb{\tilde{E}}\left[J_{s_{1},...,s_{\left\vert \beta\right\vert }}^{\beta}|\mathcal{H}_{0}^{t}\right]\right){}_{0,t}+\sum_{\beta\in\mathcal{R}\left(\mathcal{M}_{m-2}(S_{1})\right)}I_{\beta}\left(J_{s_{1},...,s_{\left\vert \beta\right\vert }}^{\beta}\right){}_{0,t}.
\]
Applying Theorem \ref{thm: Integral Representation} to $J_{s_{1},...,s_{\left\vert \beta\right\vert }}^{\beta},\beta\in\mathcal{R}\left(\mathcal{M}_{m-2}(S_{1})\right)$
we get 
\begin{align*}
F & =\sum_{\beta\in\mathcal{M}_{m-2}(S_{1})}I_{\beta}\left(\mathbb{\tilde{E}}\left[J_{s_{1},...,s_{\left\vert \beta\right\vert }}^{\beta}|\mathcal{H}_{0}^{t}\right]\right){}_{0,t}+\sum_{\beta\in\mathcal{R}\left(\mathcal{M}_{m-2}(S_{1})\right)}I_{\beta}\left(\mathbb{\tilde{E}}\left[J_{s_{1},...,s_{\left\vert \beta\right\vert }}^{\beta}|\mathcal{H}_{0}^{t}\right]\right){}_{0,t}\\
 & +\sum_{\beta\in\mathcal{R}\left(\mathcal{M}_{m-2}(S_{1})\right)}\sum_{r=1}^{d_{V}}I_{r*\beta}\left(J_{s,s_{1},...,s_{\left\vert \beta\right\vert }}^{r*\beta}\right){}_{0,t}\\
 & =\sum_{\beta\in\mathcal{M}_{m-1}(S_{1})}I_{\beta}\left(\mathbb{\tilde{E}}\left[J_{s_{1},...,s_{\left\vert \beta\right\vert }}^{\beta}|\mathcal{H}_{0}^{t}\right]\right){}_{0,t}+\sum_{\beta\in\mathcal{R}\left(\mathcal{M}_{m-1}(S_{1})\right)}I_{\beta}\left(J_{s_{1},...,s_{\left\vert \beta\right\vert }}^{\beta}\right){}_{0,t},
\end{align*}
where in the last equality we have used that 
\[
\mathcal{M}_{m-1}(S_{1})=\mathcal{M}_{m-2}(S_{1})\biguplus\mathcal{R}\left(\mathcal{M}_{m-2}(S_{1})\right),
\]
the definitions of $\mathcal{R}\left(\mathcal{M}_{m-1}(S_{1})\right)$
and the concatenation of multi-indices.

The Clark-Ocone representation of the kernels also follows from a
straightforward induction. 
\end{proof}
\begin{prop}
\label{prop: UnifBoundKernels}Let $m\in\mathbb{N}$ and assume that\textbf{\ H}$(m)$
holds and $\varphi\in C_{P}^{m+1}$. Then, the kernels $J_{s_{1},...,s_{\left\vert \beta\right\vert }}^{\beta},$
$\beta\in\mathcal{M}_{m+1}(S_{1})$ appearing in the Stroock-Taylor
formula of order $m+1$ for $\varphi(X_{t})e^{\xi_{t}}$ satisfy 
\[
\sup_{0\leq s_{1}<\cdots<s_{\left|\beta\right|}\leq t}\mathbb{\tilde{E}}\left[\left\vert J_{s_{1},...,s_{\left|\beta\right|}}^{\beta}\right\vert ^{p}\right]<\infty,
\]
for $p\geq1$ 
\end{prop}
\begin{proof}
It is a straightforward combination of Theorem \ref{thm: ST formula},
Jensen's inequality for conditional expectations and Lemma \ref{lem: BoundedIteratedMD1}. 
\end{proof}
\begin{lem}
\label{lem: RegularityKernels}Assume that\textbf{ H}$(1)$ holds
and $\varphi\in C_{P}^{2}$. Then, the kernels $J^{r}=\{J_{s}^{r},s\in\left[0,t\right]\},r=1,...,d_{V}$
in the martingale representation of $\varphi(X_{t})e^{\xi_{t}}$ satisfy
the following Hölder continuity property: 
\[
\mathbb{\tilde{E}}\left[\left\vert J_{s}^{r}-J_{u}^{r}\right\vert ^{2p}\right]\leq C\left(s-u\right)^{p},\qquad0\leq u\leq s\leq t,
\]
for $p\geq1.$ 
\end{lem}
\begin{proof}
The idea is to use the Clark-Ocone formula, Theorem \ref{thm: Integral Representation}.
That is, one has the following representation 
\[
J_{s}^{r}=\mathbb{\tilde{E}}\left[D_{s}^{r}\left\{ \varphi\left(X_{t}\right)e^{\xi_{t}}\right\} |\mathcal{H}_{s}^{t}\right],\qquad0\leq s\leq t,
\]
where $D_{s}^{r}$ denotes the Malliavin derivative with respect to
$V^{r}.$ Hence, we can write 
\begin{align*}
J_{s}^{r}-J_{u}^{r} & =\mathbb{\tilde{E}}\left[D_{s}^{r}\left\{ \varphi\left(X_{t}\right)e^{\xi_{t}}\right\} -D_{u}^{r}\left\{ \varphi\left(X_{t}\right)e^{\xi_{t}}\right\} |\mathcal{H}_{s}^{t}\right]\\
 & \quad+\mathbb{\tilde{E}}\left[D_{u}^{r}\left\{ \varphi\left(X_{t}\right)e^{\xi_{t}}\right\} |\mathcal{H}_{s}^{t}\right]-\mathbb{\tilde{E}}\left[D_{u}^{r}\left\{ \varphi\left(X_{t}\right)e^{\xi_{t}}\right\} |\mathcal{H}_{u}^{t}\right]\\
 & \triangleq A_{1}+A_{2}.
\end{align*}
For the term $A_{1},$ note that we can write 
\begin{align*}
 & D_{s}^{r}\left\{ \varphi\left(X_{t}\right)e^{\xi_{t}}\right\} -D_{u}^{r}\left\{ \varphi\left(X_{t}\right)e^{\xi_{t}}\right\} \\
 & =e^{\xi_{t}}\left(D_{s}^{r}\varphi\left(X_{t}\right)-D_{u}^{r}\varphi\left(X_{t}\right)\right)+\varphi\left(X_{t}\right)\left(D_{s}^{r}e^{\xi_{t}}-D_{u}^{r}e^{\xi_{t}}\right)\\
 & =e^{\xi_{t}}\sum_{j=1}^{d_{X}}\partial_{j}\varphi\left(X_{t}\right)\left(D_{s}^{r}X_{t}^{j}-D_{u}^{r}X_{t}^{j}\right)+e^{\xi_{t}}\varphi\left(X_{t}\right)\left(D_{s}^{r}\xi_{t}-D_{u}^{r}\xi_{t}\right),
\end{align*}
and 
\begin{align*}
D_{s}^{r}\xi_{t}-D_{u}^{r}\xi_{t} & =\sum_{i=1}^{d_{Y}}\sum_{k=1}^{d_{X}}\int_{0}^{t}\partial_{k}h^{i}\left(X_{v}\right)\left(D_{s}^{r}X_{v}^{j}-D_{u}^{r}X_{v}^{j}\right)dY_{v}^{i}\\
 & \quad+\frac{1}{2}\sum_{i=1}^{d_{Y}}\sum_{k=1}^{d_{X}}\int_{0}^{t}\partial_{k}\{h^{i}\left(X_{v}\right)^{2}\}\left(D_{s}^{r}X_{v}^{j}-D_{u}^{r}X_{v}^{j}\right)dv.
\end{align*}
Hence, the result follows from the fact that 
\[
\mathbb{\tilde{E}}\left[\left\vert D_{s}^{r}X_{v}^{j}-D_{u}^{r}X_{v}^{j}\right\vert ^{2p}\right]\leq C\left(s-u\right)^{p},
\]
$D_{s}^{r}X_{t}^{j}$ satisfies an evolution equation drive by a Brownian
motion, see Section 2.2.2 in Nualart \cite{Nu06}. For the term $A_{2}$
the result follows from the martingale representation theorem, Theorem
\ref{thm: Integral Representation}, applied to the random variable
$D_{u}\left\{ \varphi\left(X_{t}\right)e^{\xi_{t}}\right\} \in L^{2}\left(\Omega,\mathcal{H}_{t}^{t},\tilde{P}\right)$
which yields 
\[
D_{u}^{r}\left\{ \varphi\left(X_{t}\right)e^{\xi_{t}}\right\} =\mathbb{\tilde{E}}\left[\text{\ensuremath{D_{u}^{r}\left\{ \varphi\left(X_{t}\right)e^{\xi_{t}}\right\} }\ensuremath{|}}\mathcal{H}_{0}^{t}\right]+\sum_{r_{1}=1}^{d_{V}}\int_{0}^{t}G_{v}^{r_{1}}dV_{v}^{r_{1}},
\]
and, hence, $A_{2}=\sum_{r_{1}=1}^{d_{V}}\int_{u}^{s}G_{v}^{r_{1}}dV_{v}^{r_{1}}$
and 
\begin{align*}
\mathbb{\tilde{E}}\left[\left|A_{2}\right|^{2p}\right] & \leq C\sum_{r_{1}=1}^{d_{V}}\mathbb{\tilde{E}}\left[\left|\int_{u}^{s}\left|G_{v}^{r_{1}}\right|^{2}dv\right|^{p}\right]\\
 & \leq C\left(s-u\right)^{p-1}\sum_{r_{1}=1}^{d_{V}}\int_{u}^{s}\mathbb{\tilde{E}}\left[\left|G_{v}^{r_{1}}\right|^{2p}\right]dv\\
 & \leq C\left(s-u\right)^{p}\sum_{r_{1}=1}^{d_{V}}\sup_{0\leq v\leq t}\mathbb{\tilde{E}}\left[\left|G_{v}^{r_{1}}\right|^{2p}\right]\\
 & \leq C\left(s-u\right)^{p},
\end{align*}
where in the last inequality we have used Proposition \ref{prop: UnifBoundKernels}. 
\end{proof}
\begin{rem}
\label{rem: HighRegularityKernel}If \textbf{H}$(m)$ holds and $\varphi\in C_{P}^{m+1}$,
using the same reasonings as in Lemma \ref{lem: RegularityKernels},
one can show that the kernels $J^{\beta},\beta\in\mathcal{M}_{m}\left(S_{1}\right)$
in the Stroock-Taylor formula for $\varphi(X_{t})e^{\xi_{t}}$ satisfy
the following Hölder continuity property: 
\[
\mathbb{\tilde{E}}\left[\left\vert J_{s_{1},...,s,...s_{\left|\beta\right|}}^{\beta}-J_{s_{1},...,s_{i-1},u,s_{i+1},...s_{\left|\beta\right|}}^{\beta}\right\vert ^{2p}\right]\leq C\left|s-u\right|^{p},
\]
for $p\geq1,$ $s_{i-1}\leq u\leq s\leq s_{i+1}$, $i=2,...,m-1$. 
\end{rem}

\subsection{\label{subsec:BackwardMart}Backward martingales estimates}

In this section we start reviewing some basic concepts of backward
Itô integration that can be found, for instance, in Pardoux and Protter
\cite{ParPro87}, Bensoussan \cite{Ben92} and Applebaum \cite{App09}.
Then we compute some technical estimates related to products of backward
Itô integrals and backward stochastic exponentials that will be useful
in the next section.

We know that under $\tilde{P}$ the observation process $Y$ is a
Brownian motion with respect to the filtration $\mathcal{\mathbb{Y}}$.
For fixed $t\geq0,$ we can consider the process $\overleftarrow{Y}=\{\overleftarrow{Y}_{s}\triangleq Y_{s}-Y_{t}\}_{0\leq s\leq t}$
which is a Brownian motion with respect to the backward filtration
\[
\mathcal{\mathbb{Y}}^{t}=\left\{ \mathcal{Y}_{s}^{t}\triangleq\sigma\left(\overleftarrow{Y}_{u},s\leq u\leq t\right)\vee\mathcal{N}\right\} _{0\leq s\leq t},
\]
where $\mathcal{N}$ are all the $P$-null sets of $(\Omega,\mathcal{F},P)$.
We can also consider the filtration $\mathbb{Y}^{0,V}=\left\{ \mathcal{Y}_{s}^{0,V}\triangleq\mathcal{F}_{t}^{0,V}\vee\mathcal{Y}_{s}\right\} _{0\leq s\leq t}$
and the backward filtration $\mathbb{Y}^{0,V,t}=\{\mathcal{Y}_{s}^{0,V,t}\triangleq\mathcal{F}_{t}^{0,V}\vee\mathcal{Y}_{s}^{t}\}$$_{0\leq s\leq t}$.
As $X$$_{0}$ and $V$ are independent of $Y$ under $\tilde{P}$,
we also have that $Y$ is a $\mathbb{Y}^{0,V}$-Brownian motion and
$\overleftarrow{Y}$ is $\mathbb{Y}^{0,V,t}$-Brownian motion.

If $\eta=\left\{ \eta_{s}^{1},...,\eta_{s}^{d_{Y}}\right\} _{0\leq s\leq t}$
is a square integrable measurable process adapted to $\mathcal{\mathbb{Y}}^{0,V,t}$
we can define the backward Itô integral of $\eta$ with respect to
$\overleftarrow{Y}$ by 
\begin{align*}
\int_{s}^{t}\eta_{u}d\overleftarrow{Y}_{u} & \triangleq\sum_{i=1}^{d_{Y}}\int_{s}^{t}\eta_{u}^{i}d\overleftarrow{Y^{i}}_{u}\\
 & \triangleq L^{2}(\tilde{P})-\lim_{\tau\in\Pi\left(t\right),\left|\tau\right|\rightarrow0}\sum_{i=1}^{d_{Y}}\sum_{j=0}^{n-1}\eta_{t_{j+1}}\left(\overleftarrow{Y^{i}}_{t_{j+1}\lor s}-\overleftarrow{Y^{i}}_{t_{j}\lor s}\right),\\
 & =L^{2}(\tilde{P})-\lim_{\tau\in\Pi\left(t\right),\left|\tau\right|\rightarrow0}\sum_{i=1}^{d_{Y}}\sum_{j=0}^{n-1}\eta_{t_{j+1}}\left(Y_{t_{j+1}\lor s}^{i}-Y_{t_{j}\lor s}^{i}\right).
\end{align*}

\begin{rem}
\label{rem: Backward Ito integral}If a square integrable process
$\theta=\left\{ \theta_{u}\right\} _{0\leq u\leq t}$ is simultaneously
adapted to $\mathbb{Y}^{0,V}$and $\mathcal{\mathbb{Y}}^{0,V,t}$
both, the Itô and the backward Itô integrals, can be defined over
the same interval but, in general, they will be different. However,
if $\theta_{u}$ is measurable with respect to respect to $\mathcal{Y}_{0}^{0,V}=\mathcal{Y}_{t}^{0,V,t}=\mathcal{F}_{t}^{0,V}$
for all $0\leq u\leq t$, then both integrals coincide. In fact, they
coincide with the Stratonovich integral, see Pardoux and Protter \cite{ParPro87}.
This means that in the statement of Lemma \ref{lem: Main Backward}
we can change all backward Itô integrals by Itô integrals and the
estimates will hold true. However, in the proof of Lemma \ref{lem: Main Backward}
we use the properties of the backward integral and for that reason
we keep the notation of backward integration. 
\end{rem}
The backward Itô integral is analogous to the Itô integral. In particular,
the backward Itô integral has zero expectation and it is a backward
martingale with respect to $\mathcal{\mathbb{Y}}^{0,V,t}$, that is
\[
\mathbb{\tilde{E}}\left[\int_{s_{1}}^{t}\eta_{u}d\overleftarrow{Y}_{u}|\mathcal{Y}_{s_{2}}^{0,V,t}\right]=\int_{s_{2}}^{t}\eta_{u}d\overleftarrow{Y}_{u},\quad0\leq s_{1}<s_{2}\leq t.
\]
A backward Itô process is a process of the following form 
\[
Z_{s}=Z_{t}+\int_{s}^{t}\nu_{u}du+\int_{s}^{t}\eta_{u}d\overleftarrow{Y}_{u},\quad0\leq s\leq t,
\]
where $\nu$ and $\eta$ are two square integrable, measurable and
$\mathcal{\mathbb{Y}}^{0,V,t}$-adapted processes of the appropriate
dimensions. For backward Itô processes and $f$$\in C^{1,2}\left(\left(0,t\right)\times\mathbb{R}^{d_{Z}};\mathbb{R}\right)$
we have the following Itô formula, see Bensoussan \cite{Ben92}, 
\begin{align*}
f\left(s,Z_{s}\right) & =f\left(t,Z_{t}\right)+\int_{s}^{t}\left\{ -\partial_{t}f\left(u,Z_{u}\right)+\nu_{u}Df\left(u,Z_{u}\right)+\frac{1}{2}\text{tr}\left(D^{2}f\left(u,Z_{u}\right)\eta_{u}\eta_{u}^{T}\right)\right\} du\\
 & \quad+\int_{s}^{t}Df\left(u,Z_{u}\right)\eta_{u}d\overleftarrow{Y}_{u},
\end{align*}
where $D$ and $D^{2}$ stand for the gradient and the Hessian, respectively,
with respect to the space variables. As a corollary, one gets the
integration by parts formula 
\begin{align}
\left(\int_{s}^{t}\eta_{u}^{i}d\overleftarrow{Y^{i}}_{u}\right)\left(\int_{s}^{t}\eta_{u}^{j}d\overleftarrow{Y^{j}}_{u}\right) & =\int_{s}^{t}\left(\int_{u}^{t}\eta_{v}^{i}d\overleftarrow{Y^{i}}_{v}\right)\eta_{u}^{j}d\overleftarrow{Y^{j}}_{u}\label{eq: Backward_IBP}\\
 & \quad+\int_{s}^{t}\left(\int_{u}^{t}\eta_{v}^{j}d\overleftarrow{Y^{j}}_{v}\right)\eta_{u}^{i}d\overleftarrow{Y^{i}}_{u}\nonumber \\
 & \quad+\mathbf{1}_{\{i=j\}}\int_{s}^{t}\eta_{u}^{i}\eta_{u}^{j}du.\nonumber 
\end{align}
and 
\begin{equation}
\tilde{\mathbb{E}}\left[\left(\int_{s}^{t}\eta_{u}^{i}d\overleftarrow{Y^{i}}_{u}\right)\left(\int_{s}^{t}\eta_{u}^{j}d\overleftarrow{Y^{j}}_{u}\right)\right]=\mathbf{1}_{\{i=j\}}\tilde{\mathbb{E}}\left[\int_{s}^{t}\eta_{u}^{i}\eta_{u}^{j}du\right].\label{eq: EquBackMartVariation}
\end{equation}
Let $\phi\in\mathcal{B}_{b}\left(\mathbb{R}^{d_{X}};\mathbb{R}^{d_{Y}}\right)$
be a bounded measurable function and let $M^{t}\left(\phi\right)=\left\{ M_{s}^{t}\left(\phi\right)\right\} _{0\leq s\leq t}$
be the process 
\[
M_{s}^{t}\left(\phi\right)=\exp\left(\sum_{i=1}^{d_{Y}}\int_{s}^{t}\phi_{i}(X_{u})d\overleftarrow{Y_{u}^{i}}-\frac{1}{2}\sum_{i=1}^{d_{Y}}\int_{s}^{t}\phi_{i}^{2}\left(X_{u}\right)du\right).
\]
It is easy to show that $M^{t}$$\left(\phi\right)$ is a backward
martingale with respect to $\mathcal{\mathbb{Y}}^{0,V,t}$ and applying
the backward Itô formula one finds the following formula for the increments
of $M^{t}$$\left(\phi\right)$ 
\begin{equation}
M_{s_{1}}^{t}\left(\phi\right)=M_{s_{2}}^{t}\left(\phi\right)+\sum_{i=1}^{d_{Y}}\int_{s_{1}}^{s_{2}}M_{u}^{t}\left(\phi\right)\phi_{i}\left(X_{u}\right)d\overleftarrow{Y_{u}^{i}}.\label{eq: EquBackMartIncrem}
\end{equation}
Moreover, by the same reasoning as in Lemma \ref{lem: Z_t^p_Integrability},
we have for all $p\geq1$ 
\begin{equation}
\mathbb{\tilde{E}}\left[\left|M_{s}^{t}\left(\phi\right)\right|^{p}\right]<\infty.\label{eq:MomentsMts}
\end{equation}

\begin{lem}
\label{lem: Backward Martingale}Let $0\leq s_{1}\leq s_{2}\leq s_{3}\leq t$,
$\Psi$ be $\mathcal{Y}_{s_{3}}^{0,V,t}$-measurable random variable
and $\theta=\left\{ \theta_{u}\right\} _{0\leq u\leq t}$ a square
integrable and measurable process such that $\theta_{u}$ is measurable
with respect to $\mathcal{Y}_{s_{3}}^{0,V,t}$ for all $s_{2}\leq u\leq s_{3}$.
Then, 
\[
\mathbb{\tilde{E}}\left[\Psi M_{s_{1}}^{t}\left(\phi\right)\int_{s_{2}}^{s_{3}}\theta_{u}d\overleftarrow{Y_{u}^{i}}\right]=\mathbb{\tilde{E}}\left[\Psi M_{s_{3}}^{t}\left(\phi\right)\int_{s_{2}}^{s_{3}}\phi_{i}\left(X_{u}\right)\theta_{u}du\right].
\]
\end{lem}
\begin{proof}
By the backward martingale properties of $M_{s}^{t}\left(\phi\right)$
and equation $\left(\ref{eq: EquBackMartIncrem}\right)$ we can write
\begin{align*}
\mathbb{\tilde{E}}\left[\Psi M_{s_{1}}^{t}\left(\phi\right)\int_{s_{2}}^{s_{3}}\theta_{u}d\overleftarrow{Y_{u}^{i}}\right] & =\mathbb{\tilde{E}}\left[\Psi\mathbb{\tilde{E}}\left[M_{s_{1}}^{t}\left(\phi\right)|\mathcal{Y}_{s_{2}}^{0,V,t}\right]\int_{s_{2}}^{s_{3}}\theta_{u}d\overleftarrow{Y_{u}^{i}}\right]\\
 & =\mathbb{\tilde{E}}\left[\Psi M_{s_{2}}^{t}\left(\phi\right)\int_{s_{2}}^{s_{3}}\theta_{u}d\overleftarrow{Y_{u}^{i}}\right]\\
 & =\mathbb{\tilde{E}}\left[\Psi M_{s_{3}}^{t}\left(\phi\right)\int_{s_{2}}^{s_{3}}\theta_{u}d\overleftarrow{Y_{u}^{i}}\right]\\
 & \quad+\sum_{i_{1}=1}^{d_{Y}}\mathbb{\tilde{E}}\left[\Psi\left(\int_{s_{2}}^{s_{3}}M_{u}^{t}\left(\phi\right)\phi_{i_{1}}\left(X_{u}\right)d\overleftarrow{Y_{u}^{i_{1}}}\right)\left(\int_{s_{2}}^{s_{3}}\theta_{u}d\overleftarrow{Y_{u}^{i}}\right)\right]
\end{align*}
Next, note that 
\[
\mathbb{\tilde{E}}\left[\Psi M_{s_{3}}^{t}\left(\phi\right)\int_{s_{2}}^{s_{3}}\theta_{u}d\overleftarrow{Y_{u}^{i}}\right]=\mathbb{\tilde{E}}\left[\Psi M_{s_{3}}^{t}\left(\phi\right)\mathbb{\tilde{E}}\left[\int_{s_{2}}^{s_{3}}\theta_{u}d\overleftarrow{Y_{u}^{i}}|\mathcal{Y}_{s_{3}}^{0,V,t}\right]\right]=0,
\]
and using equation $\left(\ref{eq: EquBackMartVariation}\right)$
we have 
\begin{align*}
 & \mathbb{\tilde{E}}\left[\Psi\left(\int_{s_{2}}^{s_{3}}M_{u}^{t}\left(\phi\right)\phi_{i_{1}}\left(X_{u}\right)d\overleftarrow{Y_{u}^{i_{1}}}\right)\left(\int_{s_{2}}^{s_{3}}\theta_{u}d\overleftarrow{Y_{u}^{i}}\right)\right]\\
 & =\mathbf{1}_{\left\{ i_{1}=i\right\} }\mathbb{\tilde{E}}\left[\Psi\left(\int_{s_{2}}^{s_{3}}M_{u}^{t}\left(\phi\right)\phi_{i_{1}}\left(X_{u}\right)\theta_{u}du\right)\right]\\
 & =\mathbf{1}_{\left\{ i_{1}=i\right\} }\mathbb{\tilde{E}}\left[\Psi\left(\int_{s_{2}}^{s_{3}}\mathbb{\tilde{E}}\left[M_{u}^{t}\left(\phi\right)|\mathcal{Y}_{s_{3}}^{0,V,t}\right]\phi_{i_{1}}\left(X_{u}\right)\theta_{u}du\right)\right]\\
 & =\mathbf{1}_{\left\{ i_{1}=i\right\} }\mathbb{\tilde{E}}\left[\Psi M_{s_{3}}^{t}\left(\phi\right)\left(\int_{s_{2}}^{s_{3}}\phi_{i_{1}}\left(X_{u}\right)\theta_{u}du\right)\right].
\end{align*}
Hence, the result follows. 
\end{proof}
\begin{lem}
\label{lem: Backward Martingale II}Let $0\leq s_{1}\leq s_{2}\leq s_{3}\leq t$,
$\Psi$ be $\mathcal{Y}_{s_{3}}^{0,V,t}$-measurable random variable
and $\theta^{1}=\left\{ \theta_{u}^{1}\right\} _{0\leq u\leq t}$
and $\theta^{1}=\left\{ \theta_{u}^{1}\right\} _{0\leq u\leq t}$
be two square integrable measurable processes such that $\theta_{u}^{1}$
and $\theta_{u}^{2}$ are also measurable with respect to $\mathcal{Y}_{s_{3}}^{0,V,t}$
for all $s_{2}\leq u\leq s_{3}$. Then, 
\begin{align*}
 & \mathbb{\tilde{E}}\left[\Psi M_{s_{1}}^{t}\left(\phi\right)\left(\int_{s_{2}}^{s_{3}}\theta_{u}^{1}d\overleftarrow{Y_{u}^{i_{1}}}\right)\left(\int_{s_{2}}^{s_{3}}\theta_{u}^{2}d\overleftarrow{Y_{u}^{i_{2}}}\right)\right]\\
 & =\mathbb{\tilde{E}}\left[\Psi M_{s_{3}}^{t}\left(\phi\right)\int_{s_{2}}^{s_{3}}\phi_{i_{1}}\left(X_{u}\right)\theta_{u}^{1}\left(\int_{u}^{s_{3}}\phi_{i_{2}}\left(X_{v}\right)\theta_{v}^{2}dv\right)du\right]\\
 & \quad+\mathbb{\tilde{E}}\left[\Psi M_{s_{3}}^{t}\left(\phi\right)\int_{s_{2}}^{s_{3}}\phi_{i_{2}}\left(X_{u}\right)\theta_{u}^{2}\left(\int_{u}^{s_{3}}\phi_{i_{1}}\left(X_{v}\right)\theta_{v}^{1}dv\right)du\right]\\
 & \quad+\mathbf{1}_{\left\{ i_{1}=i_{2}\right\} }\mathbb{\tilde{E}}\left[\Psi M_{s_{3}}^{t}\left(\phi\right)\int_{s_{2}}^{s_{3}}\theta_{u}^{1}\theta_{u}^{2}du\right].
\end{align*}
\end{lem}
\begin{proof}
Using the integration by parts formula $\left(\ref{eq: Backward_IBP}\right)$
we can write 
\begin{align*}
\mathbb{\tilde{E}}\left[\Psi M_{s_{1}}^{t}\left(\phi\right)\left(\int_{s_{2}}^{s_{3}}\theta_{u}^{1}d\overleftarrow{Y_{u}^{i_{1}}}\right)\left(\int_{s_{2}}^{s_{3}}\theta_{u}^{2}d\overleftarrow{Y_{u}^{i_{2}}}\right)\right] & =\mathbb{\tilde{E}}\left[\Psi M_{s_{1}}^{t}\left(\phi\right)\int_{s_{2}}^{s_{3}}\theta_{u}^{1}\left(\int_{u}^{s_{3}}\theta_{v}^{2}d\overleftarrow{Y_{v}^{i_{2}}}\right)d\overleftarrow{Y_{u}^{i_{1}}}\right]\\
 & \quad+\mathbb{\tilde{E}}\left[\Psi M_{s_{1}}^{t}\left(\phi\right)\int_{s_{2}}^{s_{3}}\theta_{u}^{2}\left(\int_{u}^{s_{3}}\theta_{v}^{1}d\overleftarrow{Y_{v}^{i_{1}}}\right)d\overleftarrow{Y_{u}^{i_{2}}}\right]\\
 & \quad+\mathbf{1}_{\left\{ i_{1}=i_{2}\right\} }\mathbb{\tilde{E}}\left[\Psi M_{s_{1}}^{t}\left(\phi\right)\int_{s_{2}}^{s_{3}}\theta_{u}^{1}\theta_{u}^{2}du\right].
\end{align*}
Then, using the same reasonings as in Lemma \ref{lem: Backward Martingale}
we get that 
\begin{align*}
\mathbb{\tilde{E}}\left[\Psi M_{s_{1}}^{t}\left(\phi\right)\int_{s_{2}}^{s_{3}}\theta_{u}^{1}\left(\int_{u}^{s_{3}}\theta_{v}^{2}d\overleftarrow{Y_{v}^{i_{2}}}\right)d\overleftarrow{Y_{u}^{i_{1}}}\right] & =\mathbb{\tilde{E}}\left[\Psi\int_{s_{2}}^{s_{3}}M_{u}^{t}\left(\phi\right)\phi_{i_{1}}\left(X_{u}\right)\theta_{u}^{1}\left(\int_{u}^{s_{3}}\theta_{v}^{2}d\overleftarrow{Y_{v}^{i_{2}}}\right)du\right]
\end{align*}
Next, by Fubini's theorem, Lemma \ref{lem: Backward Martingale} and
Fubini's theorem again we obtain 
\begin{align*}
 & \mathbb{\tilde{E}}\left[\Psi\int_{s_{2}}^{s_{3}}M_{u}^{t}\left(\phi\right)\phi_{i_{1}}\left(X_{u}\right)\theta_{u}^{1}\left(\int_{u}^{s_{3}}\theta_{v}^{2}d\overleftarrow{Y_{v}^{i_{2}}}\right)du\right]\\
 & =\int_{s_{2}}^{s_{3}}\mathbb{\tilde{E}}\left[\Psi\phi_{i_{1}}\left(X_{u}\right)\theta_{u}^{1}M_{u}^{t}\left(\phi\right)\left(\int_{u}^{s_{3}}\theta_{v}^{2}d\overleftarrow{Y_{v}^{i_{2}}}\right)\right]du\\
 & =\int_{s_{2}}^{s_{3}}\mathbb{\tilde{E}}\left[\Psi\phi_{i_{1}}\left(X_{u}\right)\theta_{u}^{1}M_{s_{3}}^{t}\left(\phi\right)\left(\int_{u}^{s_{3}}\phi_{i_{2}}\left(X_{u}\right)\theta_{v}^{2}dv\right)\right]du\\
 & =\mathbb{\tilde{E}}\left[\Psi M_{s_{3}}^{t}\left(\phi\right)\int_{s_{2}}^{s_{3}}\phi_{i_{1}}\left(X_{u}\right)\theta_{u}^{1}\left(\int_{u}^{s_{3}}\phi_{i_{2}}\left(X_{u}\right)\theta_{v}^{2}dv\right)\right]du
\end{align*}
By symmetry we get an analogous expression for the term $\int_{s_{2}}^{s_{3}}\theta_{u}^{2}\left(\int_{u}^{s_{3}}\theta_{v}^{1}d\overleftarrow{Y_{v}^{i_{1}}}\right)d\overleftarrow{Y_{u}^{i_{2}}}$
. Finally, for the last term we only need to take conditional expectation
with respect to $\mathcal{Y}_{s_{3}}^{0,V,t}$ and use that $M_{s}^{t}\left(\phi\right)$
is a $\mathbb{Y}^{0,V,t}$-martingale. 
\end{proof}
The next lemma is a well known generalization of Hölder's inequality. 
\begin{lem}[Generalized Hölder's inequality]
\label{lem: Generalized Holder}Let $p_{i}>1,i=1,...,m$ such that
$\sum_{i=1}^{m}\frac{1}{p_{i}}=1,$ and $X_{i}\in L^{p_{i}}(\Omega,\mathcal{F},\tilde{P}),i=1,...,m$.
Then, 
\[
\tilde{\mathbb{E}}\left[\left|\prod_{i=1}^{m}X_{i}\right|\right]\leq\prod_{i=1}^{m}\tilde{\mathbb{E}}\left[\left|X_{i}\right|^{p_{i}}\right]^{1/p_{i}}<\infty.
\]
\end{lem}
\begin{lem}
\label{lem: Main Backward}Let $\tau=\left\{ 0=t_{0}<t_{1}<\cdots<t_{n}=t\right\} $
be a partition of $\left[0,t\right]$, $\phi\in\mathcal{B}_{b}\left(\mathbb{R}^{d_{X}};\mathbb{R}^{d_{Y}}\right)$,
$\Upsilon$$\in L^{p}(\Omega,\mathcal{Y}_{t}^{0,V,t},\tilde{P})$
for any $p\geq1$, $\beta_{s},$ be a deterministic processes satisfying
\begin{equation}
\left|\beta_{s}\mathbf{1}_{\left[t_{j},t_{j+1}\right]}\left(s\right)\right|\leq\delta^{m}\label{eq: BoundBeta}
\end{equation}
for some $m\in\mathbb{N}$ and $\theta^{1},\theta^{2},\kappa^{1},\kappa^{2},$
be stochastic processes measurable with respect to $\mathcal{Y}_{t}^{0,V,t}$,
such that 
\begin{align}
\sup_{0\leq s\leq t}\mathbb{\tilde{E}}\left[\left|\theta_{s}^{j}\right|^{p}\right] & <\infty,\quad j=1,2,\label{eq: MomentsTheta}\\
\sup_{0\leq s\leq t}\mathbb{\tilde{E}}\left[\left|\kappa_{s}^{j}\right|^{p}\right] & <\infty,\quad j=1,2,\nonumber 
\end{align}
for any $p\geq1$. Then: 
\begin{enumerate}
\item For $i\in\left\{ 1,...,d_{Y}\right\} $, we have that 
\[
\left|\mathbb{\tilde{E}}\left[\Upsilon M_{0}^{t}\left(\phi\right)\left(\int_{t_{j}}^{t_{j+1}}\beta_{s}d\overleftarrow{Y_{s}^{i_{1}}}\right)\left(\int_{t_{k}}^{t_{k+1}}\beta_{s}d\overleftarrow{Y_{s}^{i_{1}}}\right)\right]\right|\leq C\left\{ \mathbf{1}_{\left\{ j\neq k\right\} }\delta^{2m+2}+\mathbf{1}_{\left\{ j=k\right\} }\delta^{2m+1}\right\} .
\]
\item For $i,i_{1}\in\left\{ 1,...,d_{Y}\right\} $, we have that 
\begin{align*}
 & \left|\mathbb{\tilde{E}}\left[\Upsilon M_{0}^{t}\left(\phi\right)\left(\int_{t_{k+1}}^{t}\theta_{s}^{1}d\overleftarrow{Y_{s}^{i_{1}}}\right)\left(\int_{t_{j}}^{t_{j+1}}\beta_{s}d\overleftarrow{Y_{s}^{i}}\right)\left(\int_{t_{k}}^{t_{k+1}}\beta_{s}d\overleftarrow{Y_{s}^{i}}\right)\right]\right|\\
 & \leq C\left\{ \mathbf{1}_{\left\{ j\neq k\right\} }\delta^{2m+2}+\mathbf{1}_{\left\{ j=k\right\} }\delta^{2m+1}\right\} .
\end{align*}
\item For $i,i_{1},a_{1}\in\left\{ 1,...,d_{Y}\right\} $, we have that
\begin{align*}
 & \left|\mathbb{\tilde{E}}\left[\Upsilon M_{0}^{t}\left(\phi\right)\left(\int_{t_{j+1}}^{t}\theta_{s}^{1}d\overleftarrow{Y_{s}^{i_{1}}}\right)\left(\int_{t_{k+1}}^{t}\kappa_{s}^{1}d\overleftarrow{Y_{s}^{a_{1}}}\right)\left(\int_{t_{j}}^{t_{j+1}}\beta_{s}d\overleftarrow{Y_{s}^{i}}\right)\left(\int_{t_{k}}^{t_{k+1}}\beta_{s}d\overleftarrow{Y_{s}^{i}}\right)\right]\right|\\
 & \leq C\left\{ \mathbf{1}_{\left\{ j\neq k\right\} }\delta^{2m+2}+\mathbf{1}_{\left\{ j=k\right\} }\delta^{2m+1}\right\} .
\end{align*}
\item For $i,i_{1},i_{2}a_{1},a_{2}\in\left\{ 1,...,d_{Y}\right\} $, we
have that 
\begin{align*}
 & \left|\mathbb{\tilde{E}}\left[\Upsilon M_{0}^{t}\left(\phi\right)\prod_{l=1}^{2}\left\{ \left(\int_{t_{j+1}}^{t}\theta_{s}^{l}d\overleftarrow{Y_{s}^{i_{l}}}\right)\left(\int_{t_{k+1}}^{t}\kappa_{s}^{l}d\overleftarrow{Y_{s}^{a_{l}}}\right)\right\} \left(\int_{t_{j}}^{t_{j+1}}\beta_{s}d\overleftarrow{Y_{s}^{i}}\right)\left(\int_{t_{k}}^{t_{k+1}}\beta_{s}d\overleftarrow{Y_{s}^{i}}\right)\right]\right|\\
 & \leq C\left\{ \mathbf{1}_{\left\{ j\neq k\right\} }\delta^{2m+2}+\mathbf{1}_{\left\{ j=k\right\} }\delta^{2m+1}\right\} .
\end{align*}
\end{enumerate}
\end{lem}
\begin{proof}
The full proof of the lemma is lengthy and depends on applying Lemmas
\ref{lem: Backward Martingale} and \ref{lem: Backward Martingale II}
repeatedly. We do not present it in full, but only write in detail
the proof of the statement $\left(4\right)$, the others being similar
and easier. Note that by the assumptions on $\Upsilon$,$\beta^{j}$
and $\theta^{j}$ and $\phi$ the expectations in the statement of
the lemma are finite. We start with some preliminary estimations.
In what follows $C\left(t\right)$,$C$$\left(\phi\right)$and $C\left(t,\phi\right)$
will denote constants that only depends on $t$, on $\phi$ and on
$t$ and $\phi$, respectively. For any $p$$\geq1$ we have 
\begin{itemize}
\item Let $0\leq s\leq t$, then 
\begin{align*}
\mathbb{\tilde{E}}\left[\left|M_{s}^{t}\left(\phi\right)\right|^{p}\right] & =\mathbb{\tilde{E}}\left[M_{s}^{t}\left(p\phi\right)\exp\left(\frac{p^{2}-p}{2}\sum_{i=1}^{d_{Y}}\int_{s}^{t}\phi_{i}^{2}\left(X_{u}\right)du\right)\right]\\
 & \leq\exp\left(d_{Y}\left\Vert \phi\right\Vert _{\infty}^{2}\frac{p^{2}-p}{2}\left(t-s\right)\right),
\end{align*}
where we have used that for $\phi\in\mathcal{B}_{b}\left(\mathbb{R}^{d_{X}};\mathbb{R}^{d_{Y}}\right)$
one has that $M_{s}^{t}\left(\phi\right)$ is a backward martingale
with expectation equal to one. The previous estimate yields that 
\begin{equation}
\sup_{0\leq s\leq t}\mathbb{\tilde{E}}\left[\left|M_{s}^{t}\left(\phi\right)\right|^{p}\right]\leq\exp\left(d_{Y}\left\Vert \phi\right\Vert _{\infty}^{2}\frac{p^{2}-p}{2}t\right)\leq C\left(t,\phi\right).\label{eq:BE_M}
\end{equation}
\item Let $0\leq s_{1}\leq s_{2}\leq t$, then 
\begin{align}
\mathbb{\tilde{E}}\left[\left|\int_{s_{1}}^{s_{2}}\theta_{u}\overleftarrow{Y_{u}^{i}}\right|^{p}\right] & \leq\mathbb{\tilde{E}}\left[\left(\int_{s_{1}}^{s_{2}}\left|\theta_{u}\right|^{2}du\right)^{p/2}\right]\nonumber \\
 & \leq\mathbb{\tilde{E}}\left[\left(s_{2}-s_{1}\right)^{p/2-1}\left(\int_{s_{1}}^{s_{2}}\left|\theta_{u}\right|^{p}du\right)\right]\nonumber \\
 & \leq\sup_{0\leq s\leq t}\mathbb{\tilde{E}}\left[\left|\theta_{s}\right|^{p}\right]\left(s_{2}-s_{1}\right)^{p/2}\nonumber \\
 & \leq C\left(s_{2}-s_{1}\right)^{p/2},\label{eq: BE_theta_delta}
\end{align}
where we have used the Burkholder-Davis-Gundy inequality for backward
martingales, Jensen's inequality and Fubini's theorem. The previous
estimate yields that 
\begin{equation}
\sup_{0\leq s_{1}\leq s_{2}\leq t}\mathbb{\tilde{E}}\left[\left|\int_{s_{1}}^{s_{2}}\theta_{u}\overleftarrow{Y_{u}^{i}}\right|^{p}\right]\leq\sup_{0\leq s\leq t}\mathbb{\tilde{E}}\left[\left|\theta_{s}\right|^{p}\right]t^{p/2}\leq C\left(t\right).\label{eq: BE_theta}
\end{equation}
Moreover, using Jensen's inequality and Fubini's theorem we have that
\begin{align*}
\mathbb{\tilde{E}}\left[\left|\int_{s_{1}}^{s_{2}}\phi_{i}\left(X_{s}\right)\theta du\right|^{p}\right] & \leq\left\Vert \phi\right\Vert _{\infty}^{p}\mathbb{\tilde{E}}\left[\left|\int_{s_{1}}^{s_{2}}\left|\theta\right|du\right|^{p}\right]\\
 & \leq\left\Vert \phi\right\Vert _{\infty}^{p}\mathbb{\tilde{E}}\left[\left(s_{2}-s_{1}\right)^{p-1}\int_{s_{1}}^{s_{2}}\left|\theta\right|du\right]\\
 & \leq\left\Vert \phi\right\Vert _{\infty}^{p}\sup_{0\leq s\leq t}\mathbb{\tilde{E}}\left[\left|\theta_{s}\right|^{p}\right]\left(s_{2}-s_{1}\right)^{p}.
\end{align*}
The previous estimate yields that 
\begin{equation}
\mathbb{\tilde{E}}\left[\left|\int_{s_{1}}^{s_{2}}\phi_{i}\left(X_{s}\right)\theta du\right|^{p}\right]\leq\left\Vert \phi\right\Vert _{\infty}^{p}\sup_{0\leq s\leq t}\mathbb{\tilde{E}}\left[\left|\theta_{s}\right|^{p}\right]t^{p/2}\leq C\left(t,\phi\right).\label{eq: E_FiTheta}
\end{equation}
\item Let $0\leq s_{1}\leq s_{2}\leq t$, then similar reasonings as in
the previous point and hypothesis $\left(\ref{eq: BoundBeta}\right)$
give 
\begin{align}
\mathbb{\tilde{E}}\left[\left|\int_{s_{1}}^{s_{2}}\beta_{u}\overleftarrow{Y_{u}^{i}}\right|^{p}\right] & \leq\mathbb{\tilde{E}}\left[\left(\int_{s_{1}}^{s_{2}}\left|\beta_{u}\right|^{2}du\right)^{p/2}\right]\nonumber \\
 & \leq\left(s_{2}-s_{1}\right)^{p/2-1}\int_{s_{1}}^{s_{2}}\left|\beta_{u}\right|^{p}du\nonumber \\
 & \leq t^{p/2-1}\sum_{j=1}^{n}\int_{t_{j-1}}^{t_{j}}\left|\beta_{u}\right|^{p}du\nonumber \\
 & \leq t^{p/2}\delta^{mp}\leq C\left(t\right)\label{eq: EB_Beta}
\end{align}
Moreover, if $s_{1}=t_{j}$ and $s_{2}=t_{j+1}$for some $j\in\{0,...,n-1\}$
we can conclude using hypothesis $\left(\ref{eq: BoundBeta}\right)$
that 
\begin{equation}
\mathbb{\tilde{E}}\left[\left|\int_{s_{1}}^{s_{2}}\beta_{u}\overleftarrow{Y_{u}^{i}}\right|^{p}\right]\leq\delta^{p/2-1}\int_{t_{j}}^{t_{j+1}}\delta^{mp}du\leq\delta^{\left(m+\frac{1}{2}\right)p}.\label{eq: EB_Beta_delta}
\end{equation}
Finally, 
\begin{equation}
\mathbb{E}\left[\left|\int_{t_{j}}^{t_{j+1}}\phi_{i}\left(X_{s}\right)\beta_{s}ds\right|^{p}\right]\leq\left\Vert \phi\right\Vert _{\infty}^{p}\delta^{\left(m+1\right)p}=C\left(\phi\right)\delta^{\left(m+1\right)p}\label{eq: E_Beta_Fi}
\end{equation}
\end{itemize}
\textbf{\lyxarrow{}Case $j$$=k$:}

Using Cauchy-Schwarz inequality, Lemma \ref{lem: Generalized Holder}
and inequalities $\left(\ref{eq:BE_M}\right)$, $\left(\ref{eq: BE_theta}\right)$
and $\left(\ref{eq: EB_Beta_delta}\right)$ we can write 
\begin{align*}
 & \mathbb{\tilde{E}}\left[\Upsilon M_{0}^{t}\left(\phi\right)\prod_{l=1}^{2}\left\{ \left(\int_{t_{j+1}}^{t}\theta_{s}^{l}d\overleftarrow{Y_{s}^{i_{l}}}\right)\left(\int_{t_{j+1}}^{t}\kappa_{s}^{l}d\overleftarrow{Y_{s}^{a_{l}}}\right)\right\} \left(\int_{t_{j}}^{t_{j+1}}\beta_{s}d\overleftarrow{Y_{s}^{i}}\right)^{2}\right]\\
 & \leq\mathbb{\tilde{E}}\left[\left|\Upsilon M_{0}^{t}\left(\phi\right)\prod_{l=1}^{2}\left\{ \left(\int_{t_{j+1}}^{t}\theta_{s}^{l}d\overleftarrow{Y_{s}^{i_{l}}}\right)\left(\int_{t_{j+1}}^{t}\kappa_{s}^{l}d\overleftarrow{Y_{s}^{a_{l}}}\right)\right\} \right|^{2}\right]^{1/2}\mathbb{\tilde{E}}\left[\left(\int_{t_{j}}^{t_{j+1}}\beta_{s}d\overleftarrow{Y_{s}^{i_{1}}}\right)^{4}\right]^{1/2}\\
 & \leq\left(\mathbb{\tilde{E}}\left[\left|\Upsilon\right|^{12}\right]\mathbb{\tilde{E}}\left[\left|M_{0}^{t}\left(\phi\right)\right|^{12}\right]\prod_{l=1}^{2}\mathbb{\tilde{E}}\left[\left|\int_{t_{j+1}}^{t}\theta_{s}^{l}d\overleftarrow{Y_{s}^{i_{l}}}\right|^{12}\right]\mathbb{\tilde{E}}\left[\left|\int_{t_{j+1}}^{t}\kappa_{s}^{l}d\overleftarrow{Y_{s}^{a_{l}}}\right|^{12}\right]\right)^{1/12}\delta^{\frac{1}{2}\left(m+\frac{1}{2}\right)4}\\
 & \leq C\left(t\right)\delta^{2m+1}
\end{align*}
\textbf{\lyxarrow{}Case $j$$<k$:}

Using Lemma \ref{lem: Backward Martingale}, we can write 
\begin{align*}
 & \mathbb{\tilde{E}}\left[\Upsilon M_{0}^{t}\left(\phi\right)\prod_{l=1}^{2}\left\{ \left(\int_{t_{j+1}}^{t}\theta_{s}^{l}d\overleftarrow{Y_{s}^{i_{l}}}\right)\left(\int_{t_{k+1}}^{t}\kappa_{s}^{l}d\overleftarrow{Y_{s}^{a_{l}}}\right)\right\} \left(\int_{t_{j}}^{t_{j+1}}\beta_{s}d\overleftarrow{Y_{s}^{i}}\right)\left(\int_{t_{k}}^{t_{k+1}}\beta_{s}d\overleftarrow{Y_{s}^{i}}\right)\right]\\
 & =\mathbb{\tilde{E}}\left[\Upsilon\left(\int_{t_{k}}^{t_{k+1}}\beta_{s}d\overleftarrow{Y_{s}^{i}}\right)\prod_{l=1}^{2}\left\{ \left(\int_{t_{j+1}}^{t}\theta_{s}^{l}d\overleftarrow{Y_{s}^{i_{l}}}\right)\left(\int_{t_{k+1}}^{t}\kappa_{s}^{l}d\overleftarrow{Y_{s}^{a_{l}}}\right)\right\} \right.\\
 & \qquad\left.\times M_{t_{j+1}}^{t}\left(\phi\right)\left(\int_{t_{j}}^{t_{j+1}}\phi_{i}\left(X_{s}\right)\beta_{s}ds\right)\right]\\
 & =\mathbb{\tilde{E}}\left[\Upsilon_{1}M_{t_{j+1}}^{t}\left(\phi\right)\prod_{l=1}^{2}\left\{ \left(\int_{t_{j+1}}^{t}\theta_{s}^{l}d\overleftarrow{Y_{s}^{i_{l}}}\right)\right\} \right]\\
 & \triangleq\sum_{i=1}^{9}A_{i},
\end{align*}
Where 
\[
\Upsilon_{1}=\Upsilon\left(\int_{t_{k}}^{t_{k+1}}\beta_{s}d\overleftarrow{Y_{s}^{i}}\right)\left(\int_{t_{j}}^{t_{j+1}}\phi_{i}\left(X_{s}\right)\beta_{s}ds\right)\prod_{l=1}^{2}\left(\int_{t_{k+1}}^{t}\kappa_{s}^{l}d\overleftarrow{Y_{s}^{a_{l}}}\right),
\]
and 
\begin{align*}
A_{1} & =\mathbb{\tilde{E}}\left[\Upsilon_{1}M_{t_{j+1}}^{t}\left(\phi\right)\left(\int_{t_{j+1}}^{t_{k}}\theta_{s}^{1}d\overleftarrow{Y_{s}^{i_{1}}}\right)\left(\int_{t_{j+1}}^{t_{k}}\theta_{s}^{2}d\overleftarrow{Y_{s}^{i_{2}}}\right)\right]\\
A_{2} & =\mathbb{\tilde{E}}\left[\Upsilon_{1}M_{t_{j+1}}^{t}\left(\phi\right)\left(\int_{t_{j+1}}^{t_{k}}\theta_{s}^{1}d\overleftarrow{Y_{s}^{i_{1}}}\right)\left(\int_{t_{k}}^{t_{k+1}}\theta_{s}^{2}d\overleftarrow{Y_{s}^{i_{2}}}\right)\right]\\
A_{3} & =\mathbb{\tilde{E}}\left[\Upsilon_{1}M_{t_{j+1}}^{t}\left(\phi\right)\left(\int_{t_{j+1}}^{t_{k}}\theta_{s}^{1}d\overleftarrow{Y_{s}^{i_{1}}}\right)\left(\int_{t_{k+1}}^{t}\theta_{s}^{2}d\overleftarrow{Y_{s}^{i_{2}}}\right)\right]\\
A_{4} & =\mathbb{\tilde{E}}\left[\Upsilon_{1}M_{t_{j+1}}^{t}\left(\phi\right)\left(\int_{t_{k}}^{t_{k+1}}\theta_{s}^{1}d\overleftarrow{Y_{s}^{i_{1}}}\right)\left(\int_{t_{j+1}}^{t_{k}}\theta_{s}^{2}d\overleftarrow{Y_{s}^{i_{2}}}\right)\right]\\
A_{5} & =\mathbb{\tilde{E}}\left[\Upsilon_{1}M_{t_{j+1}}^{t}\left(\phi\right)\left(\int_{t_{k}}^{t_{k+1}}\theta_{s}^{1}d\overleftarrow{Y_{s}^{i_{1}}}\right)\left(\int_{t_{k}}^{t_{k+1}}\theta_{s}^{2}d\overleftarrow{Y_{s}^{i_{2}}}\right)\right]\\
A_{6} & =\mathbb{\tilde{E}}\left[\Upsilon_{1}M_{t_{j+1}}^{t}\left(\phi\right)\left(\int_{t_{k}}^{t_{k+1}}\theta_{s}^{1}d\overleftarrow{Y_{s}^{i_{1}}}\right)\left(\int_{t_{k+1}}^{t}\theta_{s}^{2}d\overleftarrow{Y_{s}^{i_{2}}}\right)\right]\\
A_{7} & =\mathbb{\tilde{E}}\left[\Upsilon_{1}M_{t_{j+1}}^{t}\left(\phi\right)\left(\int_{t_{k+1}}^{t}\theta_{s}^{1}d\overleftarrow{Y_{s}^{i_{1}}}\right)\left(\int_{t_{j+1}}^{t_{k}}\theta_{s}^{2}d\overleftarrow{Y_{s}^{i_{2}}}\right)\right]\\
A_{8} & =\mathbb{\tilde{E}}\left[\Upsilon_{1}M_{t_{j+1}}^{t}\left(\phi\right)\left(\int_{t_{k+1}}^{t}\theta_{s}^{1}d\overleftarrow{Y_{s}^{i_{1}}}\right)\left(\int_{t_{k}}^{t_{k+1}}\theta_{s}^{2}d\overleftarrow{Y_{s}^{i_{2}}}\right)\right]\\
A_{9} & =\mathbb{\tilde{E}}\left[\Upsilon_{1}M_{t_{j+1}}^{t}\left(\phi\right)\left(\int_{t_{k+1}}^{t}\theta_{s}^{1}d\overleftarrow{Y_{s}^{i_{1}}}\right)\left(\int_{t_{k+1}}^{t}\theta_{s}^{2}d\overleftarrow{Y_{s}^{i_{2}}}\right)\right]
\end{align*}
The treatment of some of the terms is completely analogous. We distinguish
four subcases:

\textbf{\lyxarrow{}\lyxarrow{}Subcase} $1$: This subcase covers
term $A_{9}.$ We apply Lemma \ref{lem: Backward Martingale II} to
write 
\begin{align*}
A_{9} & =\mathbb{\tilde{E}}\left[\Upsilon_{1}M_{t_{j+1}}^{t}\left(\phi\right)\left(\int_{t_{k+1}}^{t}\theta_{s}^{1}d\overleftarrow{Y_{s}^{i_{1}}}\right)\left(\int_{t_{k+1}}^{t}\theta_{s}^{2}d\overleftarrow{Y_{s}^{i_{2}}}\right)\right]\\
 & =\mathbb{\tilde{E}}\left[\Gamma_{1}M_{t_{k+1}}^{t}\left(\phi\right)\left(\int_{t_{j}}^{t_{j+1}}\phi_{i}\left(X_{s}\right)\beta_{s}ds\right)\left(\int_{t_{k}}^{t_{k+1}}\phi_{i}\left(X_{s}\right)\beta_{s}ds\right)\right],
\end{align*}
where 
\[
\Gamma_{1}\triangleq\Upsilon\prod_{l=1}^{2}\left(\int_{t_{k+1}}^{t}\kappa_{s}^{l}d\overleftarrow{Y_{s}^{a_{l}}}\right)\left(\int_{t_{k+1}}^{t}\theta_{s}^{l}d\overleftarrow{Y_{s}^{i_{l}}}\right).
\]
Hence, using Lemma \ref{lem: Generalized Holder} and inequality $\left(\ref{eq: BE_theta}\right)$we
have, for $p\geq1,$ that 
\[
\mathbb{E}\left[\left|\Gamma_{1}\right|^{p}\right]\leq C\left(t\right),
\]
and using Lemma \ref{lem: Generalized Holder} and inequalities $\left(\ref{eq:BE_M}\right)$
and $\left(\ref{eq: E_Beta_Fi}\right)$ we obtain 
\begin{align*}
\left|A_{9}\right| & \leq\mathbb{\tilde{E}}\left[\left|\Gamma_{1}\right|^{4}\right]^{1/4}\mathbb{\tilde{E}}\left[\left|M_{t_{k+1}}^{t}\left(\phi\right)\right|^{4}\right]^{1/4}\mathbb{E}\left[\left|\int_{t_{j}}^{t_{j+1}}\phi_{i}\left(X_{s}\right)\beta_{s}ds\right|^{4}\right]^{1/4}\\
 & \qquad\times\mathbb{E}\left[\left|\int_{t_{k}}^{t_{k+1}}\phi_{i}\left(X_{s}\right)\beta_{s}ds\right|^{4}\right]^{1/4}\\
 & \leq C\left(t,\phi\right)\delta^{2m+2}.
\end{align*}

\textbf{\lyxarrow{}\lyxarrow{}Subcase} $2$: The terms $A_{2},A_{4},A_{5},A_{6}$
and $A_{8}$ are treated analogously. We will write the proof for
$A_{2}.$ By Lemma \ref{lem: Backward Martingale II}, we can write
\begin{align*}
A_{2} & =\mathbb{\tilde{E}}\left[\Upsilon_{1}M_{t_{j+1}}^{t}\left(\phi\right)\left(\int_{t_{j+1}}^{t_{k}}\theta_{s}^{1}d\overleftarrow{Y_{s}^{i_{1}}}\right)\left(\int_{t_{k}}^{t_{k+1}}\theta_{s}^{2}d\overleftarrow{Y_{s}^{i_{2}}}\right)\right]\\
 & =\mathbb{\tilde{E}}\left[\Gamma_{2}M_{t_{j+1}}^{t}\left(\phi\right)\left(\int_{t_{j}}^{t_{j+1}}\phi_{i}\left(X_{s}\right)\beta_{s}ds\right)\left(\int_{t_{k}}^{t_{k+1}}\beta_{s}d\overleftarrow{Y_{s}^{i}}\right)\left(\int_{t_{k}}^{t_{k+1}}\theta_{s}^{2}d\overleftarrow{Y_{s}^{i_{2}}}\right)\right],
\end{align*}
where 
\[
\Gamma_{2}\triangleq\Upsilon\left(\int_{t_{j+1}}^{t_{k}}\theta_{s}^{1}d\overleftarrow{Y_{s}^{i_{1}}}\right)\left(\int_{t_{k}}^{t_{k+1}}\theta_{s}^{2}d\overleftarrow{Y_{s}^{i_{2}}}\right)\prod_{l=1}^{2}\left(\int_{t_{k+1}}^{t}\kappa_{s}^{l}d\overleftarrow{Y_{s}^{a_{l}}}\right).
\]
Hence, using Lemma \ref{lem: Generalized Holder} and inequality $\left(\ref{eq: BE_theta}\right)$we
have, for $p\geq1,$ that 
\[
\mathbb{E}\left[\left|\Gamma_{2}\right|^{p}\right]\leq C\left(t\right),
\]
and using Lemma \ref{lem: Generalized Holder} and inequalities $\left(\ref{eq:BE_M}\right)$,$\left(\ref{eq: E_Beta_Fi}\right)$,
$\left(\ref{eq: EB_Beta_delta}\right)$ and $\left(\ref{eq: BE_theta_delta}\right)$
we obtain 
\begin{align*}
\left|A_{2}\right| & \leq\mathbb{E}\left[\left|\Gamma_{2}\right|^{5}\right]^{1/5}\mathbb{E}\left[\left|M_{t_{j+1}}^{t}\left(\phi\right)\right|^{5}\right]^{1/5}\mathbb{E}\left[\left|\int_{t_{j}}^{t_{j+1}}\phi_{i}\left(X_{s}\right)\beta_{s}ds\right|^{5}\right]^{1/5}\\
 & \qquad\times\mathbb{E}\left[\left|\int_{t_{k}}^{t_{k+1}}\beta_{s}d\overleftarrow{Y_{s}^{i}}\right|^{5}\right]^{1/5}\mathbb{E}\left[\left|\int_{t_{k}}^{t_{k+1}}\theta_{s}^{2}d\overleftarrow{Y_{s}^{i_{2}}}\right|^{5}\right]^{1/5}\\
 & \leq C\left(t,\phi\right)\delta^{\left(m+1\right)}\delta^{\left(m+\frac{1}{2}\right)}\delta^{1/2}=C\left(t,\phi\right)\delta^{2m+2}.
\end{align*}

\textbf{\lyxarrow{}\lyxarrow{}Subcase} 3: The terms $A_{3}$and
$A_{7}$ are treated analogously. We will write the proof for $A_{3}$.
We apply Lemma \ref{lem: Backward Martingale II} twice to write 
\begin{align*}
A_{3} & =\mathbb{\tilde{E}}\left[\Upsilon_{1}M_{t_{j+1}}^{t}\left(\phi\right)\left(\int_{t_{j+1}}^{t_{k}}\theta_{s}^{1}d\overleftarrow{Y_{s}^{i_{1}}}\right)\left(\int_{t_{k+1}}^{t}\theta_{s}^{2}d\overleftarrow{Y_{s}^{i_{2}}}\right)\right]\\
 & =\mathbb{\tilde{E}}\left[\Gamma_{3}\left(\int_{t_{k}}^{t_{k+1}}\beta_{s}d\overleftarrow{Y_{s}^{i}}\right)M_{t_{k}}^{t}\left(\phi\right)\left(\int_{t_{j+1}}^{t_{k}}\phi_{i_{1}}\left(X_{s}\right)\theta_{s}^{1}ds\right)\left(\int_{t_{j}}^{t_{j+1}}\phi_{i}\left(X_{s}\right)\beta_{s}ds\right)\right]\\
 & =\mathbb{\tilde{E}}\left[\Gamma_{3}\left(\int_{t_{j+1}}^{t_{k}}\phi_{i_{1}}\left(X_{s}\right)\theta_{s}^{1}ds\right)M_{t_{k+1}}^{t}\left(\phi\right)\left(\int_{t_{k}}^{t_{k+1}}\phi_{i}\left(X_{s}\right)\beta_{s}ds\right)\left(\int_{t_{j}}^{t_{j+1}}\phi_{i}\left(X_{s}\right)\beta_{s}ds\right)\right]
\end{align*}
where 
\[
\Gamma_{3}=\Upsilon\left(\int_{t_{k+1}}^{t}\theta_{s}^{2}d\overleftarrow{Y_{s}^{i_{2}}}\right)\prod_{l=1}^{2}\left(\int_{t_{k+1}}^{t}\kappa_{s}^{l}d\overleftarrow{Y_{s}^{a_{l}}}\right).
\]
Using Lemma \ref{lem: Generalized Holder} and inequalities $\left(\ref{eq: BE_theta}\right)$
and $\left(\ref{eq: E_FiTheta}\right)$ we have, for $p\geq1,$ that
\[
\mathbb{E}\left[\left|\Gamma_{3}\left(\int_{t_{j+1}}^{t_{k}}\phi_{i_{1}}\left(X_{s}\right)\theta_{s}^{1}ds\right)\right|^{p}\right]\leq C\left(t,\phi\right),
\]
and using Lemma \ref{lem: Generalized Holder} and inequalities $\left(\ref{eq:BE_M}\right)$,$\left(\ref{eq: E_Beta_Fi}\right)$,
$\left(\ref{eq: EB_Beta_delta}\right)$ and $\left(\ref{eq: BE_theta_delta}\right)$
\begin{align*}
\left|A_{3}\right| & \leq\mathbb{\tilde{E}}\left[\left|\Gamma_{3}\left(\int_{t_{j+1}}^{t_{k}}\phi_{i_{1}}\left(X_{s}\right)\theta_{s}^{1}ds\right)\right|^{4}\right]^{1/4}\mathbb{\tilde{E}}\left[\left|M_{t_{k+1}}^{t}\left(\phi\right)\right|^{5}\right]^{1/5}\\
 & \qquad\times\mathbb{\tilde{E}}\left[\left|\int_{t_{k}}^{t_{k+1}}\phi_{i}\left(X_{s}\right)\beta_{s}ds\right|^{4}\right]^{1/4}\mathbb{\tilde{E}}\left[\left|\int_{t_{j}}^{t_{j+1}}\phi_{i}\left(X_{s}\right)\beta_{s}ds\right|^{4}\right]^{1/4}\\
 & \leq C\left(t,\phi\right)\delta^{m+1}\delta^{m+1}=C\delta^{2m+1}.
\end{align*}

\textbf{\lyxarrow{}\lyxarrow{}Subcase} 4: This subcase corresponds
to the term $A_{1}.$ Applying Lemma \ref{lem: Backward Martingale II}
and Lemma \ref{lem: Backward Martingale} we can write 
\begin{align*}
A_{1} & =\mathbb{\tilde{E}}\left[\Upsilon_{1}M_{t_{j+1}}^{t}\left(\phi\right)\left(\int_{t_{j+1}}^{t_{k}}\theta_{s}^{1}d\overleftarrow{Y_{s}^{i_{1}}}\right)\left(\int_{t_{j+1}}^{t_{k}}\theta_{s}^{2}d\overleftarrow{Y_{s}^{i_{2}}}\right)\right]\\
 & =\mathbb{\tilde{E}}\left[\Gamma_{4}\Gamma_{5}M_{t_{k}}^{t}\left(\phi\right)\left(\int_{t_{k}}^{t_{k+1}}\beta_{s}d\overleftarrow{Y_{s}^{i}}\right)\left(\int_{t_{j}}^{t_{j+1}}\phi_{i}\left(X_{s}\right)\beta_{s}ds\right)\right]\\
 & =\mathbb{\tilde{E}}\left[\Gamma_{4}\Gamma_{5}M_{t_{k+1}}^{t}\left(\phi\right)\left(\int_{t_{k}}^{t_{k+1}}\phi_{i}\left(X_{s}\right)\beta_{s}ds\right)\left(\int_{t_{j}}^{t_{j+1}}\phi_{i}\left(X_{s}\right)\beta_{s}ds\right)\right]
\end{align*}
where 
\begin{align*}
\Gamma_{4} & =\Upsilon\prod_{l=1}^{2}\left(\int_{t_{k+1}}^{t}\kappa_{s}^{l}d\overleftarrow{Y_{s}^{a_{l}}}\right)\\
\Gamma_{5} & =\mathbf{1}_{\left\{ i_{1}=i_{2}\right\} }\int_{t_{j+1}}^{t_{k}}\theta_{s}^{1}\theta_{s}^{2}du\\
 & \qquad+\int_{t_{j+1}}^{t_{k}}\phi_{i_{1}}\left(X_{u}\right)\theta_{u}^{1}\left(\int_{u}^{t_{k}}\phi_{i_{2}}\left(X_{v}\right)\theta_{v}^{2}dv\right)du\\
 & \qquad+\int_{t_{j+1}}^{t_{k}}\phi_{i_{2}}\left(X_{u}\right)\theta_{u}^{2}\left(\int_{u}^{t_{k}}\phi_{i_{1}}\left(X_{v}\right)\theta_{v}^{1}dv\right)du
\end{align*}
To finish the proof one follows the same reasonings in the previous
subcase taking into account that $\Gamma_{4}$ and $\Gamma_{5}$ have
moments of all orders that only depend on $t$ and $\phi$.

\textbf{\lyxarrow{}Case $j$$>k$:}

This is completely symmetric to the previous case by swapping the
role of the $\theta'$s by the $\kappa'$s. 
\end{proof}
\begin{rem}
\label{rem: BackwardEstimate}Under the assumptions of Lemma \ref{lem: Main Backward}
one can prove that for any $q\in\{1,...,m\}$ and $i,i_{1},...,i_{q}a_{1},...,a_{q}\in\left\{ 1,...,d_{Y}\right\} $,
we have that 
\begin{align*}
 & \left|\mathbb{\tilde{E}}\left[\Upsilon M_{0}^{t}\left(\phi\right)\prod_{l=1}^{q}\left\{ \left(\int_{t_{j+1}}^{t}\theta_{s}^{l}d\overleftarrow{Y_{s}^{i_{l}}}\right)\left(\int_{t_{k+1}}^{t}\kappa_{s}^{l}d\overleftarrow{Y_{s}^{a_{l}}}\right)\right\} \left(\int_{t_{j}}^{t_{j+1}}\beta_{s}d\overleftarrow{Y_{s}^{i}}\right)\left(\int_{t_{k}}^{t_{k+1}}\beta_{s}d\overleftarrow{Y_{s}^{i}}\right)\right]\right|\\
 & \leq C\left\{ \mathbf{1}_{\left\{ j\neq k\right\} }\delta^{2m+2}+\mathbf{1}_{\left\{ j=k\right\} }\delta^{2m+1}\right\} .
\end{align*}
\end{rem}

\subsection{\label{subsec:CondExpEst}Conditional expectation estimates}

In this subsection we will show the main estimates for conditional
expectations with respect to the observation filtration that will
allow the proof of our result. Throughout this section we assume that
$\varphi\in\mathcal{B}_{P}$, which ensures that Corollary \ref{cor: BoundsFi}
holds, and $m\in\mathbb{N}$. 
\begin{lem}
\label{lem: Alpha0=00003D00003Dm}Assume that\textbf{ H}$(m)$ holds.
For $\alpha\in\mathcal{R}\left(\mathcal{M}_{m-1}(S_{0})\right)$ with
$\left\vert \alpha\right\vert _{0}=m$ and $i\in\left\{ 0,1,...,d_{Y}\right\} $
we have 
\[
\mathbb{\tilde{E}}\left[\mathbb{\tilde{E}}\left[\varphi(X_{t})e^{\xi_{t}}\int_{0}^{t}I_{\alpha}(L^{\alpha}h_{0}(X_{\cdot}))_{\tau(s),s}dY_{s}^{i}|\mathcal{Y}_{t}\right]^{2}\right]\leq C\delta^{2m}.
\]
\end{lem}
\begin{proof}
Using Jensen inequality, Hölder inequality and Burkholder-Davis-Gundy
inequality, if $i\neq0$, or Jensen inequality, if $i=0$, we get
\begin{align*}
 & \mathbb{\tilde{E}}\left[\left\vert \mathbb{\tilde{E}}\left[\varphi(X_{t})e^{\xi_{t}}\int_{0}^{t}I_{\alpha}(L^{\alpha}h_{i}(X_{\cdot}))_{\tau(s),s}dY_{s}^{i}|\mathcal{Y}_{t}\right]\right\vert ^{2}\right]\\
 & \leq\mathbb{\tilde{E}}\left[\left\vert \varphi(X_{t})e^{\xi_{t}}\int_{0}^{t}I_{\alpha}(L^{\alpha}h_{i}(X_{\cdot}))_{\tau(s),s}dY_{s}^{i}\right\vert ^{2}\right]\\
 & \leq\mathbb{\tilde{E}}\left[\left\vert \varphi(X_{t})e^{\xi_{t}}\right\vert ^{2+\varepsilon}\right]^{\frac{2}{2+\varepsilon}}\mathbb{\tilde{E}}\left[\left\vert \int_{0}^{t}I_{\alpha}(L^{\alpha}h_{i}(X_{\cdot}))_{\tau(s),s}dY_{s}^{i}\right\vert ^{2\frac{2+\varepsilon}{\varepsilon}}\right]^{\frac{\varepsilon}{2+\varepsilon}}\\
 & \leq C\left\Vert \varphi(X_{t})e^{\xi_{t}}\right\Vert _{2+\varepsilon}^{2}\mathbb{\tilde{E}}\left[\left\vert \int_{0}^{t}\left\vert I_{\alpha}(L^{\alpha}h_{i}(X_{\cdot}))_{\tau(s),s}\right\vert ^{2}ds\right\vert ^{\frac{2+\varepsilon}{\varepsilon}}\right]^{\frac{\varepsilon}{2+\varepsilon}}.
\end{align*}
Next, using Jensen inequality again, Fubini's Theorem, Assumption
\textbf{H}$(m)$, Lemma \ref{lem: Moments Iterated Integral} and
that $|\alpha|+|\alpha|_{0}=2m$ we get 
\begin{align*}
 & \mathbb{\tilde{E}}\left[\left\vert \int_{0}^{t}\left\vert I_{\alpha}(L^{\alpha}h_{i}(X_{\cdot}))_{\tau(s),s}\right\vert ^{2}ds\right\vert ^{\frac{2+\varepsilon}{\varepsilon}}\right]\\
 & \leq C\mathbb{\tilde{E}}\left[\left\vert I_{\alpha}(L^{\alpha}h_{i}(X_{\cdot}))_{\tau(s),s}\right\vert ^{2\frac{2+\varepsilon}{\varepsilon}}\right]ds\\
 & \leq C\int_{0}^{t}\mathbb{\tilde{E}}[\sup_{\tau(s)\leq u\leq s}\left\vert L^{\alpha}h_{i}(X_{u})\right\vert ^{2\frac{2+\varepsilon}{\varepsilon}}]\left(s-\tau(s)\right)^{\frac{2+\varepsilon}{\varepsilon}\{|\alpha|+|\alpha|_{0}\}}ds\\
 & \leq C\delta^{2m\frac{2+\varepsilon}{\varepsilon}},
\end{align*}
from which follows the result. 
\end{proof}
\begin{lem}
\label{lem: ds}Assume that \textbf{H}$(m)$ holds. For $\alpha\in\mathcal{R}\left(\mathcal{M}_{m-1}(S_{0})\right)$
with $\left\vert \alpha\right\vert _{0}\neq m$ 
\[
\mathbb{\tilde{E}}\left[\mathbb{\tilde{E}}\left[\varphi(X_{t})e^{\xi_{t}}\int_{0}^{t}I_{\alpha}(L^{\alpha}h_{0}(X_{\cdot}))_{\tau(s),s}ds|\mathcal{Y}_{t}\right]^{2}\right]\leq C\delta^{2m}
\]
\end{lem}
\begin{proof}
We will give only the proof for the case $m\in\{1,2\}.$ The proof
for $m>2$ follows the same ideas but it is tedious to write down
and we leave it to the reader. We split the proof depending on $\left\vert \alpha\right\vert _{0}$,
the number of zeros in $\alpha$. If $m=1,\left\vert \alpha\right\vert _{0}\in\{0\}$
and if $m=2,\left\vert \alpha\right\vert _{0}\in\{0,1\}.$ We group
the three cases into two: $\left\vert \alpha\right\vert _{0}=m-1$
and $\left\vert \alpha\right\vert _{0}=0,$ (of course the two overlap
when $m=1$).

Assume that $\left\vert \alpha\right\vert _{0}=m-1.$ Then, using
Theorem \ref{thm: Integral Representation} we can write 
\begin{align*}
 & \mathbb{\tilde{E}}\left[\varphi(X_{t})e^{\xi_{t}}\int_{0}^{t}I_{\alpha}(L^{\alpha}h_{i}(X_{\cdot}))_{\tau(s),s}ds|\mathcal{Y}_{t}\right]\\
 & =\int_{0}^{t}\mathbb{\tilde{E}}\left[\varphi(X_{t})e^{\xi_{t}}I_{\alpha}(L^{\alpha}h_{i}(X_{\cdot}))_{\tau(s),s}|\mathcal{Y}_{t}\right]ds\\
 & =\int_{0}^{t}\mathbb{\tilde{E}}\left[\mathbb{\tilde{E}}\left[\varphi(X_{t})e^{\xi_{t}}|\mathcal{H}_{0}^{t}\right]I_{\alpha}(L^{\alpha}h_{i}(X_{\cdot}))_{\tau(s),s}|\mathcal{Y}_{t}\right]ds\\
 & \quad+\sum_{r=1}^{d_{V}}\int_{0}^{t}\mathbb{\tilde{E}}\left[\left(\int_{0}^{t}J_{u}^{r}dV_{u}^{r}\right)I_{\alpha}(L^{\alpha}h_{i}(X_{\cdot}))_{\tau(s),s}|\mathcal{Y}_{t}\right]ds\\
 & =\int_{0}^{t}\mathbb{\tilde{E}}\left[\varphi(X_{t})e^{\xi_{t}}|\mathcal{H}_{0}^{t}\right]\mathbb{\tilde{E}}\left[\mathbb{\tilde{E}}\left[I_{\alpha}(L^{\alpha}h_{i}(X_{\cdot}))_{\tau(s),s}|\mathcal{H}_{0}^{t}\right]|\mathcal{Y}_{t}\right]ds\\
 & \quad+\sum_{r=1}^{d_{V}}\int_{0}^{t}\mathbb{\tilde{E}}\left[\mathbb{\tilde{E}}\left[\left(\int_{0}^{t}J_{u}^{r}dV_{u}^{r}\right)I_{\alpha}(L^{\alpha}h_{i}(X_{\cdot}))_{\tau(s),s}|\mathcal{H}_{0}^{t}\right]|\mathcal{Y}_{t}\right]ds.
\end{align*}
Moreover, by Lemma \ref{lem: CondExpect} (\ref{lem: CE1}), we get
$\mathbb{\tilde{E}}\left[I_{\alpha}(L^{\alpha}h_{0}(X_{\cdot}))_{\tau(s),s}|\mathcal{H}_{0}^{t}\right]=0$
and, by Lemma \ref{lem: CondExpect} (\ref{lem: CE2}), for $r=1,...,d_{V}$
we have 
\begin{align*}
 & \mathbb{\tilde{E}}\left[\left(\int_{0}^{t}J_{u}^{r}dV_{u}^{r}\right)I_{\alpha}(L^{\alpha}h_{0}(X_{\cdot}))_{\tau(s),s}|\mathcal{Y}_{t}\right]\\
 & =\boldsymbol{1}_{\{\alpha_{m}=0\}}\int_{\tau(s)}^{s}\mathbb{\tilde{E}}\left[\left(\int_{\tau(s)}^{s}J_{v}^{r}dV_{v}^{r}\right)I_{\alpha-}(L^{\alpha}h_{i}(X_{\cdot}))_{\tau(s),u}|\mathcal{Y}_{t}\right]du\\
 & \quad+\boldsymbol{1}_{\{\alpha_{m}=r\}}\int_{\tau(s)}^{s}\mathbb{\tilde{E}}\left[J_{u}^{r}I_{\alpha-}(L^{\alpha}h_{i}(X_{\cdot}))_{\tau(s),u}|\mathcal{Y}_{t}\right]du.
\end{align*}
Next, using Jensen's inequality, Cauchy-Schwartz inequality, Itô isometry,
Lemma \ref{lem: Moments Iterated Integral} and Remark \ref{rem: M}
we get 
\begin{align*}
 & \mathbb{\tilde{E}}\left[\left\vert \int_{\tau(s)}^{s}\mathbb{\tilde{E}}\left[\left(\int_{\tau(s)}^{s}J_{v}^{r}dV_{v}^{r}\right)I_{\alpha-}(L^{\alpha}h_{i}(X_{\cdot}))_{\tau(s),u}|\mathcal{Y}_{t}\right]du\right\vert ^{2}\right]\\
 & \leq C\delta\int_{\tau(s)}^{s}\mathbb{\tilde{E}}\left[\mathbb{\tilde{E}}\left[\left(\int_{\tau(s)}^{s}J_{v}^{r}dV_{v}^{r}\right)I_{\alpha-}(L^{\alpha}h_{i}(X_{\cdot}))_{\tau(s),u}|\mathcal{Y}_{t}\right]^{2}\right]du\\
 & \leq C\delta\int_{\tau(s)}^{s}\mathbb{\tilde{E}}\left[\mathbb{\tilde{E}}\left[\left\vert \int_{\tau(s)}^{s}J_{v}^{r}dV_{v}^{r}\right\vert ^{2}|\mathcal{Y}_{t}\right]\mathbb{\tilde{E}}\left[\left\vert I_{\alpha-}(L^{\alpha}h_{i}(X_{\cdot}))_{\tau(s),u}\right\vert ^{2}|\mathcal{Y}_{t}\right]\right]du\\
 & \leq C\delta\int_{\tau(s)}^{s}\mathbb{\tilde{E}}\left[\left\vert \int_{\tau(s)}^{s}J_{v}^{r}dV_{v}^{r}\right\vert ^{2}\right]\mathbb{\tilde{E}}\left[\left\vert I_{\alpha-}(L^{\alpha}h_{i}(X_{\cdot}))_{\tau(s),u}\right\vert ^{2}\right]du\\
 & \leq C\delta\int_{\tau(s)}^{s}\int_{\tau(s)}^{s}\mathbb{\tilde{E}}\left[\left\vert J_{v}^{r}\right\vert ^{2}\right]dv(u-\tau(s))^{\left\vert \alpha-\right\vert +\left\vert \alpha-\right\vert _{0}}\mathbb{\tilde{E}}\left[\sup_{0\leq v\leq t}\left\vert L^{\alpha}h_{i}(X_{v})\right\vert ^{2}\right]du\\
 & \leq C\delta^{\left\vert \alpha-\right\vert +\left\vert \alpha-\right\vert _{0}+3}=C\delta^{m-1+m-2+3}=C\delta^{2m},
\end{align*}
and using similar reasonings we get 
\begin{align*}
\mathbb{\tilde{E}}\left[\left\vert \int_{\tau(s)}^{s}\mathbb{\tilde{E}}\left[J_{u}^{r}I_{\alpha-}(L^{\alpha}h_{i}(X_{\cdot}))_{\tau(s),u}|\mathcal{Y}_{t}\right]du\right\vert ^{2}\right] & \leq C\delta^{\left\vert \alpha-\right\vert +\left\vert \alpha-\right\vert _{0}+2}\\
 & =C\delta^{m-1+m-1+2}=C\delta^{2m},
\end{align*}
and the result for the case $\left\vert \alpha\right\vert _{0}=m-1$
follows.

The last case is $\left\vert \alpha\right\vert _{0}=0$ and $m=2.$
Applying Theorem \ref{thm: Integral Representation} we can write
\begin{align*}
 & \mathbb{\tilde{E}}\left[\varphi(X_{t})e^{\xi_{t}}\int_{0}^{t}I_{\alpha}(L^{\alpha}h_{i}(X_{\cdot}))_{\tau(s),s}ds|\mathcal{Y}_{t}\right]\\
 & =\int_{0}^{t}\mathbb{\tilde{E}}\left[\mathbb{\tilde{E}}\left[\varphi(X_{t})e^{\xi_{t}}|\mathcal{H}_{0}^{t}\right]I_{\alpha}(L^{\alpha}h_{i}(X_{\cdot}))_{\tau(s),s}|\mathcal{Y}_{t}\right]ds\\
 & \quad+\sum_{r=1}^{d_{V}}\int_{0}^{t}\mathbb{\tilde{E}}\left[\left(\int_{0}^{t}\mathbb{\tilde{E}}\left[J_{u}^{r}|\mathcal{H}_{0}^{t}\right]dV_{u}^{r}\right)I_{\alpha}(L^{\alpha}h_{i}(X_{\cdot}))_{\tau(s),s}|\mathcal{Y}_{t}\right]ds\\
 & \quad+\sum_{r_{1},r_{2}=1}^{d_{V}}\int_{0}^{t}\mathbb{\tilde{E}}\left[\left(\int_{0}^{t}\int_{0}^{u_{2}}J_{u_{1},u_{2}}^{r_{1},r_{2}}dV_{u_{1}}^{r_{1}}dV_{u_{2}}^{r_{2}}\right)I_{\alpha}(L^{\alpha}h_{i}(X_{\cdot}))_{\tau(s),s}|\mathcal{Y}_{t}\right]ds.\\
 & \triangleq A_{1}+\sum_{r=1}^{d_{V}}A_{2}(r)+\sum_{r_{1},r_{2}=1}^{d_{V}}A_{3}\left(r_{1},r_{2}\right).
\end{align*}
Applying Lemma \ref{lem: CondExpect} (\ref{lem: CE1}), we see that
the term $A_{1}$ vanishes. Applying Lemma \ref{lem: CondExpect}
(\ref{lem: CE2}) and, then, Lemma \ref{lem: CondExpect} (\ref{lem: CE1}),
for $r=1,...,d_{V}$, we can write 
\begin{align*}
A_{2}\left(r\right) & =\mathbf{1}_{\left\{ \alpha_{2}=r\right\} }\int_{0}^{t}\int_{\tau(s)}^{s}\mathbb{\tilde{E}}\left[\mathbb{\tilde{E}}\left[J_{u}^{r}|\mathcal{H}_{0}^{t}\right]I_{\alpha_{1}}(L^{\alpha}h_{i}(X_{\cdot}))_{\tau(s),u}|\mathcal{Y}_{t}\right]duds\\
 & =\mathbf{1}_{\left\{ \alpha_{2}=r\right\} }\int_{0}^{t}\int_{\tau(s)}^{s}\mathbb{\tilde{E}}\left[\mathbb{\tilde{E}}\left[J_{u}^{r}|\mathcal{H}_{0}^{t}\right]\mathbb{\tilde{E}}\left[I_{\alpha_{1}}(L^{\alpha}h_{i}(X_{\cdot}))_{\tau(s),u}|\mathcal{H}_{0}^{t}\right]|\mathcal{Y}_{t}\right]duds\\
 & =\mathbf{1}_{\left\{ \alpha_{2}=r,\left|\alpha_{1}\right|=\left|\alpha_{1}\right|_{0}\right\} }\int_{0}^{t}\int_{\tau(s)}^{s}\mathbb{\tilde{E}}\left[\mathbb{\tilde{E}}\left[J_{u}^{r}|\mathcal{H}_{0}^{t}\right]\right.\\
 & \quad\times\left.I_{\alpha_{1}}\left(\mathbb{\tilde{E}}\left[L^{\alpha}h_{i}(X_{\cdot})|\mathcal{H}_{0}^{t}\right]\right){}_{\tau(s),u}|\mathcal{Y}_{t}\right]duds
\end{align*}
which is equal to zero because $1=\left|\alpha_{1}\right|\neq\left|\alpha_{1}\right|_{0}=0.$
Applying Lemma \ref{lem: CondExpect} (\ref{lem: CE3}), for $r_{1},r_{2}=1,...,d_{V}$
we can write 
\begin{align*}
A_{3}\left(r_{1},r_{2}\right) & =\boldsymbol{1}_{\{\alpha_{2}=r_{2},\alpha_{1}=r_{1}\}}\int_{0}^{t}\int_{\tau(s)}^{s}\int_{\tau(s)}^{u_{2}\wedge u}\mathbb{\tilde{E}}\left[J_{v,u}^{r_{1},r_{2}}I_{\left(\alpha-\right)-}(L^{\alpha}h_{i}(X_{\cdot}))_{\tau(s),v}|\mathcal{Y}_{t}\right]dvduds\\
 & =\boldsymbol{1}_{\{\alpha_{2}=r_{2},\alpha_{1}=r_{1}\}}\int_{0}^{t}\int_{\tau(s)}^{s}\int_{\tau(s)}^{u}\mathbb{\tilde{E}}\left[J_{v,u}^{r_{1},r_{2}}L^{\alpha}h_{i}(X_{v})|\mathcal{Y}_{t}\right]dvduds\\
 & \leq\boldsymbol{1}_{\{\alpha_{2}=r_{2},\alpha_{1}=r_{1}\}}\int_{0}^{t}\int_{\tau(s)}^{s}\int_{\tau(s)}^{u}\mathbb{\tilde{E}}\left[\mathbb{\tilde{E}}\left[\left|J_{v,u}^{r_{1},r_{2}}\right|^{2}|\mathcal{H}_{0}^{t}\right]^{1/2}\right.\\
 & \quad\times\left.\mathbb{\tilde{E}}\left[\left|L^{\alpha}h_{i}(X_{v})\right|^{2}|\mathcal{H}_{0}^{t}\right]^{1/2}|\mathcal{Y}_{t}\right]dvduds,
\end{align*}
Hence, using Jensen's inequality, Cauchy-Schwartz inequality and Remark
\ref{rem: M} we have 
\begin{align*}
\mathbb{\tilde{E}}\left[\left|A_{3}\left(r_{1},r_{2}\right)\right|^{2}\right] & \leq C\boldsymbol{1}_{\{\alpha_{2}=r_{2},\alpha_{1}=r_{1}\}}\delta^{2}\int_{0}^{t}\int_{\tau(s)}^{s}\int_{\tau(s)}^{u}\mathbb{\tilde{E}}\left[\mathbb{\tilde{E}}\left[\left|J_{v,u}^{r_{1},r_{2}}\right|^{2}|\mathcal{H}_{0}^{t}\right]\right.\\
 & \quad\times\left.\mathbb{\tilde{E}}\left[\left|L^{\alpha}h_{i}(X_{v})\right|^{2}|\mathcal{H}_{0}^{t}\right]\right]dvduds\\
 & \leq C\boldsymbol{1}_{\{\alpha_{2}=r_{2},\alpha_{1}=r_{1}\}}\delta^{2}\int_{0}^{t}\int_{\tau(s)}^{s}\int_{\tau(s)}^{u}\mathbb{\tilde{E}}\left[\left|J_{v,u}^{r_{1},r_{2}}\right|^{2}\right]dvduds\\
 & \leq C\boldsymbol{1}_{\{\alpha_{2}=r_{2},\alpha_{1}=r_{1}\}}\delta^{4}=C\boldsymbol{1}_{\{\alpha_{2}=r_{2},\alpha_{1}=r_{1}\}}\delta^{2m},
\end{align*}
and we can conclude. 
\end{proof}
\begin{lem}
\label{lem: Convenient}Let $m\in\{1,2\}$ and assume that \textbf{H}$(m)$
holds. For $\alpha\in\mathcal{R}\left(\mathcal{M}_{m-1}(S_{0})\right)$,
$\left\vert \alpha\right\vert _{0}\neq\left|\alpha\right|$ and $i\neq0$,
we can write 
\end{lem}
\begin{align}
 & \mathbb{\tilde{E}}\left[\varphi(X_{t})e^{\xi_{t}}\int_{0}^{t}I_{\alpha}(L^{\alpha}h_{i}(X_{\cdot}))_{\tau(s),s}dY_{s}^{i}|\mathcal{Y}_{t}\right]\nonumber \\
 & =\sum_{r=1}^{d_{V}}\sum_{j=0}^{n-1}\mathbb{\tilde{E}}\left[\left(\int_{t_{j}}^{t_{j+1}}J_{s}^{r}dV_{s}^{r}\right)\left(\int_{t_{j}}^{t_{j+1}}\left(Y_{t_{j+1}}^{i}-Y_{t_{j}}^{i}\right)I_{\alpha-}(L^{\alpha}h_{i}(X_{\cdot}))_{\tau(s),s}dV_{s}^{\alpha_{|\alpha|}}\right)|\mathcal{Y}_{t}\right],\label{eq: Convenient_1}
\end{align}
and 
\begin{align}
 & \mathbb{\tilde{E}}\left[\varphi(X_{t})e^{\xi_{t}}\int_{0}^{t}I_{\alpha}(L^{\alpha}h_{i}(X_{\cdot}))_{\tau(s),s}dY_{s}^{i}|\mathcal{Y}_{t}\right]\nonumber \\
 & =\sum_{r_{1}=1}^{d_{V}}\sum_{j=0}^{n-1}\mathbb{\tilde{E}}\left[\left(\int_{t_{j}}^{t_{j+1}}\mathbb{\tilde{E}}\left[J_{s}^{r_{1}}|\mathcal{H}_{0}^{t}\right]dV_{s}^{r_{1}}\right)\right.\label{eq: Convenient_2}\\
 & \quad\times\left.\left(\int_{t_{j}}^{t_{j+1}}\left(Y_{t_{j+1}}^{i}-Y_{t_{j}}^{i}\right)I_{\alpha-}(L^{\alpha}h_{i}(X_{\cdot}))_{\tau(s),s}dV_{s}^{\alpha_{|\alpha|}}\right)|\mathcal{Y}_{t}\right]\nonumber \\
 & \quad+\sum_{r_{1},r_{2}=1}^{d_{V}}\sum_{j=0}^{n-1}\mathbb{\tilde{E}}\left[\left(\int_{t_{j}}^{t_{j+1}}\int_{0}^{s_{2}}J_{s_{1},s_{2}}^{r_{1},r_{2}}dV_{s_{1}}^{r_{1}}dV_{s_{2}}^{r_{2}}\right)\right.\nonumber \\
 & \quad\times\left.\left(\int_{t_{j}}^{t_{j+1}}\left(Y_{t_{j+1}}^{i}-Y_{t_{j}}^{i}\right)I_{\alpha-}(L^{\alpha}h_{i}(X_{\cdot}))_{\tau(s),s}dV_{s}^{\alpha_{|\alpha|}}\right)|\mathcal{Y}_{t}\right].\nonumber 
\end{align}

\begin{proof}
Note that, as $\left\vert \alpha\right\vert _{0}\neq\left|\alpha\right|,$
by Lemma \ref{lem: CondExpect} $\left(\ref{lem: CE1}\right)$ we
have that if $0\leq u\leq v\leq w\leq t$ then 
\begin{equation}
\mathbb{\tilde{E}}\left[I_{\alpha}(L^{\alpha}h_{i}(X_{\cdot}))_{v,w}|\mathcal{H}_{u}^{t}\right]=0.\label{eq: NullCondExp}
\end{equation}
Using Theorem \ref{thm: Integral Representation} we can write 
\begin{align*}
 & \varphi(X_{t})e^{\xi_{t}}\int_{0}^{t}I_{\alpha}(L^{\alpha}h_{i}(X_{\cdot}))_{\tau(s),s}dY_{s}^{i}\\
 & =\varphi(X_{t})e^{\xi_{t}}\sum_{j_{1}=0}^{n-1}\int_{t_{j_{1}}}^{t_{j_{1}+1}}I_{\alpha}(L^{\alpha}h_{i}(X_{\cdot}))_{t_{j_{1}},s}dY_{s}^{i}\\
 & =\mathbb{\tilde{E}}\left[\varphi(X_{t})e^{\xi_{t}}|\mathcal{H}_{0}^{t}\right]\sum_{j_{1}=0}^{n-1}\int_{t_{j_{1}}}^{t_{j_{1}+1}}I_{\alpha}(L^{\alpha}h_{i}(X_{\cdot}))_{t_{j_{1}},s}dY_{s}^{i}\\
 & \quad+\sum_{r=1}^{d_{V}}\sum_{j_{1},j_{2}=0}^{n-1}\left(\int_{t_{j_{2}}}^{t_{j_{2}+1}}J_{s}^{r}dV_{s}^{r}\right)\left(\int_{t_{j_{1}}}^{t_{j_{1}+1}}I_{\alpha}(L^{\alpha}h_{i}(X_{\cdot}))_{t_{j_{1}},s}dY_{s}^{i}\right).
\end{align*}
Next, for $j_{1}\geq0$ we get, using equation $\left(\ref{eq: NullCondExp}\right)$
that 
\begin{align*}
 & \mathbb{\tilde{E}}\left[\mathbb{\tilde{E}}\left[\varphi(X_{t})e^{\xi_{t}}|\mathcal{H}_{0}^{t}\right]\int_{t_{j_{1}}}^{t_{j_{1}+1}}I_{\alpha}(L^{\alpha}h_{i}(X_{\cdot}))_{t_{j_{1}},s}dY_{s}^{i}|\mathcal{H}_{0}^{t}\right]\\
 & =\mathbb{\tilde{E}}\left[\varphi(X_{t})e^{\xi_{t}}|\mathcal{H}_{0}^{t}\right]\int_{t_{j_{1}}}^{t_{j_{1}+1}}\mathbb{\tilde{E}}\left[I_{\alpha}(L^{\alpha}h_{i}(X_{\cdot}))_{t_{j_{1}},s}|\mathcal{H}_{0}^{t}\right]dY_{s}^{i}=0.
\end{align*}
Moreover, for $r=1,...,d_{V},$ if $j_{2}>j_{1}$ we get that 
\begin{align*}
 & \mathbb{\tilde{E}}\left[\left(\int_{t_{j_{2}}}^{t_{j_{2}+1}}J_{s}^{r}dV_{s}^{r}\right)\left(\int_{t_{j_{1}}}^{t_{j_{1}+1}}I_{\alpha}(L^{\alpha}h_{i}(X_{\cdot}))_{t_{j_{1}},s}dY_{s}^{i}\right)|\mathcal{H}_{t_{j_{2}}}^{t}\right]\\
 & =\left(\int_{t_{j_{1}}}^{t_{j_{1}+1}}I_{\alpha}(L^{\alpha}h_{i}(X_{\cdot}))_{t_{j_{1}},s}dY_{s}^{i}\right)\mathbb{\tilde{E}}\left[\int_{t_{j_{2}}}^{t_{j_{2}+1}}J_{s}^{r}dV_{s}^{r}|\mathcal{H}_{t_{j_{2}}}^{t}\right]=0,
\end{align*}
and if $j_{2}<j_{1}$ we get that 
\begin{align*}
 & \mathbb{\tilde{E}}\left[\left(\int_{t_{j_{2}}}^{t_{j_{2}+1}}J_{s}^{r}dV_{s}^{r}\right)\left(\int_{t_{j_{1}}}^{t_{j_{1}+1}}I_{\alpha}(L^{\alpha}h_{i}(X_{\cdot}))_{t_{j_{1}},s}dY_{s}^{i}\right)|\mathcal{H}_{t_{j_{1}}}^{t}\right]\\
 & =\left(\int_{t_{j_{2}}}^{t_{j_{2}+1}}J_{s}^{r}dV_{s}^{r}\right)\int_{t_{j_{1}}}^{t_{j_{1}+1}}\mathbb{\tilde{E}}\left[I_{\alpha}(L^{\alpha}h_{i}(X_{\cdot}))_{t_{j_{1}},s}|\mathcal{H}_{t_{j_{1}}}^{t}\right]dY_{s}^{i}=0.
\end{align*}
Hence, using the tower property of the conditional expectation we
can write 
\begin{align}
 & \mathbb{\tilde{E}}\left[\varphi(X_{t})e^{\xi_{t}}\int_{0}^{t}I_{\alpha}(L^{\alpha}h_{i}(X_{\cdot}))_{\tau(s),s}dY_{s}^{i}|\mathcal{Y}_{t}\right]\nonumber \\
 & =\sum_{r=1}^{d_{V}}\mathbb{\tilde{E}}\left[\sum_{j=0}^{n-1}\left(\int_{t_{j}}^{t_{j+1}}J_{s}^{r}dV_{s}^{r}\right)\left(\int_{t_{j}}^{t_{j+1}}I_{\alpha}(L^{\alpha}h_{i}(X_{\cdot}))_{\tau(s),s}dY_{s}^{i}\right)|\mathcal{Y}_{t}\right].\label{eq: Convenient Intermediate}
\end{align}
By integration by parts formula for $\mathcal{F}_{s}^{V,0}\vee\mathcal{Y}_{s}$-semimartingales
we have 
\begin{align*}
\int_{t_{j}}^{t_{j+1}}I_{\alpha}(L^{\alpha}h_{i}(X_{\cdot}))_{\tau(s),s}dY_{s}^{i} & =\left(Y_{t_{j+1}}^{i}-Y_{t_{j}}^{i}\right)I_{\alpha}(L^{\alpha}h_{i}(X_{\cdot}))_{t_{j},t_{j+1}}\\
 & \quad-\int_{t_{j}}^{t_{j+1}}\left(Y_{s}^{i}-Y_{t_{j}}^{i}\right)I_{\alpha-}(L^{\alpha}h_{i}(X_{\cdot}))_{\tau(s),s}dV_{s}^{\alpha_{1}}.
\end{align*}
Moreover, we can rewrite the right hand side of the previous equality
as a well defined $\mathcal{H}_{s}^{t}$-iterated integral and obtain
\begin{align*}
\int_{t_{j}}^{t_{j+1}}I_{\alpha}(L^{\alpha}h_{i}(X_{\cdot}))_{\tau(s),s}dY_{s}^{i} & =\int_{t_{j}}^{t_{j+1}}\left(Y_{t_{j+1}}^{i}-Y_{s}^{i}\right)I_{\alpha-}(L^{\alpha}h_{i}(X_{\cdot}))_{\tau(s),s}dV^{\alpha_{|\alpha|}},
\end{align*}
which combined with equation $\left(\ref{eq: Convenient Intermediate}\right)$
gives equation $\left(\ref{eq: Convenient_1}\right)$. Finally, using
Theorem \ref{thm: ST formula} with $k$=1 and repeating the same
reasonings as before we get equation $\left(\ref{eq: Convenient_2}\right)$. 
\end{proof}
\begin{lem}
\label{lem: dY_m=00003D00003D1}Assume that \textbf{H}$(1)$ holds
and $\varphi\in C_{P}^{2}$. For $\alpha\in\mathcal{R}\left(\mathcal{M}_{0}(S_{0})\right)$
with $\left\vert \alpha\right\vert _{0}\neq1$ and $i\neq0$ we have
that 
\begin{align*}
\mathbb{\tilde{E}}\left[\mathbb{\tilde{E}}\left[\varphi(X_{t})e^{\xi_{t}}\int_{0}^{t}I_{\alpha}(L^{\alpha}h_{i}(X_{\cdot}))_{\tau(s),s}dY_{s}^{i}|\mathcal{Y}_{t}\right]^{2}\right] & \leq C\delta^{2}.
\end{align*}
\end{lem}
\begin{proof}
We divide the proof into several steps.

\textbf{Step 1}. First we will find a more convenient expression for
\[
\mathbb{\tilde{E}}\left[\varphi(X_{t})e^{\xi_{t}}\int_{0}^{t}I_{\alpha}(L^{\alpha}h_{i}(X_{\cdot}))_{\tau(s),s}dY_{s}^{i}|\mathcal{Y}_{t}\right].
\]
Recall that $\alpha\in\mathcal{R}\left(\mathcal{M}_{0}(S_{0})\right)$
with $\left\vert \alpha\right\vert _{0}\neq1$ concides with the set
of multiindices $\alpha=\left(\alpha_{1}\right)$ with $\alpha_{1}\in\left\{ 1,...,d_{V}\right\} $.
Using Lemma \ref{lem: Convenient}, equation $\left(\ref{eq: Convenient_1}\right)$,
and taking into account that $I_{\alpha-}(L^{\alpha}h_{i}(X_{\cdot}))_{\tau(s),s}=L^{\alpha_{1}}h_{i}(X_{s}),$
we can write 
\begin{align*}
 & \mathbb{\tilde{E}}\left[\varphi(X_{t})e^{\xi_{t}}\int_{0}^{t}I_{\alpha}(L^{\alpha}h_{i}(X_{\cdot}))_{\tau(s),s}dY_{s}^{i}|\mathcal{Y}_{t}\right]\\
 & =\sum_{r=1}^{d_{V}}\sum_{j=0}^{n-1}\mathbb{\tilde{E}}\left[\left(\int_{t_{j}}^{t_{j+1}}J_{s}^{r}dV_{s}^{r}\right)\left(\int_{t_{j}}^{t_{j+1}}\left(Y_{t_{j+1}}^{i}-Y_{t_{j}}^{i}\right)L^{\alpha_{1}}h_{i}(X_{s})dV_{s}^{\alpha_{1}}\right)|\mathcal{Y}_{t}\right]
\end{align*}
Next, by Lemma \ref{lem: CondExpect} (\ref{lem: CE2}) we get that
\begin{align*}
 & \mathbb{\tilde{E}}\left[\varphi(X_{t})e^{\xi_{t}}\int_{0}^{t}I_{\alpha}(L^{\alpha}h_{i}(X_{\cdot}))_{\tau(s),s}dY_{s}^{i}|\mathcal{Y}_{t}\right]\\
 & =\sum_{r=1}^{d_{V}}\boldsymbol{1}_{\left\{ \alpha_{1}=r\right\} }\mathbb{\tilde{E}}\left[\sum_{j=0}^{n-1}\int_{t_{j}}^{t_{j+1}}J_{s}^{r}\left(Y_{t_{j+1}}^{i}-Y_{s}^{i}\right)L^{\alpha_{1}}h_{i}(X_{s})ds|\mathcal{Y}_{t}\right]\\
 & =\sum_{r=1}^{d_{V}}\boldsymbol{1}_{\left\{ \alpha_{1}=r\right\} }\mathbb{\tilde{E}}\left[\int_{0}^{t}J_{s}^{r}\left(Y_{\eta(s)}^{i}-Y_{s}^{i}\right)L^{\alpha_{1}}h_{i}(X_{s})ds|\mathcal{Y}_{t}\right]\\
 & =\sum_{r=1}^{d_{V}}\boldsymbol{1}_{\left\{ \alpha_{1}=r\right\} }\left(B_{1}\left(r\right)+B_{2}\left(r\right)+B_{3}\left(r\right)\right),
\end{align*}
where 
\begin{align*}
B_{1}\left(r\right) & \triangleq\mathbb{\tilde{E}}\left[\int_{0}^{t}\left(J_{s}^{r}-J_{\eta(s)}^{r}\right)\left(Y_{\eta(s)}^{i}-Y_{s}^{i}\right)L^{r}h_{i}(X_{s})ds|\mathcal{Y}_{t}\right],\\
B_{2}\left(r\right) & \triangleq\mathbb{\tilde{E}}\left[\int_{0}^{t}J_{\eta(s)}^{r}\left(Y_{\eta(s)}^{i}-Y_{s}^{i}\right)\left(L^{r}h_{i}(X_{s})-L^{r}h_{i}(X_{\tau(s)})\right)ds|\mathcal{Y}_{t}\right],\\
B_{3}\left(r\right) & \triangleq\mathbb{\tilde{E}}\left[\int_{0}^{t}J_{\eta(s)}^{r}L^{r}h_{i}(X_{\tau(s)})\left(Y_{\eta(s)}^{i}-Y_{s}^{i}\right)ds|\mathcal{Y}_{t}\right].
\end{align*}

\textbf{Step 2}. Next, we prove the result for $B_{1}\left(r\right).$
Applying Jensen inequality, Cauchy-Schwarz inequality, Hôlder inequality,
Remark \ref{rem: M}, that $Y$$^{i}$ is a Brownian motion under
$\tilde{P}$ and Lemma \ref{lem: RegularityKernels} we have that
\begin{align*}
\mathbb{\tilde{E}}\left[\left|B_{1}\left(r\right)\right|^{2}\right] & \leq C(t)\int_{0}^{t}\mathbb{\tilde{E}}\left[\mathbb{\tilde{E}}\left[\left(J_{s}^{r}-J_{\eta(s)}^{r}\right)^{2}\left(Y_{\eta(s)}^{i}-Y_{s}^{i}\right)^{2}|\mathcal{Y}_{t}\right]\mathbb{\tilde{E}}\left[\left|L^{r}h_{i}(X_{s})\right|^{2}|\mathcal{Y}_{t}\right]\right]ds\\
 & \leq C(t)\int_{0}^{t}\mathbb{\tilde{E}}\left[\left(J_{s}^{r}-J_{\eta(s)}^{r}\right)^{2}\left(Y_{\eta(s)}^{i}-Y_{s}^{i}\right)^{2}\right]ds\\
 & \leq C(t)\int_{0}^{t}\mathbb{\tilde{E}}\left[\left(J_{s}^{r}-J_{\eta(s)}^{r}\right)^{2+\varepsilon}\right]^{2/(2+\varepsilon)}\mathbb{\tilde{E}}\left[\left(Y_{\eta(s)}^{i}-Y_{s}^{i}\right)^{2\frac{2+\varepsilon}{\varepsilon}}\right]^{\varepsilon/(2+\varepsilon)}ds\\
 & \leq C(t)\delta\int_{0}^{t}\mathbb{\tilde{E}}\left[\left(J_{s}^{r}-J_{\eta(s)}^{r}\right)^{2+\varepsilon}\right]^{2/(2+\varepsilon)}ds\\
 & \leq C\left(t\right)t\delta^{2}.
\end{align*}

\textbf{Step 3}. Here, we prove the result for $B_{2}\left(r\right).$
Applying Jensen inequality and Cauchy-Schwarz inequality we get 
\begin{align*}
\mathbb{\tilde{E}}\left[\left|B_{2}\left(r\right)\right|^{2}\right] & \leq C(t)\int_{0}^{t}\mathbb{\tilde{E}}\left[\mathbb{\tilde{E}}\left[\left|J_{\eta(s)}^{r}\left(Y_{\eta(s)}^{i}-Y_{s}^{i}\right)\right|^{2}|\mathcal{Y}_{t}\right]\mathbb{\tilde{E}}\left[\left|L^{r}h_{i}(X_{s})-L^{r}h_{i}(X_{\tau(s)})\right|^{2}|\mathcal{Y}_{t}\right]\right]ds\\
 & \leq C(t)\int_{0}^{t}\mathbb{\tilde{E}}\left[\left|J_{\eta(s)}^{r}\left(Y_{\eta(s)}^{i}-Y_{s}^{i}\right)\right|^{2}\right]\mathbb{\tilde{E}}\left[\left|L^{r}h_{i}(X_{s})-L^{r}h_{i}(X_{\tau(s)})\right|^{2}\right]ds.
\end{align*}
Applying Hölder inequality and Proposition \ref{prop: UnifBoundKernels}
we can conclude that 
\begin{align*}
\mathbb{\tilde{E}}\left[\left|J_{\eta(s)}^{r}\left(Y_{\eta(s)}^{i}-Y_{s}^{i}\right)\right|^{2}\right] & \leq\mathbb{\tilde{E}}\left[\left|J_{\eta(s)}^{r}\right|^{2+\varepsilon}\right]^{2/(2+\varepsilon)}\mathbb{\tilde{E}}\left[\left|\left(Y_{\eta(s)}^{i}-Y_{s}^{i}\right)\right|^{2\frac{2+\varepsilon}{\varepsilon}}\right]^{\varepsilon/\left(2+\varepsilon\right)}\\
 & \leq\delta\sup_{0\leq s\leq t}\mathbb{\tilde{E}}\left[\left|J_{s}^{r}\right|^{2+\varepsilon}\right]^{2/(2+\varepsilon)}\leq C\delta.
\end{align*}
On the other hand, we can write 
\[
L^{r}h_{i}(X_{s})-L^{r}h_{i}(X_{\tau(s)})=\int_{\tau\left(s\right)}^{s}L^{(0,r)}h_{i}(X_{u})du+\sum_{r_{1}=1}^{d_{V}}\int_{\tau\left(s\right)}^{s}L^{(r_{1},r)}h_{i}(X_{u})dV_{u}^{r_{1}}.
\]
As the worst rate is achieved by the terms with the stochastic integral,
it suffices to show that 
\[
\mathbb{\tilde{E}}\left[\left|\int_{\tau\left(s\right)}^{s}L^{(r_{1},r)}h_{i}(X_{u})dV_{u}^{r_{1}}\right|^{2}\right]\leq C\delta,
\]
which easily follows by Itô isometry and Remark \ref{rem: M}.

\textbf{Step 4}. Finally, we prove the result for $B_{3}\left(r\right).$
We can write 
\begin{align*}
B_{3}\left(r\right) & =\sum_{j=0}^{n-1}\mathbb{\tilde{E}}\left[\int_{t_{j}}^{t_{j+1}}J_{t_{j+1}}^{r}L^{r}h_{i}(X_{t_{j}})\left(Y_{t_{j+1}}^{i}-Y_{s}^{i}\right)ds|\mathcal{Y}_{t}\right]\\
 & =\sum_{j=0}^{n-1}\mathbb{\tilde{E}}\left[J_{t_{j+1}}^{r}L^{r}h_{i}(X_{t_{j}})\int_{t_{j}}^{t_{j+1}}\left(Y_{t_{j+1}}^{i}-Y_{s}^{i}\right)ds|\mathcal{Y}_{t}\right]\\
 & =\sum_{j=0}^{n-1}\mathbb{\tilde{E}}\left[J_{t_{j+1}}^{r}L^{r}h_{i}(X_{t_{j}})\int_{t_{j}}^{t_{j+1}}\left(s-t_{j}\right)dY_{s}^{i}|\mathcal{Y}_{t}\right]\\
 & \triangleq\sum_{j=0}^{n-1}\mathbb{\tilde{E}}\left[J_{t_{j+1}}^{r}L^{r}h_{i}(X_{t_{j}})\int_{t_{j}}^{t_{j+1}}\beta_{s}^{j}dY_{s}^{i}|\mathcal{Y}_{t}\right]
\end{align*}
Moreover, 
\[
J_{t_{j+1}}^{r}=\mathbb{\tilde{E}}\left[D_{t_{j+1}}^{r}\left[\varphi\left(X_{t}\right)e^{\xi_{t}}\right]|\mathcal{H}_{t_{j+1}}^{t}\right],
\]
by the Clark-Ocone formula. Using the \textcolor{black}{product formula}
for the Malliavin derivative, we get 
\begin{align*}
D_{t_{j+1}}^{r}\left[\varphi\left(X_{t}\right)e^{\xi_{t}}\right] & =e^{\xi_{t}}D_{t_{j+1}}^{r}\varphi\left(X_{t}\right)+\varphi\left(X_{t}\right)D_{t_{j+1}}^{r}e^{\xi_{t}}
\end{align*}
Therefore, using the tower property of the conditional expectation
and the previous expression for the Malliavin derivative, we have
\begin{align*}
 & \mathbb{\tilde{E}}\left[J_{t_{j+1}}^{r}L^{r}h_{i}(X_{t_{j}})\int_{t_{j}}^{t_{j+1}}\beta_{s}^{j}dY_{s}^{i}|\mathcal{Y}_{t}\right]\\
 & =\mathbb{\tilde{E}}\left[\mathbb{\tilde{E}}\left[D_{t_{j+1}}^{r}\left[\varphi\left(X_{t}\right)e^{\xi_{t}}\right]|\mathcal{H}_{t_{j+1}}^{t}\right]L^{r}h_{i}(X_{t_{j}})\int_{t_{j}}^{t_{j+1}}\beta_{s}^{j}dY_{s}^{i}|\mathcal{Y}_{t}\right]\\
 & =\mathbb{\tilde{E}}\left[e^{\xi_{t}}D_{t_{j+1}}^{r}\varphi\left(X_{t}\right)L^{r}h_{i}(X_{t_{j}})\int_{t_{j}}^{t_{j+1}}\beta_{s}^{j}dY_{s}^{i}|\mathcal{Y}_{t}\right]\\
 & \quad+\mathbb{\tilde{E}}\left[\varphi\left(X_{t}\right)D_{t_{j+1}}^{r}e^{\xi_{t}}L^{r}h_{i}(X_{t_{j}})\int_{t_{j}}^{t_{j+1}}\beta_{s}^{j}dY_{s}^{i}|\mathcal{Y}_{t}\right].
\end{align*}
Then, 
\begin{align*}
\mathbb{\tilde{E}}\left[\left|B_{3}\left(r\right)\right|^{2}\right] & \leq\mathbb{\tilde{E}}\left[\left|\sum_{j=0}^{n-1}\mathbb{\tilde{E}}\left[J_{t_{j+1}}^{r}L^{r}h_{i}(X_{t_{j}})\int_{t_{j}}^{t_{j+1}}\beta_{s}^{j}dY_{s}^{i}|\mathcal{Y}_{t}\right]\right|^{2}\right]\\
 & \leq2\mathbb{\tilde{E}}\left[\left|\sum_{j=0}^{n-1}e^{\xi_{t}}D_{t_{j+1}}^{r}\varphi\left(X_{t}\right)L^{r}h_{i}(X_{t_{j}})\int_{t_{j}}^{t_{j+1}}\beta_{s}^{j}dY_{s}^{i}\right|^{2}\right]\\
 & \quad+2\mathbb{\tilde{E}}\left[\left|\sum_{j=0}^{n-1}\varphi\left(X_{t}\right)D_{t_{j+1}}^{r}e^{\xi_{t}}L^{r}h_{i}(X_{t_{j}})\int_{t_{j}}^{t_{j+1}}\beta_{s}^{j}dY_{s}^{i}\right|^{2}\right]\\
 & \triangleq2A_{1}\left(r\right)+2A_{2}\left(r\right).
\end{align*}
Next, note that 
\begin{align*}
D_{t_{j+1}}^{r}e^{\xi_{t}} & =e^{\xi_{t}}\left\{ \sum_{k=1}^{d_{Y}}\int_{0}^{t}D_{t_{j+1}}^{r}h_{k}(X_{s})dY_{s}^{k}-\frac{1}{2}\sum_{k=1}^{d_{Y}}\int_{0}^{t}D_{t_{j+1}}^{r}\left[h_{k}(X_{s})^{2}\right]ds\right\} \\
 & =e^{\xi_{t}}\left\{ \sum_{k=1}^{d_{Y}}\int_{t_{j+1}}^{t}D_{t_{j+1}}^{r}h_{k}(X_{s})dY_{s}^{k}-\frac{1}{2}\sum_{k=1}^{d_{Y}}\int_{t_{j+1}}^{t}D_{t_{j+1}}^{r}\left[h_{k}(X_{s})^{2}\right]ds\right\} \\
 & \triangleq e^{\xi_{t}}\left\{ \sum_{k=1}^{d_{Y}}\int_{t_{j+1}}^{t}\alpha_{s}^{j,k,1}dY_{s}^{k}-\frac{1}{2}\sum_{k=1}^{d_{Y}}\int_{t_{j+1}}^{t}\alpha_{s}^{j,k,2}ds\right\} ,
\end{align*}
where we have used that $D_{u}^{r}h_{k}(X_{s})=0,s<u<t$. In addition,
note that 
\[
e^{2\xi_{t}}=M_{0}^{t}\left(2h\right)\exp\left(\sum_{i=1}^{d_{Y}}\int_{0}^{t}h_{i}^{2}\left(X_{u}\right)du\right),
\]
where 
\[
M_{s}^{t}\left(h\right)=\exp\left(\sum_{i=1}^{d_{Y}}\int_{s}^{t}h_{i}(X_{u})dY_{u}^{i}-\frac{1}{2}\sum_{i=1}^{d_{Y}}\int_{s}^{t}h_{i}^{2}\left(X_{u}\right)du\right),
\]
is an exponential martingale. Defining 
\[
\Gamma(j_{1},j_{2})\triangleq D_{t_{j_{1}+1}}^{r}\varphi\left(X_{t}\right)D_{t_{j_{2}+1}}^{r}\varphi\left(X_{t}\right)L^{r}h_{i}(X_{t_{j_{1}}})L^{r}h_{i}(X_{t_{j_{2}}})\exp\left(\sum_{i=1}^{d_{Y}}\int_{0}^{t}h_{i}^{2}\left(X_{u}\right)du\right),
\]
and 
\[
\varLambda(j_{1},j_{2})\triangleq\varphi\left(X_{t}\right)^{2}L^{r}h_{i}(X_{t_{j_{1}}})L^{r}h_{i}(X_{t_{j_{2}}})\exp\left(\sum_{i=1}^{d_{Y}}\int_{0}^{t}h_{i}^{2}\left(X_{u}\right)du\right),
\]
we can write 
\[
A_{1}\left(r\right)=\sum_{j_{1},j_{2}=0}^{n-1}\mathbb{\tilde{E}}\left[\Gamma(j_{1},j_{2})M_{0}^{t}\left(2h\right)\int_{t_{j_{1}}}^{t_{j_{1}+1}}\beta_{s}^{j_{1}}dY_{s}^{i}\int_{t_{j_{2}}}^{t_{j_{2}+1}}\beta_{s}^{j_{2}}dY_{s}^{i}\right],
\]
and 
\[
A_{2}\left(r\right)=\sum_{k_{1},k_{2}=1}^{d_{Y}}A_{2,1}\left(r,k_{1},k_{2}\right)-\frac{1}{2}A_{2,2}\left(r,k_{1},k_{2}\right)-\frac{1}{2}A_{2,3}\left(r,k_{1},k_{2}\right)+\frac{1}{4}A_{2,4}\left(r,k_{1},k_{2}\right),
\]
where 
\begin{align*}
A_{2,1}\left(r,k_{1},k_{2}\right) & \triangleq\sum_{j_{1},j_{2}=0}^{n-1}\mathbb{\tilde{E}}\left[\varLambda(j_{1},j_{2})M_{0}^{t}\left(2h\right)\int_{t_{j_{1}+1}}^{t}\alpha_{s}^{j_{1},k_{1},1}dY_{s}^{k_{1}}\int_{t_{j_{2}+1}}^{t}\alpha_{s}^{j_{2},k_{2},1}dY_{s}^{k_{2}}\right.\\
 & \quad\times\left.\int_{t_{j_{1}}}^{t_{j_{1}+1}}\beta_{s}^{j_{1}}dY_{s}^{i}\int_{t_{j_{2}}}^{t_{j_{2}+1}}\beta_{s}^{j_{2}}dY_{s}^{i}\right],\\
A_{2,2}\left(r,k_{1},k_{2}\right) & \triangleq\sum_{j_{1},j_{2}=0}^{n-1}\mathbb{\tilde{E}}\left[\varLambda(j_{1},j_{2})M_{0}^{t}\left(2h\right)\int_{t_{j_{1}+1}}^{t}\alpha_{s}^{j_{1},k_{1},1}dY_{s}^{k_{1}}\int_{t_{j_{2}+1}}^{t}\alpha_{s}^{j_{2},k_{2},2}ds\right.\\
 & \quad\times\left.\int_{t_{j_{1}}}^{t_{j_{1}+1}}\beta_{s}^{j_{1}}dY_{s}^{i}\int_{t_{j_{2}}}^{t_{j_{2}+1}}\beta_{s}^{j_{2}}dY_{s}^{i}\right],\\
A_{2,3}\left(r,k_{1},k_{2}\right) & \triangleq\sum_{j_{1},j_{2}=0}^{n-1}\mathbb{\tilde{E}}\left[\varLambda(j_{1},j_{2})M_{0}^{t}\left(2h\right)\int_{t_{j_{1}+1}}^{t}\alpha_{s}^{j_{1},k_{1},2}ds\int_{t_{j_{2}+1}}^{t}\alpha_{s}^{j_{2},k_{2},1}dY_{s}^{k_{2}}\right.\\
 & \quad\times\left.\int_{t_{j_{1}}}^{t_{j_{1}+1}}\beta_{s}^{j_{1}}dY_{s}^{i}\int_{t_{j_{2}}}^{t_{j_{2}+1}}\beta_{s}^{j_{2}}dY_{s}^{i}\right],\\
A_{2,4}\left(r,k_{1},k_{2}\right) & \triangleq\sum_{j_{1},j_{2}=0}^{n-1}\mathbb{\tilde{E}}\left[\varLambda(j_{1},j_{2})M_{0}^{t}\left(2h\right)\int_{t_{j_{1}+1}}^{t}\alpha_{s}^{j_{1},k_{1},2}ds\int_{t_{j_{2}+1}}^{t}\alpha_{s}^{j_{2},k_{2},2}ds\right.\\
 & \quad\times\left.\int_{t_{j_{1}}}^{t_{j_{1}+1}}\beta_{s}^{j_{1}}dY_{s}^{i}\int_{t_{j_{2}}}^{t_{j_{2}+1}}\beta_{s}^{j_{2}}dY_{s}^{i}\right].
\end{align*}
The result follows by applying Lemma \ref{lem: Main Backward}, taking
into account Remark \ref{rem: Backward Ito integral}, to the terms
$A_{1},$$A_{2,1},A_{2,2},$$A_{2,3}$ and $A_{2,4}$. 
\end{proof}
\begin{lem}
\label{lem: dY_m=00003D00003D2_alpha=00003D00003D1}Assume that \textbf{H}$(2)$
holds and $\varphi\in C_{P}^{3}$. For $\alpha\in\mathcal{R}\left(\mathcal{M}_{1}(S_{0})\right)$
with $\left\vert \alpha\right\vert _{0}=1$ and $i\neq0$ we have
that 
\begin{align*}
\mathbb{\tilde{E}}\left[\mathbb{\tilde{E}}\left[\varphi(X_{t})e^{\xi_{t}}\int_{0}^{t}I_{\alpha}(L^{\alpha}h_{i}(X_{\cdot}))_{\tau(s),s}dY_{s}^{i}|\mathcal{Y}_{t}\right]^{2}\right] & \leq C\delta^{4}.
\end{align*}
\end{lem}
\begin{proof}
The proof of this lemma is analogous to the proof of Lemma \ref{lem: dY_m=00003D00003D1}.
Using Lemma \ref{lem: Convenient}, we can write 
\begin{align*}
 & \mathbb{\tilde{E}}\left[\varphi(X_{t})e^{\xi_{t}}\int_{0}^{t}I_{\alpha}(L^{\alpha}h_{i}(X_{\cdot}))_{\tau(s),s}dY_{s}^{i}|\mathcal{Y}_{t}\right]\\
 & =\sum_{r=1}^{d_{V}}\sum_{j=0}^{n-1}\mathbb{\tilde{E}}\left[\left(\int_{t_{j}}^{t_{j+1}}J_{s}^{r}dV_{s}^{r}\right)\left(\int_{t_{j}}^{t_{j+1}}\left(Y_{t_{j+1}}^{i}-Y_{t_{j}}^{i}\right)I_{\alpha-}(L^{\alpha}h_{i}(X_{\cdot}))_{\tau(s),s}dV_{s}^{\alpha_{|\alpha|}}\right)|\mathcal{Y}_{t}\right]\\
 & \triangleq\sum_{r=1}^{d_{V}}A\left(r\right).
\end{align*}
Therefore, by Lemma \ref{lem: CondExpect} $\left(\ref{lem: CE2}\right)$,
we have that 
\begin{align*}
 & A\left(r\right)\\
 & =\sum_{j=0}^{n-1}\mathbb{\tilde{E}}\left[\left(\int_{t_{j}}^{t_{j+1}}J_{s}^{r}dV_{s}^{r}\right)\left(\int_{t_{j}}^{t_{j+1}}\left(Y_{t_{j+1}}^{i}-Y_{s}^{i}\right)\left(\int_{t_{j}}^{s}L^{\alpha}h_{i}\left(X_{u}\right)dV_{u}^{\alpha_{1}}\right)dV_{s}^{\alpha_{2}}\right)|\mathcal{Y}_{t}\right]\\
 & =\mathbf{1}_{\left\{ \alpha_{1}=0,\alpha_{2}=r\right\} }\mathbb{\tilde{E}}\left[\left(\int_{0}^{t}\left(Y_{\eta(s)}^{i}-Y_{s}^{i}\right)J_{s}^{r}\left(\int_{\tau(s)}^{s}L^{\alpha}h_{i}\left(X_{u}\right)du\right)ds\right)|\mathcal{Y}_{t}\right]\\
 & \quad+\mathbf{1}_{\left\{ \alpha_{1}=r,\alpha_{2}=0\right\} }\mathbb{\tilde{E}}\left[\left(\int_{0}^{t}\left(Y_{\eta(s)}^{i}-Y_{s}^{i}\right)\left(\int_{\tau(s)}^{s}J_{u}^{r}L^{\alpha}h_{i}\left(X_{u}\right)du\right)ds\right)|\mathcal{Y}_{t}\right]\\
 & \triangleq\mathbf{1}_{\left\{ \alpha_{1}=0,\alpha_{2}=r\right\} }A_{1}\left(r\right)+\mathbf{1}_{\left\{ \alpha_{1}=r,\alpha_{2}=0\right\} }A_{2}\left(r\right).
\end{align*}
Next, the proof follows by similar reasonings as in Lemma \ref{lem: dY_m=00003D00003D1}. 
\end{proof}
\begin{lem}
\label{lem: dY_m=00003D00003D2_alpha=00003D00003D2}Assume that \textbf{H}$(2)$
holds and $\varphi\in C_{P}^{3}$. For $\alpha\in\mathcal{R}\left(\mathcal{M}_{1}(S_{0})\right)$
with $\left\vert \alpha\right\vert _{0}=0$ and $i\neq0$ we have
that 
\begin{align*}
\mathbb{\tilde{E}}\left[\mathbb{\tilde{E}}\left[\varphi(X_{t})e^{\xi_{t}}\int_{0}^{t}I_{\alpha}(L^{\alpha}h_{i}(X_{\cdot}))_{\tau(s),s}dY_{s}^{i}|\mathcal{Y}_{t}\right]^{2}\right] & \leq C\delta^{4}.
\end{align*}
\end{lem}
\begin{proof}
We divide the proof into several steps.

\textbf{Step 1}. Using Lemma \ref{lem: Convenient}, we can write
\begin{align*}
 & \mathbb{\tilde{E}}\left[\varphi(X_{t})e^{\xi_{t}}\int_{0}^{t}I_{\alpha}(L^{\alpha}h_{i}(X_{\cdot}))_{\tau(s),s}dY_{s}^{i}|\mathcal{Y}_{t}\right]\\
 & =\sum_{r_{1}=1}^{d_{V}}\sum_{j=0}^{n-1}\mathbb{\tilde{E}}\left[\left(\int_{t_{j}}^{t_{j+1}}\mathbb{\tilde{E}}\left[J_{s}^{r_{1}}|\mathcal{H}_{0}^{t}\right]dV_{s}^{r_{1}}\right)\right.\\
 & \quad\times\left.\left(\int_{t_{j}}^{t_{j+1}}\left(Y_{t_{j+1}}^{i}-Y_{t_{j}}^{i}\right)I_{\alpha-}(L^{\alpha}h_{i}(X_{\cdot}))_{\tau(s),s}dV_{s}^{\alpha_{|\alpha|}}\right)|\mathcal{Y}_{t}\right]\\
 & \quad+\sum_{r_{1},r_{2}=1}^{d_{V}}\sum_{j=0}^{n-1}\mathbb{\tilde{E}}\left[\left(\int_{t_{j}}^{t_{j+1}}\int_{0}^{s_{2}}J_{s_{1},s_{2}}^{r_{1},r_{2}}dV_{s_{1}}^{r_{1}}dV_{s_{2}}^{r_{2}}\right)\right.\\
 & \quad\times\left.\left(\int_{t_{j}}^{t_{j+1}}\left(Y_{t_{j+1}}^{i}-Y_{t_{j}}^{i}\right)I_{\alpha-}(L^{\alpha}h_{i}(X_{\cdot}))_{\tau(s),s}dV_{s}^{\alpha_{|\alpha|}}\right)|\mathcal{Y}_{t}\right]\\
 & \triangleq\sum_{r_{1}=1}^{d_{V}}\sum_{j=0}^{n-1}A\left(r_{1},j\right)+\sum_{r_{1,}r_{2}=1}^{d_{V}}\sum_{j=0}^{n-1}A\left(r_{1},r_{2},j\right).
\end{align*}
Therefore, by Lemma \ref{lem: CondExpect} $\left(\ref{lem: CE2}\right)$,
we have that 
\begin{align*}
 & A\left(r_{1},j\right)\\
 & =\mathbb{\tilde{E}}\left[\left(\int_{t_{j}}^{t_{j+1}}\mathbb{\tilde{E}}\left[J_{s}^{r_{1}}|\mathcal{H}_{0}^{t}\right]dV_{s}^{r_{1}}\right)\left(\int_{t_{j}}^{t_{j+1}}\left(Y_{t_{j+1}}^{i}-Y_{s}^{i}\right)\left(\int_{t_{j}}^{s}L^{\alpha}h_{i}\left(X_{u}\right)dV_{u}^{\alpha_{1}}\right)dV_{s}^{\alpha_{2}}\right)|\mathcal{Y}_{t}\right]\\
 & =\mathbf{1}_{\left\{ \alpha_{2}=r_{1}\right\} }\mathbb{\tilde{E}}\left[\left(\int_{t_{j}}^{t_{j+1}}\mathbb{\tilde{E}}\left[J_{s}^{r_{1}}|\mathcal{H}_{0}^{t}\right]\left(Y_{t_{j+1}}^{i}-Y_{s}^{i}\right)\left(\int_{t_{j}}^{s}L^{\alpha}h_{i}\left(X_{u}\right)dV_{u}^{\alpha_{1}}\right)ds\right)|\mathcal{Y}_{t}\right]\\
 & =\mathbf{1}_{\left\{ \alpha_{2}=r_{1}\right\} }\mathbb{\tilde{E}}\left[\left(\int_{t_{j}}^{t_{j+1}}\mathbb{\tilde{E}}\left[J_{s}^{r_{1}}|\mathcal{H}_{0}^{t}\right]\left(Y_{t_{j+1}}^{i}-Y_{s}^{i}\right)\mathbb{\tilde{E}}\left[\left(\int_{t_{j}}^{s}L^{\alpha}h_{i}\left(X_{u}\right)dV_{u}^{\alpha_{1}}\right)|\mathcal{H}_{t_{j}}^{t}\right]ds\right)|\mathcal{Y}_{t}\right]\\
 & =0.
\end{align*}
and, by Lemma \ref{lem: CondExpect} $\left(\ref{lem: CE3}\right)$
and Lemma \ref{lem: CondExpect} $\left(\ref{lem: CE2}\right)$, we
obtain 
\begin{align*}
 & A\left(r_{1},r_{2},j\right)\\
 & =\mathbb{\tilde{E}}\left[\left(\int_{t_{j}}^{t_{j+1}}\int_{0}^{s_{2}}J_{s_{1},s_{2}}^{r_{1},r_{2}}dV_{s_{1}}^{r_{1}}dV_{s_{2}}^{r_{2}}\right)\left(\int_{t_{j}}^{t_{j+1}}\left(Y_{t_{j+1}}^{i}-Y_{s}^{i}\right)\left(\int_{t_{j}}^{s}L^{\alpha}h_{i}\left(X_{u}\right)dV_{u}^{\alpha_{1}}\right)dV_{s}^{\alpha_{2}}\right)|\mathcal{Y}_{t}\right]\\
 & =\mathbf{1}_{\left\{ \alpha_{2}=r_{2}\right\} }\mathbb{\tilde{E}}\left[\left(\int_{t_{j}}^{t_{j+1}}\left(Y_{t_{j+1}}^{i}-Y_{s}^{i}\right)\left(\int_{0}^{s}J_{s_{1},s}^{r_{1},r_{2}}dV_{s_{1}}^{r_{1}}\right)\left(\int_{t_{j}}^{s}L^{\alpha}h_{i}\left(X_{u}\right)dV_{u}^{\alpha_{1}}\right)ds\right)|\mathcal{Y}_{t}\right]\\
 & =\mathbf{1}_{\left\{ \alpha_{2}=r_{2},\alpha_{1}=r_{1}\right\} }\mathbb{\tilde{E}}\left[\left(\int_{t_{j}}^{t_{j+1}}\left(Y_{t_{j+1}}^{i}-Y_{s}^{i}\right)\left(\int_{t_{j}}^{s}J_{u,s}^{r_{1},r_{2}}L^{\alpha}h_{i}\left(X_{u}\right)du\right)ds\right)|\mathcal{Y}_{t}\right].
\end{align*}
Hence, we can write 
\begin{align*}
 & \mathbb{\tilde{E}}\left[\varphi(X_{t})e^{\xi_{t}}\int_{0}^{t}I_{\alpha}(L^{\alpha}h_{i}(X_{\cdot}))_{\tau(s),s}dY_{s}^{i}|\mathcal{Y}_{t}\right]\\
 & =\sum_{r_{1,}r_{2}=1}^{d_{V}}\mathbf{1}_{\left\{ \alpha_{2}=r_{2},\alpha_{1}=r_{1}\right\} }\mathbb{\tilde{E}}\left[\left(\int_{0}^{t}\left(Y_{\eta(s)}^{i}-Y_{s}^{i}\right)\left(\int_{\tau(s)}^{s}J_{u,s}^{r_{1},r_{2}}L^{\alpha}h_{i}\left(X_{u}\right)du\right)ds\right)|\mathcal{Y}_{t}\right]\\
 & =\sum_{r_{1,}r_{2}=1}^{d_{V}}\mathbf{1}_{\left\{ \alpha_{2}=r_{2},\alpha_{1}=r_{1}\right\} }\left(B_{1}\left(r_{1},r_{2}\right)+B_{2}\left(r_{1},r_{2}\right)+B_{3}\left(r_{1},r_{2}\right)+B_{4}\left(r_{1},r_{2}\right)\right),
\end{align*}
where 
\begin{align*}
B_{1}\left(r_{1},r_{2}\right) & =\mathbb{\tilde{E}}\left[\left(\int_{0}^{t}\left(Y_{\eta(s)}^{i}-Y_{s}^{i}\right)\left(\int_{\tau(s)}^{s}\left(J_{u,s}^{r_{1},r_{2}}-J_{s,s}^{r_{1},r_{2}}\right)L^{\alpha}h_{i}\left(X_{u}\right)du\right)ds\right)|\mathcal{Y}_{t}\right],\\
B_{2}\left(r_{1},r_{2}\right) & =\mathbb{\tilde{E}}\left[\left(\int_{0}^{t}\left(Y_{\eta(s)}^{i}-Y_{s}^{i}\right)J_{s,s}^{r_{1},r_{2}}\left(\int_{\tau(s)}^{s}\left(L^{\alpha}h_{i}\left(X_{u}\right)-L^{\alpha}h_{i}\left(X_{\tau(s)}\right)\right)du\right)ds\right)|\mathcal{Y}_{t}\right]\\
B_{3}\left(r_{1},r_{2}\right) & =\mathbb{\tilde{E}}\left[\left(\int_{0}^{t}\left(Y_{\eta(s)}^{i}-Y_{s}^{i}\right)\left(J_{s,s}^{r_{1},r_{2}}-J_{\eta\left(s\right),\eta\left(s\right)}^{r_{1},r_{2}}\right)L^{\alpha}h_{i}\left(X_{\tau(s)}\right)\left(\int_{\tau(s)}^{s}du\right)ds\right)|\mathcal{Y}_{t}\right]\\
B_{4}\left(r_{1},r_{2}\right) & =\mathbb{\tilde{E}}\left[\left(\int_{0}^{t}\left(Y_{\eta(s)}^{i}-Y_{s}^{i}\right)J_{\eta\left(s\right),\eta\left(s\right)}^{r_{1},r_{2}}L^{\alpha}h_{i}\left(X_{\tau(s)}\right)\left(\int_{\tau(s)}^{s}du\right)ds\right)|\mathcal{Y}_{t}\right]
\end{align*}

\textbf{Step 2}. That the terms $B_{1}\left(r_{1},r_{2}\right),B_{2}\left(r_{1},r_{2}\right)$
and $B_{3}\left(r_{1},r_{2}\right)$ have the right order is deduced
analogously to the \textbf{Steps 2} and \textbf{3} in Lemma \ref{lem: dY_m=00003D00003D1}.

\textbf{Step 3}. Finally, we prove the result for $B_{4}\left(r_{1},r_{2}\right).$
We can write 
\begin{align*}
B_{4}\left(r_{1},r_{2}\right) & =\sum_{j=0}^{n-1}\mathbb{\tilde{E}}\left[\int_{t_{j}}^{t_{j+1}}J_{t_{j+1},t_{j+1}}^{r_{1},r_{2}}L^{\left(r_{1},r_{2}\right)}h_{i}\left(X_{t_{j}}\right)\left(\int_{t_{j}}^{s}du\right)\left(Y_{t_{j+1}}^{i}-Y_{s}^{i}\right)ds|\mathcal{Y}_{t}\right]\\
 & =\sum_{j=0}^{n-1}\mathbb{\tilde{E}}\left[J_{t_{j+1},t_{j+1}}^{r_{1},r_{2}}L^{\left(r_{1},r_{2}\right)}h_{i}\left(X_{t_{j}}\right)\int_{t_{j}}^{t_{j+1}}\left(\int_{t_{j}}^{s}du\right)\left(Y_{t_{j+1}}^{i}-Y_{s}^{i}\right)ds|\mathcal{Y}_{t}\right]\\
 & =\sum_{j=0}^{n-1}\mathbb{\tilde{E}}\left[J_{t_{j+1},t_{j+1}}^{r_{1},r_{2}}L^{\left(r_{1},r_{2}\right)}h_{i}(X_{t_{j}})\int_{t_{j}}^{t_{j+1}}\frac{\left(s-t_{j}\right)^{2}}{2}dY_{s}^{i}|\mathcal{Y}_{t}\right]\\
 & \triangleq\sum_{j=0}^{n-1}\mathbb{\tilde{E}}\left[J_{t_{j+1},t_{j+1}}^{r_{1},r_{2}}L^{\left(r_{1},r_{2}\right)}h_{i}(X_{t_{j}})\int_{t_{j}}^{t_{j+1}}\beta_{s}^{j}dY_{s}^{i}|\mathcal{Y}_{t}\right],
\end{align*}
where $\left|\beta_{s}^{j}\right|\leq\delta^{2}.$ Moreover, 
\[
J_{t_{j+1},t_{j+1}}^{r_{1},r_{2}}=\mathbb{\tilde{E}}\left[D_{t_{j+1},t_{j+1}}^{r_{1},r_{2}}\left\{ \varphi\left(X_{t}\right)e^{\xi_{t}}\right\} |\mathcal{H}_{t_{j}}^{t}\right],
\]
by the Clark-Ocone formula. Using the \textcolor{black}{definition
of the iterated Malliavin derivative and the product formula}, we
get 
\begin{align*}
D_{t_{j+1},t_{j+1}}^{r_{1},r_{2}}\left\{ \varphi\left(X_{t}\right)e^{\xi_{t}}\right\}  & =D_{t_{j+1}}^{r_{1}}\left(D_{t_{j+1}}^{r_{2}}\left\{ \varphi\left(X_{t}\right)e^{\xi_{t}}\right\} \right)\\
 & =D_{t_{j+1}}^{r_{1}}\left(e^{\xi_{t}}D_{t_{j+1}}^{r_{2}}\varphi\left(X_{t}\right)+\varphi\left(X_{t}\right)D_{t_{j+1}}^{r_{2}}e^{\xi_{t}}\right)\\
 & =D_{t_{j+1}}^{r_{1}}e^{\xi_{t}}D_{t_{j+1}}^{r_{2}}\varphi\left(X_{t}\right)+e^{\xi_{t}}D_{t_{j+1},t_{j+1}}^{r_{1},r_{2}}\varphi\left(X_{t}\right)\\
 & +D_{t_{j+1}}^{r_{1}}\varphi\left(X_{t}\right)D_{t_{j+1}}^{r_{2}}e^{\xi_{t}}+\varphi\left(X_{t}\right)D_{t_{j+1},t_{j+1}}^{r_{1},r_{2}}e^{\xi_{t}}
\end{align*}
Reasoning as in Step 4 of Lemma \ref{lem: dY_m=00003D00003D1}, we
get that 
\begin{align*}
\mathbb{\tilde{E}}\left[\left|B_{4}\left(r_{1},r_{2}\right)\right|^{2}\right] & \leq\mathbb{\tilde{E}}\left[\left|\sum_{j=0}^{n-1}\mathbb{\tilde{E}}\left[J_{t_{j+1},t_{j+1}}^{r_{1},r_{2}}L^{\left(r_{1},r_{2}\right)}h_{i}(X_{t_{j}})\int_{t_{j}}^{t_{j+1}}\beta_{s}^{j}dY_{s}^{i}|\mathcal{Y}_{t}\right]\right|^{2}\right]\\
 & \leq C\mathbb{\tilde{E}}\left[\left|\sum_{j=0}^{n-1}D_{t_{j+1}}^{r_{1}}e^{\xi_{t}}D_{t_{j+1}}^{r_{2}}\varphi\left(X_{t}\right)L^{\left(r_{1},r_{2}\right)}h_{i}(X_{t_{j}})\int_{t_{j}}^{t_{j+1}}\beta_{s}^{j}dY_{s}^{i}\right|^{2}\right]\\
 & \quad+C\mathbb{\tilde{E}}\left[\left|\sum_{j=0}^{n-1}e^{\xi_{t}}D_{t_{j+1},t_{j+1}}^{r_{1},r_{2}}\varphi\left(X_{t}\right)L^{\left(r_{1},r_{2}\right)}h_{i}(X_{t_{j}})\int_{t_{j}}^{t_{j+1}}\beta_{s}^{j}dY_{s}^{i}\right|^{2}\right]\\
 & \quad+C\mathbb{\tilde{E}}\left[\left|\sum_{j=0}^{n-1}D_{t_{j+1}}^{r_{1}}\varphi\left(X_{t}\right)D_{t_{j+1}}^{r_{2}}e^{\xi_{t}}L^{\left(r_{1},r_{2}\right)}h_{i}(X_{t_{j}})\int_{t_{j}}^{t_{j+1}}\beta_{s}^{j}dY_{s}^{i}\right|^{2}\right]\\
 & \quad+C\mathbb{\tilde{E}}\left[\left|\sum_{j=0}^{n-1}\varphi\left(X_{t}\right)D_{t_{j+1},t_{j+1}}^{r_{1},r_{2}}e^{\xi_{t}}L^{\left(r_{1},r_{2}\right)}h_{i}(X_{t_{j}})\int_{t_{j}}^{t_{j+1}}\beta_{s}^{j}dY_{s}^{i}\right|^{2}\right]\\
 & \triangleq C\left\{ F_{1}\left(r_{1},r_{2}\right)+F_{2}\left(r_{1},r_{2}\right)+F_{3}\left(r_{1},r_{2}\right)+F_{4}\left(r_{1},r_{2}\right)\right\} .
\end{align*}
The term $F_{2}\left(r_{1},r_{2}\right)$ is analogous to the term
$A$$_{1}\left(r\right)$ in Lemma \ref{lem: dY_m=00003D00003D1}
and the terms $F_{2}\left(r_{1},r_{2}\right)$ and $F_{3}\left(r_{1},r_{2}\right)$
are analogous to the term $A_{2}\left(r\right)$ in Lemma \ref{lem: dY_m=00003D00003D1}.
For the term $F_{4}\left(r_{1},r_{2}\right)$ we have that 
\begin{align*}
D_{t_{j+1},t_{j+1}}^{r_{1},r_{2}}e^{\xi_{t}} & =D_{t_{j+1}}^{r_{1}}\left\{ e^{\xi_{t}}\left\{ \sum_{k=1}^{d_{Y}}\int_{t_{j+1}}^{t}D_{t_{j+1}}^{r_{2}}h_{k}(X_{s})dY_{s}^{k}-\frac{1}{2}\sum_{k=1}^{d_{Y}}\int_{t_{j+1}}^{t}D_{t_{j+1}}^{r_{2}}\left[h_{k}(X_{s})^{2}\right]ds\right\} \right\} \\
 & =e^{\xi_{t}}\left\{ \sum_{k=1}^{d_{Y}}\int_{t_{j+1}}^{t}D_{t_{j+1}}^{r_{1}}h_{k}(X_{s})dY_{s}^{k}-\frac{1}{2}\sum_{k=1}^{d_{Y}}\int_{t_{j+1}}^{t}D_{t_{j+1}}^{r_{1}}\left[h_{k}(X_{s})^{2}\right]ds\right\} \\
 & \qquad\times\left\{ \sum_{k=1}^{d_{Y}}\int_{t_{j+1}}^{t}D_{t_{j+1}}^{r_{2}}h_{k}(X_{s})dY_{s}^{k}-\frac{1}{2}\sum_{k=1}^{d_{Y}}\int_{t_{j+1}}^{t}D_{t_{j+1}}^{r_{2}}\left[h_{k}(X_{s})^{2}\right]ds\right\} \\
 & +e^{\xi_{t}}\left\{ \sum_{k=1}^{d_{Y}}\int_{t_{j+1}}^{t}D_{t_{j+1},t_{j+1}}^{r_{1},r_{2}}h_{k}(X_{s})dY_{s}^{k}-\frac{1}{2}\sum_{k=1}^{d_{Y}}\int_{t_{j+1}}^{t}D_{t_{j+1},t_{j+1}}^{r_{1},r_{2}}\left[h_{k}(X_{s})^{2}\right]ds\right\} 
\end{align*}
All the terms obtained in the previous expression can be dealt analogously
to the terms in Lemma \ref{lem: dY_m=00003D00003D1} except the terms
\[
G\left(j,r_{1},r_{2}\right)\triangleq e^{\xi_{t}}\sum_{k_{1},k_{2}=1}^{d_{Y}}\left(\int_{t_{j+1}}^{t}D_{t_{j+1}}^{r_{1}}h_{k_{1}}(X_{s})dY_{s}^{k_{1}}\right)\left(\int_{t_{j+1}}^{t}D_{t_{j+1}}^{r_{2}}h_{k_{2}}(X_{s})dY_{s}^{k_{2}}\right).
\]
Let 
\[
H\triangleq\mathbb{\tilde{E}}\left[\left|\sum_{j=0}^{n-1}\varphi\left(X_{t}\right)G\left(j,r_{1},r_{2}\right)L^{\left(r_{1},r_{2}\right)}h_{i}(X_{t_{j}})\int_{t_{j}}^{t_{j+1}}\beta_{s}^{j}dY_{s}^{i}\right|^{2}\right].
\]
Defining 
\[
\varLambda(j_{1},j_{2})\triangleq\varphi\left(X_{t}\right)^{2}L^{\left(r_{1},r_{2}\right)}h_{i}(X_{t_{j_{1}}})L^{\left(r_{1},r_{2}\right)}h_{i}(X_{t_{j_{2}}})\exp\left(\sum_{i=1}^{d_{Y}}\int_{0}^{t}h_{i}^{2}\left(X_{u}\right)du\right),
\]
we can write 
\begin{align*}
H & =\sum_{k_{1},...,k_{4}=1}^{d_{Y}}\sum_{j_{1},j_{2}=0}^{n-1}\mathbb{\tilde{E}}\left[\varLambda(j_{1},j_{2})M_{0}^{t}\left(2h\right)\right.\\
 & \qquad\times\left(\int_{t_{j_{1}+1}}^{t}D_{t_{j_{1}+1}}^{r_{1}}h_{k_{1}}(X_{s})dY_{s}^{k_{1}}\right)\left(\int_{t_{j_{1}+1}}^{t}D_{t_{j_{1}+1}}^{r_{2}}h_{k_{2}}(X_{s})dY_{s}^{k_{2}}\right)\\
 & \qquad\times\left(\int_{t_{j_{2}+1}}^{t}D_{t_{j_{2}+1}}^{r_{1}}h_{k_{3}}(X_{s})dY_{s}^{k_{3}}\right)\left(\int_{t_{j_{2}+1}}^{t}D_{t_{j_{2}+1}}^{r_{2}}h_{k_{4}}(X_{s})dY_{s}^{k_{4}}\right)\\
 & \qquad\times\left.\left(\int_{t_{j_{1}}}^{t_{j_{1}+1}}\beta_{s}^{j_{1}}dY_{s}^{i}\right)\left(\int_{t_{j_{2}}}^{t_{j_{2}+1}}\beta_{s}^{j_{2}}dY_{s}^{i}\right)\right],
\end{align*}
and the result follows from Lemma \ref{lem: Main Backward} and Remark
\ref{rem: Backward Ito integral}. 
\end{proof}
\begin{rem}
Following Remarks \ref{rem: HighRegularityKernel} and \ref{rem: BackwardEstimate},
the results in Lemmas \ref{lem: dY_m=00003D00003D2_alpha=00003D00003D1}
and \ref{lem: dY_m=00003D00003D2_alpha=00003D00003D2} can be extended
analogously to $m>2$ and $\alpha\in\mathcal{R}\left(\mathcal{M}_{m-1}(S_{0})\right)$
with $\left\vert \alpha\right\vert _{0}\in\{0,...,m-1\}$ without
any additional difficulties. 
\end{rem}


\begin{thebibliography}{10}
\bibitem{App09}Applebaum, D. (2009). \emph{Lévy Processes and Stochastic
Calculus}. \emph{Second Edition}. Cambridge Studies in Advanced Mathematics,
vol. \textbf{116}. Cambridge University Press, Cambridge.

\bibitem{BaCr08}Bain, A. and Crisan, D. (2008). \emph{Fundamentals
of Stochastic Filtering}. Stochastic Modelling and Applied Probability,
vol \textbf{60}. Springer Verlag, New York.

\bibitem{Ben92}Bensoussan, A. (1992).\emph{ Stochastic Control of
Partially Observable Systems}. Cambridge University Press, Cambridge.

\bibitem{Cris11}Crisan, D. (2011). Discretizing the Continuous Time
Filtering Problem. Order of Convergence. In \textit{The Oxford Handbook
of Nonlinear Filtering}. Oxford University Press, Oxford.

\bibitem{CrRo11}Crisan, D. and Rozovsky, B. (2011). \textit{The Oxford
handbook of nonlinear filtering}. Oxford University Press, Oxford.

\bibitem{CrOr2013}Crisan, D. and Ortiz-Latorre, S. (2013). A Kusuoka-Lyons-Victoir
particle filter. \textit{Proc. R. Soc. Lond. Ser. A Math. Phys. Eng.
Sci.} \textbf{469}, no. 2156.

\bibitem{Delm04}Del Moral, P. (2004). \textit{Feynman-Kac formulae.
Genealogical and interacting particle systems with applications}.
Probab. Appl. . Springer-Verlag, New York.

\bibitem{DFG01}Doucet, A., De Freitas, N. and Gordon, N. (2001).
\textit{Sequential Monte Carlo methods in practice}. Springer-Verlag,
New York.

\bibitem{KS91}Karatzas, I. and Shreve, S.E. (1991). \textit{Brownian
Motion and Stochastic Calculus}. Graduate Texts in Mathematics, vol.
\textbf{113}, Second Edition. Springer-Verlag, New York.

\bibitem{KlPl92}Kloeden, P. and Platen, E. (1992). \textit{Numerical
solution of stochastic differential equations}. Stochastic Modelling
and Applied Probability, vol. \textbf{23}, Springer Verlag, New York.

\bibitem{KuStr84}Kusuoka, S. and Stroock, D. (1984). The partial
Malliavin Calculus and its application to non-linear filtering. \emph{Stochastics},
\textbf{12}, 83-142.

\bibitem{Nu06}Nualart, D. (2006). \textit{The Malliavin Calculus
and Related Topics}. Springer Verlag, New York.

\bibitem{NuZa89}Nualart, D. and Zakai, M. (1989). The partial Malliavin
calculus. Séminaire de Probabilités XXIII. Lecture Notes in Math.
\textbf{1372}, 362-381, Springer Verlag, Berlin.

\bibitem{ParPro87}Pardoux, E. and Protter, P. (1987). A two-sided
stochastic integral and its calculus. \emph{Probability Theory and
Related Fields}, \textbf{76}, 15-49.

\bibitem{Pi84}Picard, J. (1984) Approximation of nonlinear filtering
problems and order of convergence. Filtering and control of random
processes (Paris, 1983), \textit{Lecture Notes in Control and Information
Sciences}, \textbf{61}, 219\textendash 236, Springer Verlag, Berlin.

\bibitem{Stro87}Stroock, D. (1987) Homogeneous chaos revisited. Séminaire
de Probabilités XXI, Lecture Notes in Math., vol \textbf{1247}, 1-7,
Springer Verlag, Berlin.

\bibitem{Ta14} Tanaka, H. (2014) A new proof for the convergence
of Picard's filter using partial Malliavin calculus. arXiv:1311.6090v3 
\end{thebibliography}
\end{document}